\documentclass[10pt]{iopart}
\usepackage{graphicx}
\graphicspath{{}}
\expandafter\let\csname equation*\endcsname\relax
\expandafter\let\csname endequation*\endcsname\relax
\usepackage[tbtags]{amsmath}%amsthm is included in amsmath package centertags
\usepackage{amsfonts,amssymb}
\usepackage{mathabx}
\usepackage{mathtools}
\mathtoolsset{showonlyrefs=fasle,showmanualtags}
\usepackage{amsthm} % use amsthm envirenments
\theoremstyle{plain}% default
\newtheorem{theorem}{Theorem}
\newtheorem{lem}[theorem]{Lemma}

\newtheorem{as}{Assumption}
\theoremstyle{definition}

\theoremstyle{remark}
\newtheorem*{remark}{Remark}
\usepackage{amssymb}
\usepackage[english]{babel}
\usepackage{color}
\usepackage{latexsym,bm}
\usepackage{algorithm}
\usepackage{algorithmic}
\usepackage{graphicx}
\usepackage{float}
\usepackage{multirow}
\usepackage{extpfeil}
\usepackage{grffile}
\usepackage{subfig}
\usepackage{makecell}

\usepackage{comment}
\newcommand{\norm}[1]{\left\lVert#1\right\rVert}
\DeclareMathOperator*{\argmin}{arg\,min}

\DeclareMathOperator*{\X}{X}

\newcommand{\rd}{\,\mathrm{d}}

\begin{document}

\title[A Hybrid Reconstruction Approach for Absorption Coefficient by FPAT]{A Hybrid Reconstruction Approach for Absorption Coefficient by Fluorescence Photoacoustic Tomography}

\author{Chao Wang$^1$ \& Tie Zhou$^2$}

\address{$^1$$^2$ School of Mathematical Sciences, Peking University, China}
\ead{chaowyww@pku.edu.cn}
\ead{tzhou@math.pku.edu.cn}
\begin{abstract}
In this paper, we propose a hybrid method to reconstruct the absorption coefficient by fluorescence photoacoustic tomography (FPAT), which combines a squeeze iterative method (SIM) and a nonlinear optimization method. The SIM is to use two monotonic sequences to squeeze the exact coefficient, and it quickly locates near the exact coefficient. The nonlinear optimization method is utilized to attain a higher accuracy. The hybrid method inherits the advantages of each method with higher accuracy and faster convergence. The hybrid reconstruction method is also suitable for multi-measurement. Numerical experiments show that the hybrid method converges faster than the optimization method in multi-measurement case, and that the accuracy is also higher in one-measurement case.
\end{abstract}
%\tableofcontents
\section{Introduction}
\label{sec:1}

Fluorescence photoacoustic tomography (FPAT) can achieve targeted imaging for some specific biological tissues marked by fluorescent dyes~\cite{Ren2013}, which combines photoacoustic tomography (PAT) and fluorescence molecular tomography (FMT) with the merits of high spatial resolution and high optical contrast respectively. In practical application, the fluorescent dye is injected into the biological tissue and the tissue is illuminated by a series of short pulsed laser with given wavelength, called excitation light, then it propagates through tissue and the energy of the light is absorbed by the tissue and dyestuffs. The fluorophores  in the tissue are  also illuminted and excited to emit light at a different wavelength, called emission light. The emission light is also absorbed. All the energy absorbed by the tissue and fluorescent dyestuffs comes partly from the excitation light and partly from the light emission light. As the energy is absorbed and released, tissue expands and contracts, which gives rise to an ultrasound wave. Then the wave spreads outward and is recorded by the ultrasound detectors outside. The physical process is also illustrated in figure \ref{fig:process}. Because the light travels much faster than the ultrasound wave, we think that the two parts of energy generate initial pressure almost at the same time.

\begin{figure}[H]
    \centering
\includegraphics[width=0.5\textwidth]{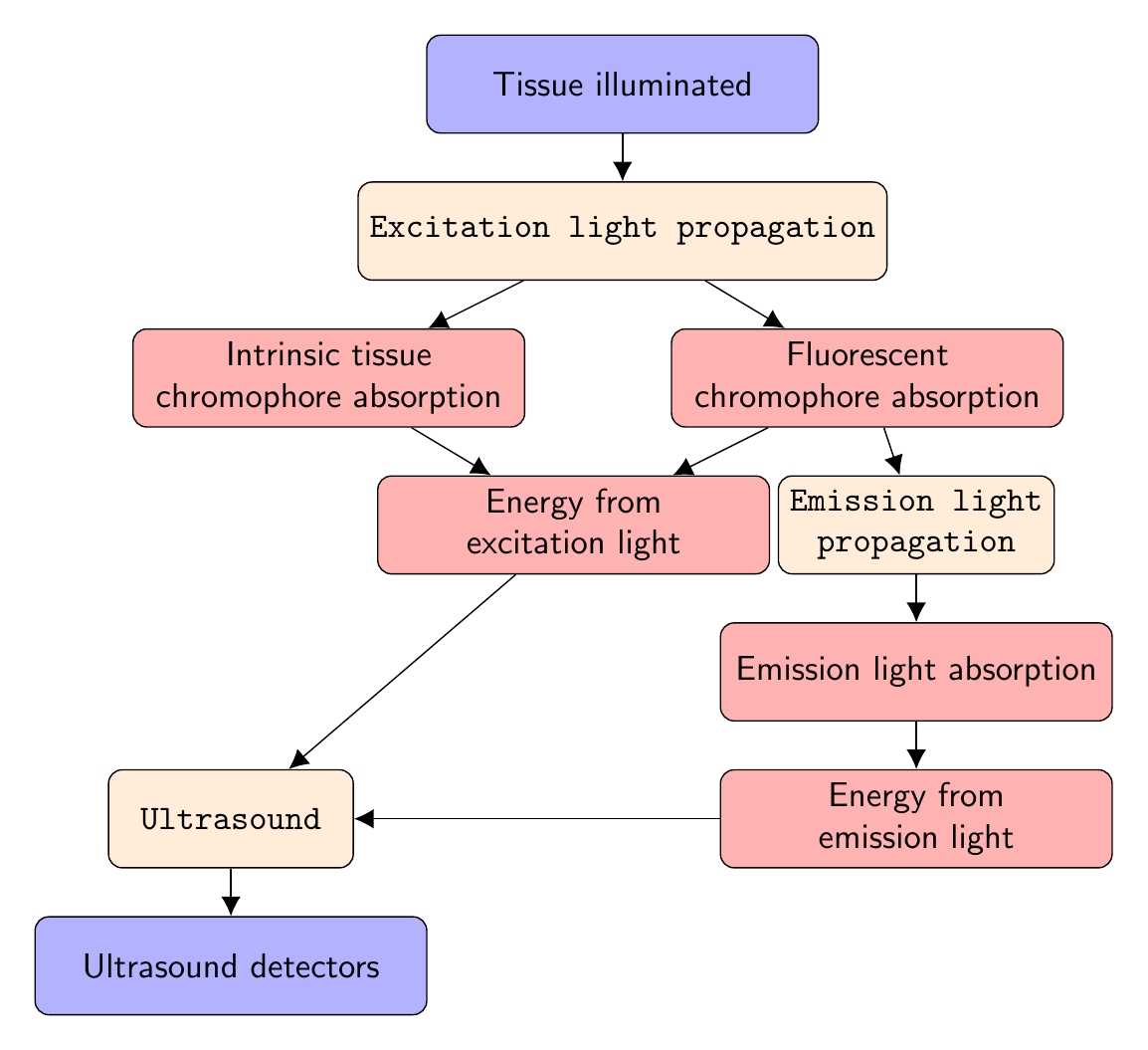}
 \caption{Diagram of FPAT.\label{fig:process}}
\end{figure}

The optical coefficient of biological tissue plays a key role in medical diagnosis. Since exogenous contrast agents, such as fluorescent dyes, can improve contrast, sensitive and specificity, fluorescence-based tomography is applied to medical imaging~\cite{Ammari2012,Wu2005,Kumar2008,Godavarty2005,Alvarez2009,Soloviev2007,Godavarty2005}. Researchers have found that fluorescent-based tomography is more targeted and specific for the visualization of cancerous areas because fluorescence enhances glycolysis of cancer cells \cite{Wu2005}. Despite this, existing fluorescence-based imaging cannot image the optical coefficients of deep tissue due to the strong scattering of infrared light by biological tissue~\cite{Wang2012}. Because of the low scattering of the ultrasound in biological tissues, the PAT can recover the internal optical information from the ultrasound, and it overcomes the diffraction limit of optical imaging and can achieve higher spatial resolution. However, PAT has lower optical contrast than fluorescence optical tomography, which is verified in experiment \cite{Wang2012}. Therefore, the combination of light and ultrasound can bring optical information from deep tissue to the outside. The fluorescence photoacoustic tomography as a hybrid imaging technique is expected to obtain higher spatial resolution and higher contrast simultaneously. Also acoustic radiation is used to improve the resolution of FMT~\cite{Li2018}, which also makes use of the characteristics of low-scattering sound.

There are two stages in the image reconstruction of FPAT: one is to determine the spatial distribution of initial pressure from ultrasound information, called the regular PAT; the other one is to recover optical coefficients from initial pressure, called quantitative step. A lot of researchers have studied the theories and algorithms of the regular PAT~\cite{Haltmeier2018,Lucka2018,Schwab2018,Hauptmann2018,Xu2002,Xu2006,Agranovsky2007,Lv2014,Kuchment}. Provided energy distribution, quantitative step is essential for exploring the optical properties inside tissue and improving visualization capability~\cite{Mamonov2012}. Quantitative photoacoustic tomography (QPAT) recovers intrinsic optical coefficients, including absorption, scattering coefficient and conversion efficiency~\cite{Rabanser2018,Haltmeier,Cox2009}. Quantitative FPAT recovers fluorescence absorption coefficient $\mu_{a,xf}$, quantum efficiency $\eta$ and conversion efficiency $\gamma$. There are some experiments combined with PAT and fluorescence optical imaging~\cite{Langer2018,Razansky2007,Razansky2009,Wang2010}. Both the diffusion approximation (DA)~\cite{Bal,Ding2015} and radiative transfer equation (RTE)~\cite{Rabanser2018,Wang2017,Mamonov2012,Tarvainen2012} are applied to model the propagation of light in tissue. DA model for FPAT is firstly derived in~\cite{Ren2013} and some uniqueness and stability results are established. RTE model for FPAT is firstly described in~\cite{Ren2015}, and the uniqueness of reconstruction is established for fluorescent absorption coefficient $\mu_{a,xf}$ and quantum efficiency $\eta$ under some assumptions. It has been proved in~\cite{Ren2015} that given one of $\mu_{a,xf}$ and $\eta$, initial pressure can uniquely determine the other one. In medical diagnosis, fluorescence absorption coefficient reflects the density of fluorescent markers, which contributes to determine the location of lesion. In this paper, provided quantum efficiency $\eta$ and conversion efficiency $\gamma$, we focus on recovering the fluorescence absorption coefficient $\mu_{a,xf}$ based on RTE, which is considered more accurately to describe the light propagating through tissue than DA~\cite{Yao2010}.  

In this paper, our goal is to design an efficient numerical method with fast convergence and high accuracy to recover the fluorescence absorption coefficient $\mu_{a,xf}$ from initial pressure based on RTE model. Firstly, a squeeze iterative method (SIM) is expected to approximate the exact value $\mu_{a,xf}^*$ from two sides and it quickly loactes near the exact coefficient. Then, the nonlinear optimization method, as a state-of-art method, is used to attain a higher accuracy stably. Therefore, combined with two methods, hybrid method is proposed to achieve high accuracy and fast convergence simultaneously. Simulations show that the hybrid method are comparable and even better than optimization method in one-measurement case.

The rest of the paper is organized as follows. We introduce the mathematical model of FPAT based on~\cite{Ren2015} in section~\ref{sec:2}. Then we propose the hybrid algorithm to recover fluorescence absorption coefficient $\mu_{a,xf}$ and discuss the properties of SIM in section~\ref{sec:3}. Numerical experiments based on synthetic data are presented in section~\ref{sec:4} by comparing the hybrid algorithm and the nonlinear optimization algorithm. Conclusions are drawn in section~\ref{sec:5}.

\section{Mathematical model}
\label{sec:2}
In this section, we present the mathematical model of FPAT refered to \cite{Ren2015} and its several properties. Through this paper,we assume $\Omega\in \mathcal{R}^d (d=2,3)$ is bounded convex region with Lipschitz boundary $\partial \Omega$, and $\mathcal{S}^{d-1}$ is the angular space in $\mathcal{R}^d$. We denote the phase space by $X=\Omega\times \mathcal{S}^{d-1}$. Outflow and inflow boundaries are denoted by $\Gamma_+$ and $\Gamma_-$ respectively, which represent $\Gamma_{\pm}:=\{(x,\theta)\in \partial \Omega \times \mathcal{S}^{d-1}| \pm \theta\cdot \nu(x)>0\}$, where $\nu$ is the outward unit normal vector. We define the scattering operator $\bm{K}$, average operator $\bm{A}$ and collecting operator $\tilde{\bm{A}}$ respectively by
\begin{equation*}
  \begin{aligned}
 &(\bm{K}\phi)(x):=\oint_{\mathcal{S}^{d-1}}f(\theta,\theta')\phi(x,\theta')\rd \theta',\\
&(\bm{A}\phi)(x):=\oint_{\mathcal{S}^{d-1}}\phi(x,\theta)\rd \theta,\\
   & (\tilde{\bm{A}}\phi)(x):=\frac{1}{S_d}\oint_{\mathcal{S}^{d-1}}\phi(x,\theta)\rd \theta,
  \end{aligned}
\end{equation*}
where $S_d$ is the area of the unit sphere in $\mathcal{R}^d$.

The light traveling inside can be described by the stationary radiative transfer equations (RTE)
\begin{equation}
  \label{eq:1}
\left\{
  \begin{aligned}
   & (\theta\cdot \nabla+\mu_{a,x}(x)+\mu_{s,x}(x)-\mu_{s,x}(x)\bm{K})\phi_x(x,\theta)=0,&& \text{in }X\\
&(\theta\cdot\nabla+\mu_{a,m}(x)+\mu_{s,m}(x)-\mu_{s,m}(x)\bm{K})\phi_m(x,\theta)=\eta(x)\mu_{a,xf}(x)\tilde{\bm{A}}\phi_x(x),&&\text{in }X\\
&\phi_x(x,\theta)=q_b(x,\theta), \quad \phi_m(x,\theta)=0, &&\text{on }\Gamma_-,
  \end{aligned}
\right. 
\end{equation}
where the subscripts $x$ and $m$ represent the quantities in the state of excitation and emission respectively, and we list them in table \ref{tab:def}.
\begin{table}
  \centering
\begin{tabular}{@{}cc|cc@{}}
\hline
    Symbol&Quantity&Symbol&Quantity\\
\hline
$x$&Spatial point&$\theta$&Direction\\
\hline
$\mu_{a,xi}(x)$&\makecell[cc]{Intrinsic chromophores \\absorption coef.}&$\mu_{a,xf}(x)$&\makecell[cc]{Fluorophores \\absorption coef.}\\
\hline
$\mu_{a,x}(x)$& $\mu_{a,xi}+\mu_{a,xf}$ &$\mu_{a,m}(x)$&\makecell[cc]{Absorption coef. \\at emission state}\\
\hline
$\mu_{s,x}(x)$&\makecell[cc]{Scattering coef. \\at excitation state}&$\mu_{s,m}(x)$&\makecell{Scattering coef. \\at emission state}\\
\hline
  $\phi_x(x,\theta)$&\makecell[cc]{Density of energy \\of excited light}&$\phi_m(x,\theta)$&\makecell[cc]{Density of energy \\of emission light}\\
  \hline
  $\eta(x)$&\makecell[cc]{Quantum efficiency\\ of the fluorophores}& & \\
  \hline
\end{tabular}
\caption{\label{tab:def}Symbols}
\end{table}
It is remarkable that $\phi_x(x,\theta)$ is caused by internal external light source $q_b$ and $\phi_m$ is caused by internal fluorescent markers source, which is formed by the excited photon energy absorption, that is $\eta\mu_{a,xf}\tilde{\bm{A}}\phi_x$. Scattering operator is characterized by scattering kernel $f(\theta,\theta')$, which represents the probability of light traveling from direction $\theta'$ to the direction $\theta$, and we usually use well-known Henyey-Greenstein (H-G) scattering function of the form
\begin{equation}
  \label{eq:3}
  f(\theta,\theta')=
\left\{
  \begin{aligned}
    &\frac{1-g^2}{2\pi(1+g^2-2g\theta\cdot\theta')},&&\quad n=2,\\
&\frac{1-g^2}{2\pi(1+g^2-2g\theta\cdot\theta')^{3/2}},&&\quad n=3,
  \end{aligned}
\right.
\end{equation}
which is symmetric and satisfies
\begin{equation}
  \label{eq:9}
  \oint_{\mathcal{S}^{d-1}}f(\theta,\theta')\rd \theta'=1.
\end{equation}

After light traveling and energy conversion, the initial pressure generates with the form
\begin{equation}
  \label{eq:4}
  p_0(x):=\gamma(x)h(x),
\end{equation}
where $\gamma(x)$ is the spatially varying conversion efficiency from absorbed photon energy to initial pressure and $h(x)$ is the absorbed energy,
\begin{equation}
  \label{eq:heat}
 h(x)=(\mu_{a,xi}+(1-\eta)\mu_{a,xf})(\bm{A}\phi_x)(x)+\mu_{a,m}(\bm{A}\phi_m)(x),
\end{equation}
where $\phi_x(x,\theta)$ and $\phi_m(x,\theta)$ are the solutions of RTE system \eqref{eq:1} depending on boundary condition $q_b(x,\theta)$ and optical coefficients $\mu_{a,x}$, $\mu_{a,m}$, $\mu_{s,x}$ and $\mu_{s,m}$. The absorbed energy at the excitation state is $(\mu_{a,xi}+\mu_{a,xf})(\bm{A}\phi_x)(x)$, a portion of which is $\eta\mu_{a,xf}(\bm{A}\phi_x)(x)$ to excite fluorescent light. Then the remaining energy at the excitation state $((\mu_{a,xi}+(1-\eta)\mu_{a,xf})(\bm{A}\phi_x)(x)$, and the absorbed energy at emission state, $\mu_{a,m}(\bm{A}\phi_m)(x)$, together generates the initial pressure with the conversion efficiency $\gamma$. And then, the tissue expands outward due to the absorbed energy, which brings out ultrasound traveling through tissue formulated by wave function
\begin{equation}
  \label{eq:5}
\left\{
\begin{aligned}
  &\frac{1}{c^2(x)}\frac{\partial^2}{\partial t^2}p(x,t)-\Delta p(x,t)=0, \quad &&(x,t)\in \mathcal{R}^n\times(0,T], \\
&p(x,0)=p_0(x),&& x\in \Omega,\\
&\frac{\partial p}{\partial t}(x,0)=0,&& x\in \Omega,
\end{aligned}
\right.
\end{equation}
where $c(x)$ is the speed of ultrasound inside tissue; $p(x,t)$ is the pressure of sound in the spatial point $x\in \mathcal{R}^d$ and the temporal point $t$; $p_0(x)$ is the initial pressure. The measurement $p(x,t)|_{\partial \Omega\times(0,T]}$ is obtained on the surface $\partial \Omega$ by ultrasound detectors. 
 
The FPAT is mainly concerned with the reconstruction of $\mu_{a,xf}$, $\eta$ and $\gamma$, assuming that the related optical coefficients $\mu_{a,xi}$, $\mu_{s,x}$, $\mu_{a,m}$ and $\mu_{s,m}$ can be acquired by other imaging technology such as DOT and QPAT. In the imaging experiments, firstly we need to reconstruct initial pressure $p_0(x)\ (x\in\Omega)$ from ultrasound data $p(x,t)|_{\partial \Omega\times(0,T]}$, where $T$ is large enough to ensure that information inside the tissue has been already transmitted. Secondly, optical coefficients $\mu_{a,xf}$, $\eta$ and $\gamma$ are recovered from $p_0(x)\ (x\in\Omega)$. In this paper, assuming $\gamma(x)=1$, we focus on the inverse problem of the reconstruction of $\mu_{a,xf}(x)$ from $h(x;\mu_{a,xf},q_b)$ given $\eta$.

In this paper, we use superscript to indicate the number of iteration. The fluorescence absorption coefficient is denoted by $\mu_{a,xf}^i$ in $i$th iteration. We assume $S$ measurements and corresponding boundary conditions are $q_{b,s}\ (s=0,1,\dots,S-1)$. Then we denote the solutions of RTE system \eqref{eq:1} and the data in $i$th iteration and $s$th measurement by $\phi_{x,s}^i(x,\theta;\mu_{a,xf}^i,q_{b,s})$, $\phi_{m,s}^i(x,\theta;\mu_{a,xf}^i,q_{b,s})$ and $h^i_s(x;\mu_{a,xf}^i,q_{b,s})$. Using symbol '$*$' to replace the '$i$', true quantities are denoted by $\mu_{a,xf}^*,\ \phi_{x,s}^*,\ \phi_{m,s}^*$, and $h_s^*$ for $s$th measurement. 

In order to discuss the properties of RTE system \eqref{eq:1}, we denote the space of all measurable functions defined in $X$ by $\mathcal{L}^p(X)\ (1\leq p\leq\infty)$, and its norm is
\begin{equation*}
  \|\phi(x,\theta)\|_p:=
  \begin{cases}
    \left(\int_{X}|\phi(x,\theta)|^p\rd \theta \rd x\right)^{\frac{1}{p}}&\text{for }p <\infty,\\
\text{ess}\ \sup_{(x,\theta)\in X)}|\phi(x,\theta)|&\text{for }p=\infty.
  \end{cases}
\end{equation*}
Correspondingly, the space of all measurable functions defined on $\Gamma_{\pm}$ is denoted by $\|\phi(x,\theta)\|_{\mathcal{L}^p(\Gamma_{\pm},|\theta\cdot \nu|)}$ and its norm is
\begin{equation*}
  \|\phi(x,\theta)\|_{\mathcal{L}^p(\Gamma_{\pm},|\theta\cdot \nu|)}:=
  \begin{cases}
    \left(\int_{\Gamma_{\pm}}|\theta\cdot \nu||\phi(x,\theta)|^p\rd \theta \rd x\right)^{\frac{1}{p}}&\text{for }p<\infty,\\
\text{ess}\ \sup_{(x,\theta)\in \Gamma_{\pm}}|\theta\cdot\nu||\phi(x,\theta)|&\text{for }p=\infty.
  \end{cases}
\end{equation*}

First of all, we make some assumptions on optical coefficients.
\begin{as}
\label{as:1}
   Assume optical coefficients and boundary source satisfy
  \begin{enumerate}
  \item  $\mu_{a,xf},\mu_{a,xi},\mu_{s,x},\mu_{a,m},\mu_{s,m}\in \mathcal{D}(\Omega):=\{u\in\Omega:0<c_1\leq u\leq c_2<\infty\}$ for some $c_1,c_2>0$;
  \item  $0< q_{b,s}(x,\theta)\in\mathcal{L}^p(X)$ for $s=0,1,\dots,S-1$. 
  \end{enumerate}

\end{as}

We define operators $\bm{T}_{1,s}$, $\bm{T}_{2,s}$ and $\bm{H}_{s}\ (s=0,1,\dots,S-1)$ by
\begin{equation}
  \label{eq:def}
  \begin{aligned}
    &\bm{T}_{1,s}:\mathcal{D}(\Omega)\mapsto  \mathcal{L}^p(X), \bm{T}_{1,s}(\mu_{a,xf})=\phi_{x,s}(x,\theta;\mu_{a,xf})\\
&\bm{T}_{2,s}:\mathcal{D}(\Omega)\mapsto \mathcal{L}^p(X), \bm{T}_{2,s}(\mu_{a,xf})=\phi_{m,s}(x,\theta;\mu_{a,xf})\\
&\bm{H}_s:\mathcal{D}(\Omega)\mapsto \mathcal{L}^p(\Omega), \bm{H}_s(\mu_{a,xf})=h_s(x;\mu_{a,xf})
  \end{aligned}
\end{equation}
In fact, $\bm{T}_{1,s}$ and $\bm{T}_{2,s}$ are the process of solving the first and the second RTE in \eqref{eq:1} given boundary condition $q_{b,s}\ (s=0,1,\dots,S-1)$. Under the physically reasonable assumption \ref{as:1}, we study the uniqueness and stability of the solutions of RTE system \eqref{eq:1}.
\begin{lem}
\label{lem:1}
  If the optical coefficients $\mu_a$, $\mu_s$ and boundary source $q_b$ of stationary RTE
  \begin{equation}
\label{eq:2}
\left\{
  \begin{aligned}
    &(\theta\cdot\nabla+\mu_a+\mu_s-\mu_s\bm{K})\phi(x,\theta)=q(x,\theta),\quad(x,\theta)\in X,\\
&\phi|_{\Gamma_-}=q_b(x,\theta)
  \end{aligned}
\right.
  \end{equation}
satisfy (i) of assumption \ref{as:1}, and its source term $q(x,\theta)\in\mathcal{L}^p(X)$,  then equation \eqref{eq:2} admits a unique solution $\phi(x,\theta)$ that satisfies 
\begin{equation}
  \label{eq:6}
  \norm{\phi}_p\leq c_4\norm{q}_p+c_5\norm{q_b}_{\mathcal{L}^p(\Gamma_-,|\theta\cdot\nu|)}
\end{equation}
for some $c_3,\ c_4>0$ depending on $\mu_a$ and $\Omega$.
\end{lem}
\begin{proof}
 For $1\leq p<\infty$, multiplying $\phi|\phi|^{p-2}$ on both sides of the equation \eqref{eq:2} and integrating over $X$, we can obtain
  \begin{multline}
\label{eq:8}
   \frac{1}{p}(\norm{\phi}_{\mathcal{L}^p(\Gamma_+,|\theta\cdot\nu|)}^p-\norm{\phi}^p_{\mathcal{L}^p(\Gamma_-,|\theta\cdot\nu|)})+\norm{\mu_a^{\frac{1}{p}}\phi}^p_p+\int_X\mu_s(x)\phi^2|\phi|^{p-2}\rd x\rd \theta\\
=\int_X\mu_s(\bm{K}\phi)\phi|\phi|^{p-2}\rd x\rd \theta+\int_X q \phi|\phi|^{p-2}\rd x\rd \theta.
  \end{multline}
By the property of the scattering kernel $f$ \eqref{eq:9} and Young's inequality, we have
\begin{equation*}
  \begin{aligned}
   &\left|\int_X\mu_s(\bm{K}\phi)\phi|\phi|^{p-2}\rd x\rd \theta\right|\\
=&\left|\int_{\Omega}\mu_s(x)\int_{\mathcal{S}^{d-1}}\int_{\mathcal{S}^{d-1}}f(\theta,\theta')\phi(x,\theta')\rd \theta'\phi(x,\theta)|\phi(x,\theta)|^{p-2}\rd \theta\rd x\right|\\
\leq&\int_{\Omega}\mu_s(x)\int_{\mathcal{S}^{d-1}}\int_{\mathcal{S}^{d-1}}f(\theta,\theta')\left|\phi(x,\theta')\right|\rd \theta'\left|\phi(x,\theta)\right|^{p-1}\rd \theta \rd x\\
\leq&\int_{\Omega}\mu_s(x)\int_{\mathcal{S}^{d-1}}\int_{\mathcal{S}^{d-1}}f(\theta,\theta')\left(\frac{1}{p}\left|\phi(x,\theta')\right|^p+\frac{p-1}{p}\left|\phi(x,\theta)\right|^p\right)\rd \theta' \rd \theta \rd x\\
\leq&\int_\Omega\mu_s(x)\int_{\mathcal{S}^{d-1}}\left(\int_{\mathcal{S}^{d-1}}f(\theta,\theta')\rd \theta\right)\frac{1}{p}\left|\phi(x,\theta')\right|^p\rd \theta'\rd x\\
&\phantom{aaa}+\int_\Omega\mu_s(x)\int_{\mathcal{S}^{d-1}}\left(\int_{\mathcal{S}^{d-1}}f(\theta,\theta')\rd \theta'\right)\frac{p-1}{p}\left|\phi(x,\theta)\right|^p\rd \theta \rd x\\
=&\int_\Omega\mu_s(x)\int_{\mathcal{S}^{d-1}}\left|\phi\right|^p\rd \theta\rd x.
  \end{aligned}
\end{equation*}
Moreover, by H\"older's and Young's inequalities, we can get
\begin{equation}
  \label{eq:10}
  \begin{aligned}
    & \left|\int_X q\phi \left|\phi\right|^{p-2}\rd \theta\rd x\right|\\
\leq&\int_X\left|\mu_a^{\frac{1-p}{p}}q\right|\left|\mu_a^{\frac{p-1}{p}}\left|\phi\right|^{p-1}\right|\rd \theta\rd x\\
\leq&\norm{\mu_a^{\frac{1-p}{p}}q}_p\norm{\mu_a^\frac{1}{p}\phi}_p^{p-1}\\
\leq&\frac{1}{p}\norm{\mu_a^{\frac{1-p}{p}}q}_p^p+\frac{p-1}{p}\norm{\mu_a^\frac{1}{p}\phi}_p^p.
  \end{aligned}
\end{equation}
Accordingly,
\begin{equation}
  \label{eq:11}
  \norm{\mu_a^{\frac{1}{p}}\phi}^p_p\leq\norm{\phi}^p_{\mathcal{L}^p(\Gamma_-,|\theta\cdot\nu|)}+\norm{\mu_a^{\frac{1-p}{p}}q}_p^p.
\end{equation}
Considering the boundedness of $\mu_a$, therefore
\begin{equation*}
  \norm{\phi}_p\leq c_3\norm{q}_p+c_4\norm{q_b}_{\mathcal{L}^p(\Gamma_-,|\theta\cdot \nu|)}.
\end{equation*}

When $p\rightarrow\infty$, it is obvious that inequality \eqref{eq:11} still holds, and so does inequality \eqref{eq:6}.
\end{proof}
With lemma \ref{lem:1}, naturally we can obtain the uniqueness and stability of RTE system \eqref{eq:1}. Similar proofs of uniqueness and stability of stationary RTE can be referred in \cite{Egger2013}.
\begin{theorem}
\label{the:1}
  If the optical coefficients and boundary condition $q_b$ in equation \eqref{eq:1} satisfy assumption \ref{as:1}, RTE system \eqref{eq:1} admits unique solutions $\phi_x(x,\theta)$ and $\phi_m(x,\theta)$ that satisfy
  \begin{equation}
    \label{eq:42}
    \norm{\phi_x}_p+\norm{\phi_m}_p\leq c_5\norm{q_b}_{\mathcal{L}^p(\Gamma_-,|\theta\cdot \nu|)}
  \end{equation}
for some $c_5>0$ depending on $\mu_{a,x}$, $\mu_{a,m}$ and $\Omega$.
\end{theorem}

\begin{proof}
  Owing to $\norm{\eta\mu_{a,xf}(\tilde{\bm{A}}\phi_x)}\leq c'\norm{\eta}\norm{\mu_{a,xf}}\norm{\phi_x}$, using lemma \ref{lem:1}, the result is obvious.

\end{proof}

By lemma \ref{the:1}, the solutions of \eqref{eq:1} $\phi_x$ and $\phi_m$ depend continuously on boundary $q_b$. By this result, $\phi_x$ and $\phi_m$ depend continuously on $\mu_{a,xf}$. Furthermore, so internal data $h$ does, which is proved in~\cite{Ren2015}

\section{Hybrid reconstruction of fluorescence absorption coefficient $\mu_{a,xf}$}
\label{sec:3}
In this section, we suggest hybrid algorithm on the reconstruction of $\mu_{a,xf}$ combined with SIM and the nonlinear optimization method. Firstly, we present SIM algorithm, nonlinear optimization algorithm and hybrid algorithm. Then, update scheme for multi-measurement data is derived.

\subsection{The SIM algorithm}
\label{sec:3.1}
For convenience, we omit the subscript '$s$' on measurement in section \ref{sec:3.1} and \ref{sec:3.3}. From the definition \eqref{eq:4} of $h(x)$, given exact data $h^*$, we easily get a fixed-point iteration scheme
\begin{equation*}
  \begin{aligned}
    \mu_{a,xf}^{i+1}(x)=\bm{F}_1(\mu_{a,xf}^i)(x):&=\frac{1}{1-\eta(x)}\left(\frac{h^*(x)-\mu_{a,m}(\bm{A}\phi_m^i)(x)}{(\bm{A}\phi_x^i)(x)}-\mu_{a,xi}(x)\right)\\
& =\frac{1}{1-\eta}\left(\frac{h^*-\mu_{a,m}\bm{A}(\bm{T_2}(\mu_{a,xf}^i))}{\bm{A}\bm{T}_1(\mu_{a,xf}^i)}\right).
  \end{aligned}
\end{equation*}
For convergence, it is often required that the initial guess is close to the true coefficient, which is also discussed in \cite{Cox2009,Cox2006,Harrison2013} for QPAT. Based on QPAT, an improved fixed-point iteration is proposed in \cite{Wang2017}, where absorption coefficient $\mu_{a}^i$ satisfies $\mu_a^i\leq\mu_a^{i+1}<\mu_a^*$. And assuming corresponding data are $h^i$ and $h^*$, it is proved that the data $h^i$ converges to $h^*$ in $\mathcal{L}^1$-norm. Heuristically, for a initial guess $\mu_{a,xf}^0<\mu_{a,xf}^*$, we update $\mu_{a,xf}^i$ by
\begin{equation}
  \label{eq:alg1:1}
 \mu_{a,xf}^{i+1}(x)=\bm{F}_2(\mu_{a,xf}^i)(x):=\max\left\{\mu_{a,xf}^i(x),\bm{F}_1(\mu_{a,xf}^i)\right\}.
\end{equation}
Despite the scheme \eqref{eq:alg1:1} is similar with the improved fixed-point iterative method in \cite{Wang2017}, the boundedness of sequence ($\mu_{a,xf}^i<\mu_{a,xf}^*$) is not guaranteed theoretically. So monotonically increasing sequence $\mu_{a,xf}^i$ may exceed and even stay away from $\mu_{a,xf}^*$. Therefore, we propose its variant SIM, see algorithm \ref{alg:2}.

\begin{algorithm}[H]
\caption{Squeeze iteration method (SIM)\label{alg:2}}
\begin{algorithmic}[1]
 \REQUIRE Given initialization $\underline{\mu_{a,xf}}^0=c_1$ and $\overline{\mu_{a,xf}}^0=c_2$ with $c_1$ and $c_2$ mentioned in assumption \ref{as:1}, data $h^*$, coefficients $\eta$, $\mu_{a,xi}$, $\mu_{a,m}$, $\mu_{s,x}$, $\mu_{s,m}$, boundary source $q_b$, and tolerance $\epsilon_1$.
\FOR {$i=0,1,\dots$}
\STATE $\underline{\phi_x}^i=\bm{T}_1(\underline{\mu_{a,xf}}^i)$, $\overline{\phi_x}^i=\bm{T}_1(\overline{\mu_{a,xf}}^i)$;
\STATE Solve the second RTE in equation \eqref{eq:1} with source terms $\eta\overline{\mu_{a,xf}}^i(\tilde{\bm{A}}\underline{\phi_x}^i)$ and $\eta\underline{\mu_{a,xf}}^i(\tilde{\bm{A}}\overline{\phi_x}^i)$ respectively to obtain $\overline{\phi_m}^i$ and $\underline{\phi_m}^i$; 
\STATE 
\begin{equation}
  \label{eq:alg2:1}
  \underline{\mu_{a,xf}}^{i+1}(x)=\max\left\{\underline{\mu_{a,xf}}^i(x),\frac{1}{1-\eta}\left(\frac{h^*(x)-\mu_{a,m}(\bm{A}\overline{\phi_m}^i)(x)}{(\bm{A}\underline{\phi_x}^i)(x)}-\mu_{a,xi}(x)\right)\right\},
\end{equation}
 \begin{equation}
   \label{eq:alg2:2}
   \overline{\mu_{a,xf}}^{i+1}(x)=\min\left\{\overline{\mu_{a,xf}}^i(x),\frac{1}{1-\eta}\left(\frac{h^*(x)-\mu_{a,m}(\bm{A}\underline{\phi_m}^i)(x)}{(\bm{A}\overline{\phi_x}^i)(x)}-\mu_{a,xi}(x)\right)\right\};
 \end{equation}
\STATE If $\norm{\underline{\mu_{a,xf}}^{i+1}-\underline{\mu_{a,xf}}^{i}}/\norm{\underline{\mu_{a,xf}}^{i}}<\epsilon_1$ and $\norm{\overline{\mu_{a,xf}}^{i+1}-\overline{\mu_{a,xf}}^{i}}/\norm{\overline{\mu_{a,xf}}^{i}}<\epsilon_1$, end up with $\mu_{a,xf}=\underline{\mu_{a,xf}}^{i+1}$; otherwise, go to step 2.
\ENDFOR
\end{algorithmic}
\end{algorithm}

We claim that sequences $\{\underline{\mu_{a,xf}^i(x)}\}$ and $\{\overline{\mu_{a,xf}}^i(x)\}$ are bounded monotonic and therefore converged. First of all, we present several properties of stationary RTE.
\begin{theorem}[\cite{Case1963}]
  \label{th:2}
  For RTE \eqref{eq:2} with bounded absorption and scattering coefficient, if source $ 0\leq q\in\mathcal{L}^\infty$ and boundary source $0\leq q_b\in\mathcal{L}^\infty$, there exists unique non-negative solution $\phi(x,\theta)$.
\end{theorem}

\begin{lem}[\cite{Wang2017}]
\label{lem:2}
  Let $\phi^1(x)$ and $\phi^2(x)$ be the solutions of RTEs
  \begin{equation}
    \label{eq:9}\left\{
    \begin{aligned}
      &\left[\theta\cdot \nabla +\mu_a(x)+\mu_s(x)\right]\phi(x,\theta)-\mu_s(x)(\bm{K}\phi)(x,\theta)=0,\ && (x,\theta)\in \Omega\times \mathcal{S}^{n-1},\\
&\phi(x,\theta)=q_b(x,\theta), &&(x,\theta)\in \Gamma_-,
    \end{aligned}
\right.
  \end{equation}
with $\mu_a$ being $\mu_a^1$ and $\mu_a^2$ respectively.
Then $\phi^1(x,\theta)\geq \phi^2(x,\theta)$ provided $\mu_a^1(x)\leq \mu_a^2(x)(\forall x\in \Omega)$. Note that the superscripts of $\mu_a^1$ and $\mu_a^2$ are only to distinguish the different absorption coefficients.
\end{lem}

\begin{remark}
  Essentially, theorem \ref{th:2} implies the monotonicity of the solution $\phi$ of RTE with respect to source, including source term $q$ and boundary $q_b$. Naturally, we conclude that $\phi$ is monotonic with respect to absorption coefficient $\mu_a$.
\end{remark}
Due to the monotonic relationship between the absorption coefficient and the solution of corresponding RTE, we can get the monotonicity of the sequences obtained in algorithm \ref{alg:2} in following.
\begin{theorem}
  \label{th:3}
If $\underline{\mu_{a,xf}}^0<\mu_{a,xf}^*<\overline{\mu_{a,xf}}^0$ and other optical coefficients and boundary satisfy assumption \ref{as:1}, then sequences $\{\underline{\mu_{a,xf}}^i\}_{i=0}^\infty$, $\{\overline{\mu_{a,xf}}^i\}_{i=0}^\infty$, $\{\underline{\phi_x}^i\}_{i=0}^\infty$, $\{\overline{\phi_x}^i\}_{i=0}^\infty$, $\{\underline{\phi_m}^i\}_{i=0}^\infty$ and $\{\overline{\phi_m}^i\}_{i=0}^\infty$ from algorithm \ref{alg:2} satisfy
\begin{align}
\label{eq:43}
  &\underline{\mu_{a,xf}}^0\leq\underline{\mu_{a,xf}}^1\leq\dots\leq\underline{\mu_{a,xf}}^i\leq\dots\leq\underline{\mu_{a,xf}}^*\leq\dots<\overline{\mu_{a,xf}}^i\leq\dots\leq\overline{\mu_{a,xf}}^1\leq\overline{\mu_{a,xf}}^0,\\
\label{eq:44}
&\underline{\phi_x}^0\geq\underline{\phi_x}^1\geq\dots\geq\underline{\phi_x}^i\geq\dots\geq\underline{\phi_x}^*\geq\dots\geq\overline{\phi_x}^i\geq\dots\geq\overline{\phi_x}^1\geq\overline{\phi_x}^0,\\
\label{eq:45}
&\underline{\phi_m}^0\leq\underline{\phi_m}^1\leq\dots\leq\underline{\phi_m}^i\leq\dots\leq\underline{\phi_m}^*\leq\dots<\overline{\phi_m}^i\leq\dots\leq\overline{\phi_m}^1\leq\overline{\phi_m}^0.
\end{align}

\end{theorem}

\begin{proof}
We assume \eqref{eq:43}, \eqref{eq:44} and \eqref{eq:45} hold for $i$. Then from $\underline{\phi_x}^i>0$, obviously we have
\begin{equation*}
  \frac{1}{1-\eta}\left(\frac{h^*(x)-\mu_{a,m}(\bm{A}\overline{\phi_m}^i)(x)}{(\bm{A}\underline{\phi_x}^i)(x)}-\mu_{a,xi}(x)\right)\leq \frac{1}{1-\eta}\left(\frac{h^*(x)-\mu_{a,m}(\bm{A}\phi_m^*)(x)}{(\bm{A}\phi_x^*)(x)}-\mu_{a,xi}(x)\right)=\mu_{a,xf}^*.
\end{equation*}
Combining $\underline{\mu_{a,xf}}^i\leq \mu_{a,xf}^*$, naturally $\underline{\mu_{a,xf}}^i\leq\underline{\mu_{a,xf}}^{i+1}\leq \mu_{a,xf}^*$ holds, which implies
\begin{equation*}
 \underline{\phi_x}^{i}\geq\underline{\phi_x}^{i+1}\geq\phi_x^*
\end{equation*}
from lemma \ref{lem:2}. Similarly,
\begin{equation*}
  \frac{1}{1-\eta}\left(\frac{h^*(x)-\mu_{a,m}(\bm{A}\underline{\phi_m}^i)(x)}{(\bm{A}\overline{\phi_x}^i)(x)}-\mu_{a,xi}(x)\right)\geq \frac{1}{1-\eta}\left(\frac{h^*(x)-\mu_{a,m}(\bm{A}\phi_m^*)(x)}{(\bm{A}\phi_x^*)(x)}-\mu_{a,xi}(x)\right)=\mu_{a,xf}^*
\end{equation*}
indicates
\begin{equation*}
  \overline{\mu_{a,xf}}^i\geq\overline{\mu_{a,xf}}^{i+1}\geq \mu_{a,xf}^*.
\end{equation*}
Then easily we have
\begin{equation*}
  \eta\overline{\mu_{a,xf}}^{i}(\tilde{\bm{A}}\underline{\phi_x}^{i})\geq\eta\overline{\mu_{a,xf}}^{i+1}(\tilde{\bm{A}}\underline{\phi_x}^{i+1})\geq \eta\mu_{a,xf}^*(\tilde{\bm{A}}\phi_x^*),
\end{equation*}
which induces
\begin{equation*}
  \overline{\phi_m}^{i}\geq\overline{\phi_m}^{i+1}\geq\phi_m^*
\end{equation*}
by theorem \ref{th:2}.
Using the same way, \eqref{eq:43}, \eqref{eq:44} and \eqref{eq:45} holds for $\overline{\phi_x}^i$, $\overline{\mu_{a,xf}}^i$ and $\underline{\phi_m^i}$.

Therefore it completes the proof.
\end{proof}

\subsection{Nonlinear optimization method}
\label{sec:3.2}
As we all know, optimization method as a state-of-art can relatively stably minimize error function in image reconstruction generally. Here, we use log-type function
\begin{equation}
\label{eq:49}
  \mathcal{F}(\mu_{a,xf})=\frac{1}{2}\sum_{s=0}^{S-1}\norm{\log (\bm{H}_s(\mu_{a,xf})-\log(h^*_s))}_2^2,
\end{equation}
as our error function. Compared with more widely used least square error function $\frac{1}{2}\sum_{s=0}^{S-1}\norm{ \bm{H}_s(\mu_{a,xf})-h_s^*}_2^2$, log-type function accelerates convergence~\cite{Tarvainen2012}. And some discussion about log-type function can be found in~\cite{Tarvainen2012,Wang2017}. For fixed $\mu_{a,xf}$ and any feasible direction $h_f$ (there exists $\delta>0$ such that for any $0<s<\delta$, $\mu_{a,xf}+sh_f\in\mathcal{D}(\Omega)$), the directional derivative $\nabla\mathcal{F}$ of $\mathcal{F}$ is defined by
\begin{equation}
  \label{eq:50}
  \mathcal{F}'(\mu_{a,xf})(h_f)=\left<\nabla\mathcal{F},h_f\right>_{\mathcal{L}^2(\Omega)}.
\end{equation}
From \cite{Ren2015}, the directional derivative of $\bm{H}_s(\mu_{a,xf})(h_f)$ exists with respect to $\mu_{a,xf}$ in any feasible direction $h_f$, so \eqref{eq:50} is well-defined. In order to handle the implicit derivative, adjoint method is applied to get the gradient of log-type error function, detailed in appendix A. And for saving time, we take BB stepsize as our step in the direction of negative gradient to avoid linesearch which needs solve RTEs \eqref{eq:1} for several times. BB stepsize takes the value of
\begin{equation*}
  s_{k1}=\frac{(\mu_{a,xf}^k-\mu_{a,xf}^{k-1})^\top(\nabla\mathcal{F}_k-\nabla\mathcal{F}_{k-1})}{\norm{\nabla\mathcal{F}_k-\nabla\mathcal{F}_{k-1}}^2}
\end{equation*}
or
\begin{equation*}
  s_{k2}=\frac{\norm{\mu_{a,xf}^k-\mu_{a,xf}^{k-1}}^2}{(\mu_{a,xf}^k-\mu_{a,xf}^{k-1})^\top(\nabla\mathcal{F}_k-\nabla\mathcal{F}_{k-1})},
\end{equation*}
where $\nabla\mathcal{F}_k$ is the gradient of $\mathcal{F}$ when $\mu_{a,xf}=\mu_{a,xf}^k$~\cite{Barzilai1988}. And then the update scheme is
\begin{equation}
\label{eq:bbupdate}
  \mu_{a,xf}^{k+1}=\mu_{a,xf}^k-s_k\nabla\mathcal{F}_k.
\end{equation}
Based on BB stepsize, the nonlinear optimization method is presented in algorithm \ref{alg:3}.

\begin{algorithm}[H]
\caption{Nonlinear optimization method\label{alg:3}}
\begin{algorithmic}[1]
 \REQUIRE Given initialization $\underline{\mu_{a,xf}}^0=c_1$ mentioned in assumption \ref{as:1}, data $h^*_s\ (s=0,1,\dots,S-1)$, coefficients $\eta$, $\mu_{a,xi}$, $\mu_{a,m}$, $\mu_{s,x}$, $\mu_{s,m}$, boundary source $q_b\ (s=0,1,\dots,S-1)$, tolerance $\epsilon_1$ and $\epsilon_2$.
\FOR {$k=0,1,\dots$}
\STATE For $s=0,1,\dots,S-1$, calculate $\phi_{x,s}^k=\bm{T}_{1,s}(\mu_{a,xf}^k)$, $\phi_{m,s}^k=\bm{T}_{2,s}(\mu_{a,xf}^k)$, $h^k_s=\bm{H}_s(\mu_{a,xf}^k)$, $\mathcal{F}_k=\mathcal{F}(\mu_{a,xf}^k)$;
\STATE If $\mathcal{F}_k<\epsilon_1$, end up with $\mu_{a,xf}=\mu_{a,xf}^k$; otherwise go to next step; 
\STATE Solve adjoint RTE \eqref{eq:34}, \eqref{eq:35}, then obtain gradient $\nabla\mathcal{F}_k$ from \eqref{eq:36}.
\STATE If $\norm{\nabla\mathcal{F}_k}<\epsilon_2$, end up with $\mu_{a,xf}=\mu_{a,xf}^{k+1}$; otherwise go to next step.
\STATE If $k=0$, take $s_0$ small enough to ensure $\mathcal{F}_k$ decrease in the direction of $-\nabla\mathcal{F}_k$; otherwise take $s_k$ as $s_{k1}$ or $s_{k2}$. Then using \eqref{eq:bbupdate} to update $\mu_{a,xf}$.
\ENDFOR
\end{algorithmic}
\end{algorithm}

\begin{remark}
  In fact, the algorithm \ref{alg:3} also can be used to recover $\mu_{a,xf}$ and $\eta$ simultaneously, and the gradient of error function \eqref{eq:49} with respective to $\eta$ is also deduced in appendix A.
\end{remark}

\subsection{Hybrid method}
\label{sec:3.3}

In simulations, we find that the algorithm \ref{alg:2} converges quickly at first steps, but then the relative error increases after arriving minimum, see section \ref{sec:4.1}. In fact, even though $\underline{\mu_{a,xf}}^i\leq \mu_{a,xf}^*\leq\overline{\mu_{a,xf}}^i$ holds theoretically, it may still not hold in synthetic simulations. To stabilize the algorithm and attain higher accuracy, the optimization method that is considered stable is incorporated, see hybrid algorithm \ref{alg:4}.

\begin{algorithm}[H]
\caption{Hybrid method\label{alg:4}}
\begin{algorithmic}[1]
 \REQUIRE Given initialization $\underline{\mu_{a,xf}}^0=c_1$ and $\overline{\mu_{a,xf}}^0=c_2$ with $c_1$ and $c_2$ mentioned in assumption \ref{as:1}, data $h^*$, coefficients $\eta$, $\mu_{a,xi}$, $\mu_{a,m}$, $\mu_{s,x}$, $\mu_{s,m}$, boundary source $q_b$, tolerance $\epsilon_1$, $\epsilon_2$ and $\epsilon_3$.
\FOR {$i=0,1,\dots$}
\STATE $\underline{\phi_x}^i=\bm{T}_1(\underline{\mu_{a,xf}}^i)$ and $\overline{\phi_x}^i=\bm{T}_1(\overline{\mu_{a,xf}}^i)$.
\STATE Solve the second RTE in equation \eqref{eq:1} with source terms $\eta\overline{\mu_{a,xf}}^i(\tilde{\bm{A}}\underline{\phi_x}^i)$ and $\eta\underline{\mu_{a,xf}}^i(\tilde{\bm{A}}\overline{\phi_x}^i)$ to obtain $\overline{\phi_m}$ and $\underline{\phi_m}^i$; 
\STATE 
\begin{equation}
  \label{eq:alg4:1}
  \underline{\mu_{a,xf}}^{i+1}(x)=\max\left\{\underline{\mu_{a,xf}}^i(x),\frac{1}{1-\eta}\left(\frac{h^*(x)-\mu_{a,m}(\bm{A}\overline{\phi_m}^i)(x)}{(\bm{A}\underline{\phi_x}^i)(x)}-\mu_{a,xi}(x)\right)\right\}
\end{equation}
\begin{equation}
  \label{eq:alg4:2}
  \overline{\mu_{a,xf}}^{i+1}(x)=\min\left\{\overline{\mu_{a,xf}}^i(x),\frac{1}{1-\eta}\left(\frac{h^*(x)-\mu_{a,m}(\bm{A}\underline{\phi_m}^i)(x)}{(\bm{A}\overline{\phi_x}^i)(x)}-\mu_{a,xi}(x)\right)\right\}
\end{equation}
\STATE If $\norm{\underline{\mu_{a,xf}}^{i+1}-\underline{\mu_{a,xf}}^{i}}\Big/\norm{\underline{\mu_{a,xf}}^{i}}<\epsilon_1$ and $\norm{\overline{\mu_{a,xf}}^{i+1}-\overline{\mu_{a,xf}}^{i}}\Big/\norm{\overline{\mu_{a,xf}}^{i}}<\epsilon_1$, jump out of this loop with $\mu_{a,xf}^0=\underline{\mu_{a,xf}}^{i+1}$; otherwise go to step 2.
\ENDFOR
\FOR{$i=0,1,\dots$}
\STATE $\phi_x^i=\bm{T}_1(\mu_{a,xf}^i)$, $\phi_m^i=\bm{T}_2(\mu_{a,xf}^i)$, $h^i=\bm{H}(\mu_{a,xf}^i)$;
\STATE If $\mathcal{F}_i<\epsilon_2$, end up with $\mu_{a,xf}=\mu_{a,xf}^i$;
\STATE Calculate the gradient of $\nabla \mathcal{F}(\mu_{a,xf}^i)$. If $\norm{\nabla \mathcal{F}}<\epsilon_3$, end up with $\mu_{a,xf}=\mu_{a,xf}^i$; otherwise go to next step
\STATE Update $\mu_{a,xf}$ using BB stepsize.
\ENDFOR
\end{algorithmic}
\end{algorithm}

\begin{remark}
  The tolerance $\epsilon_1$ in algorithm \ref{alg:4} is generally set larger than that in algorithm \ref{alg:2}. Otherwise, due to the instability of SIM loop, $\underline{\mu_{a,xf}}^i$ may increase and go away from the true coefficient $\mu_{a,xf}^*$.
\end{remark}
By adjoint method, we need to solve two RTEs to obtain the gradient, see appendix A. Therefore, the optimization method and SIM method have the same computational cost at each step, which both need to solve RTE for four times. Optimization method is generally more stable than fixed-point iterative method. From another perspective, optimization method possibility falls into the local minimum if object function is not convex. Differently, fixed-point iteration depends more on the properties of iteration operator, and can converge to the true value if the operator is contracted. Although we cannot prove that the iteration operator of SIM is contracted, it is not expanded and sequence $\underline{\mu_{a,xf}}^i(x)$ converges due to its monotonicity and boundedness in the sense of infinite dimension. Therefore, in finite dimension, after a few steps of SIM, the sequence $\underline{\mu_{a,xf}}^i$ will soon be near the $\mu_{a,xf}^*$, and then optimization method is expected to stably approach the true value, which avoids the instability of SIM. This advantage of the hybrid method is more pronounced when the number of measurements  is small. 
\subsection{Multi-measurement case}
\label{sec:3.4}

Multi-measurement usually can be used to improve the stability in inverse problem. In QPAT, multi-measurement model has a good performance~\cite{Bal}. Omitting superscript $'*'$, assume $S$ measurements are available, denote our data matrix by
\begin{equation*}
  \bm{h}:=
  \begin{pmatrix}
   h_0&h_2&\dots&h_{S-1} 
  \end{pmatrix}
^\top.
\end{equation*}
We denote $\bm{A\phi_x}$ and $\bm{A\phi_m}$ by
\begin{equation*}
  \bm{A\phi_x}:=
  \begin{pmatrix}
    \bm{A}\phi_{x,0}&
\bm{A}\phi_{x,1}&
\dots&
\bm{A}\phi_{x,S-1}
  \end{pmatrix}^\top
\quad \text{and}\quad
  \bm{A\phi_m}:=
  \begin{pmatrix}
    \bm{A}\phi_{m,0}&
\bm{A}\phi_{m,1}&
\dots&
\bm{A}\phi_{m,S-1}
  \end{pmatrix}^\top.
\end{equation*}

Our goal is to estimate $\mu_{a,xf}$ such that
\begin{equation*}
  h_s=(\mu_{a,xi}+(1-\eta)\mu_{a,xf})(\bm{A}\phi_{x,s})+\mu_{a,m}(\bm{A}\phi_{m,s}),\quad \text{for } s=0,1,\dots, S-1.
\end{equation*}
Using least square model, we need to estimate
\begin{equation}
  \label{eq:46}
  \mu_{a,xf}:=\argmin\limits\limits_{\mu_{a,xf}\in\mathcal{D}(\Omega)}\norm{
\bm{h}
-(\mu_{a,xi}+(1-\eta)\mu_{a,xf})
\bm{A\phi_x}
-\mu_{a,m}
\bm{A\phi_m}
}_2^2.
\end{equation}
If $ \bm{A}\phi_{x,s}$ and $\bm{A}\phi_{m,s}$ are not related to $\mu_{a,xf}$, then the minimizer of \eqref{eq:46} is also the solution of 
\begin{equation}
  \label{eq:47}
  (\bm{A\phi_x})^\top(\bm{h}-\mu_{a,m}\bm{A\phi_m})=(\mu_{a,xi}+(1-\eta)\mu_{a,xf})(\bm{A\phi_x})^\top(\bm{A\phi_x}).
\end{equation}
Obviously we can estimate $\mu_{a,xf}$ using iteration
\begin{equation}
  \label{eq:48}
  \mu_{a,xf}^{i+1}=\frac{1}{1-\eta} \left(\frac{(\bm{A\phi_x}^i)^\top(\bm{h}-\mu_{a,m}\bm{A\phi_m}^i)}{(\bm{A\phi_x}^i)^\top(\bm{A\phi_x}^i)}-\mu_{a,xi}\right),
\end{equation}
where
\begin{equation*}
  \bm{A\phi_x}^i=  
  \begin{pmatrix}
    \bm{A}\phi_{x,0}^i&
\bm{A}\phi_{x,1}^i&
\dots&
\bm{A}\phi_{x,S-1}^i
  \end{pmatrix}^\top=
\begin{pmatrix}
    \bm{T}_{1,0}(\mu_{a,xf}^i)&
\bm{T}_{1,1}(\mu_{a,xf}^i)&
\dots&
\bm{T}_{1,S-1}(\mu_{a,xf}^i)
  \end{pmatrix}^\top,
\end{equation*}

\begin{equation*}
  \bm{A\phi_m}^i= 
  \begin{pmatrix}
    \bm{A}\phi_{m,0}^i&
\bm{A}\phi_{m,1}^i&
\dots&
\bm{A}\phi_{m,S-1}^i
  \end{pmatrix}^\top=
\begin{pmatrix}
    \bm{T}_{2,0}(\mu_{a,xf}^i)&
\bm{T}_{2,1}(\mu_{a,xf}^i)&
\dots&
\bm{T}_{2,S-1}(\mu_{a,xf}^i)
  \end{pmatrix}^\top.
\end{equation*}
	
Therefore, utilizing multi-measurement to iterate $\mu_{a,xf}$ is expected improve the algorithm stability. As for algorithm \ref{alg:2} and algorithm \ref{alg:4}, corresponding iteration scheme \eqref{eq:alg2:1}, \eqref{eq:alg2:2} and \eqref{eq:alg4:1}, \eqref{eq:alg4:2} can be respectively replaced by
\begin{equation}
  \label{eq:alg2:1v}
  \underline{\mu_{a,xf}}^{i+1}=\frac{1}{1-\eta} \left(\frac{(\bm{A\underline{\phi_x}}^i)^\top(\bm{h}-\mu_{a,m}\bm{A\overline{\phi_m}}^i)}{(\bm{A\underline{\phi_x}}^i)^\top(\bm{A\underline{\phi_x}}^i)}-\mu_{a,xi}\right)
\end{equation}
and
\begin{equation}
  \label{eq:alg2:2v}
  \overline{\mu_{a,xf}}^{i+1}=\frac{1}{1-\eta} \left(\frac{(\bm{A\overline{\phi_x}}^i)^\top(\bm{h}-\mu_{a,m}\bm{A\underline{\phi_m}}^i)}{(\bm{A\overline{\phi_x}}^i)^\top(\bm{A\overline{\phi_x}}^i)}-\mu_{a,xi}\right),
\end{equation}
where
\begin{equation*}
  \begin{aligned}
    & \bm{A\underline{\phi_x}}^i=  
  \begin{pmatrix}
    \bm{A}\underline{\phi_{x,0}}^i&
\bm{A}\underline{\phi_{x,1}}^i&
\dots&
\bm{A}\underline{\phi_{x,S-1}}^i
  \end{pmatrix}^\top,\\
&\bm{A\underline{\phi_m}}^i= 
  \begin{pmatrix}
    \bm{A}\underline{\phi_{m,0}}^i&
\bm{A}\underline{\phi_{m,1}}^i&
\dots&
\bm{A}\underline{\phi_{m,S-1}}^i
  \end{pmatrix}^\top,\\
& \bm{A\overline{\phi_x}}^i=  
  \begin{pmatrix}
    \bm{A}\overline{\phi_{x,0}}^i&
\bm{A}\overline{\phi_{x,1}}^i&
\dots&
\bm{A}\overline{\phi_{x,S-1}}^i
  \end{pmatrix}^\top,\\
&\bm{A\overline{\phi_m}}^i= 
  \begin{pmatrix}
    \bm{A}\overline{\phi_{m,0}}^i&
\bm{A}\overline{\phi_{m,1}}^i&
\dots&
\bm{A}\overline{\phi_{m,S-1}}^i
  \end{pmatrix}^\top.
  \end{aligned}
\end{equation*}

\section{Numerical simulations}
\label{sec:4}
In numerical simulations, given quantum efficiency $\eta(x)$, we investigate SIM algorithm \ref{alg:2}, nonlinear optimization method algorithm \ref{alg:3} and hybrid algorithm \ref{alg:4} only in 2D. The investigated region is a circle centered in $(0,0)$ with the radius 20. The anisotropic factor $g$ equals 0.9. We let $\mu_{a,xi}$ equals $\mu_{a,m}$ and $\mu_{s,x}$ equals $\mu_{s,m}$ as following:
\begin{equation*}
  \mu_{a,xi}=\mu_{a,m}=0.02+0.01\sin(\frac{\pi}{8}x)\text{, }\mu{_{s,x}=\mu{_{s,m}=2+\sin(\frac{\pi}{8}y)}}.
\end{equation*}
And they are illustrated in figure~\ref{fig:uaus}.
\begin{figure}[htpb]

    \centering
    \begin{minipage}{0.35\linewidth}
    \includegraphics[width=\textwidth]{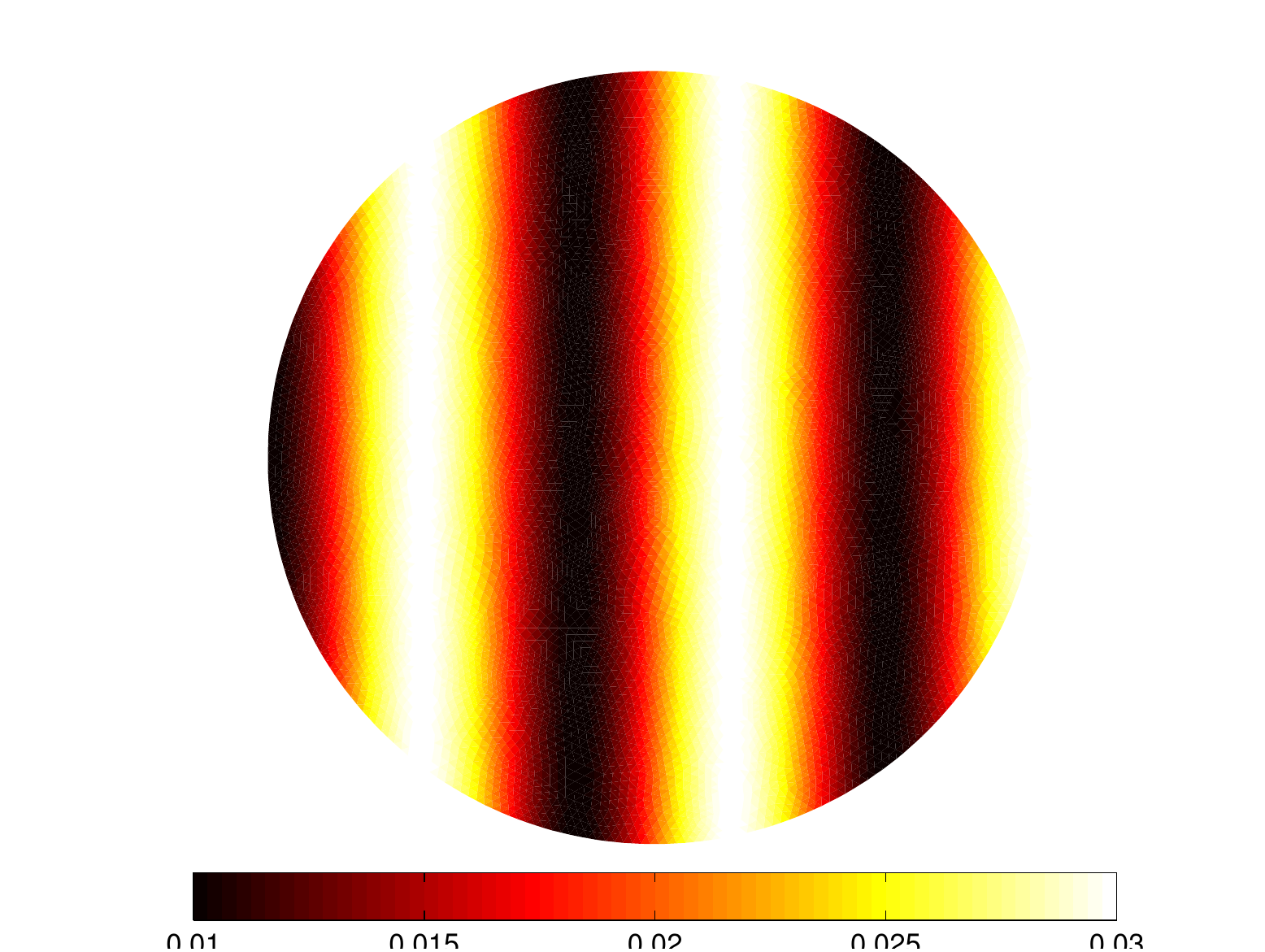}  
    \end{minipage}
\begin{minipage}{0.35\linewidth}
    \includegraphics[width=\textwidth]{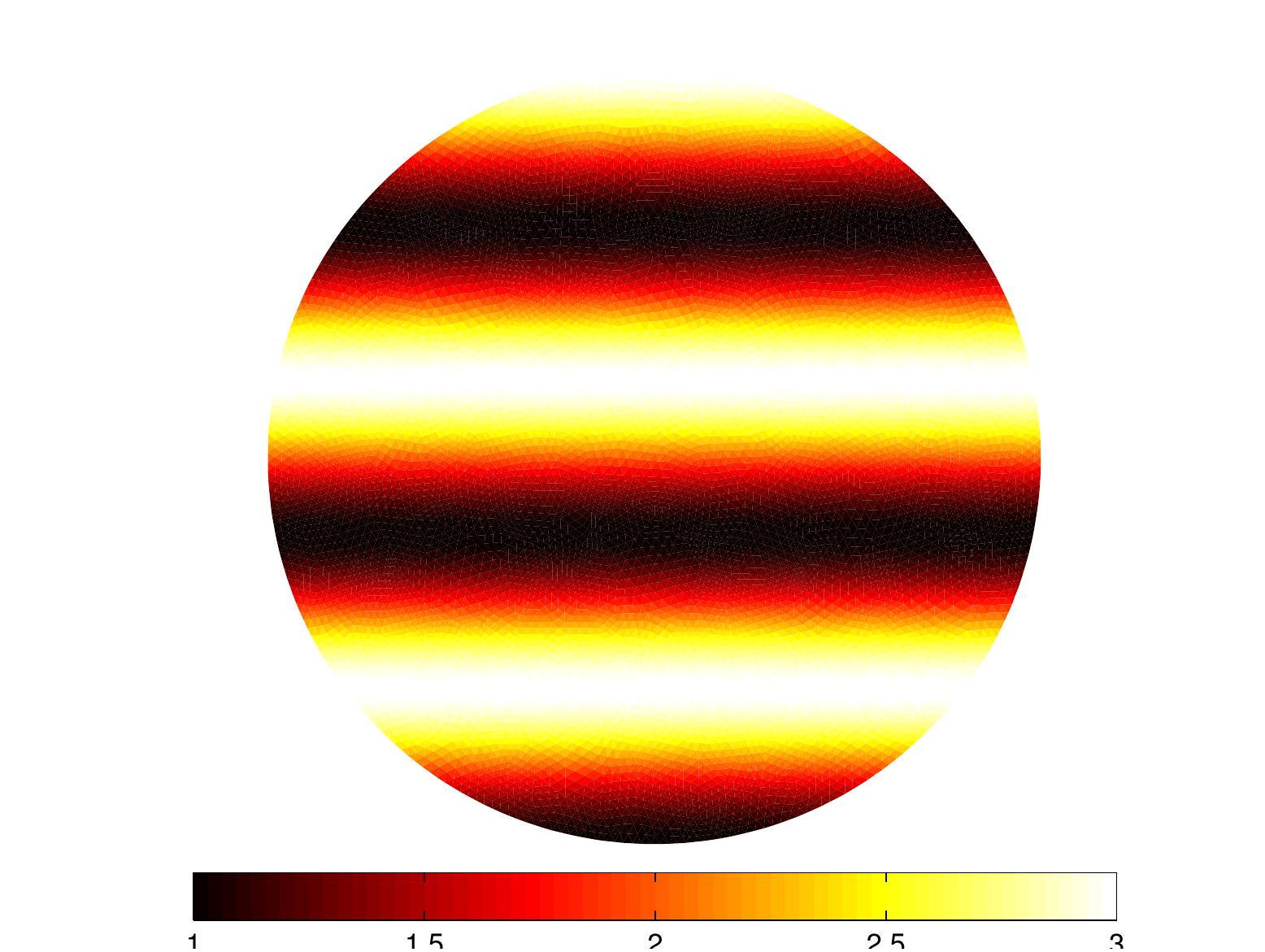}  
    \end{minipage}
 \caption{Intrinsic optical absorption coefficient $\mu_{a,xi}(\mu_{a,m})$ and scattering coefficient $\mu_{s,x}(\mu_{s,m})$\label{fig:uaus}}
\end{figure}
As for fluorescence absorption coefficient $\mu_{a,xf}$ and quantum efficiency $\eta$, we use two templates as follows:
\begin{enumerate}
\item Region $\Omega_0=\{(x,y)|x^2+y^2=20^2\}$ and five inclusions: $\Omega_1=\{(x,y)|(x+10)^2+(y-8)^2=4^2\}$, $\Omega_2=\{(x,y)|x^2+(y-8)^2=4^2\}$, $\Omega_3=\{(x,y)|(x+10)^2+(y+6)^2=4^2\}$, $\Omega_4=\{(x,y)|(x^2+(y+6)^2=4^2)\}$, and $\Omega_5=\{(x,y)|(x-10)^2/4^2+(y-2)^2/10^2=1\}$;
\item Region $\Omega_0=\{(x,y)|x^2+y^2=20^2\}$ and three inclusions: $\Omega_1=\{(x,y)|(x+10)^2+(y-4)^2=5^2\}$, $\Omega_2=\{(x,y)|5\leq x\leq 12,\ 0\leq y\leq 12\}$, and $\Omega_3=\{(x,y)|-8\leq x\leq 10, \ -12\leq y\leq -4\}$.
\end{enumerate}
And their $\mu_{a,xf}$ and $\eta$, see figure \ref{fig:original}, take the value as follows:
\begin{enumerate}
\item $\mu_{a,xf}=
  \begin{cases}
    0.02, &(x,y)\in \Omega_1\\
0.03, &(x,y)\in \Omega_5\\
0.04, &(x,y)\in\Omega_4\\
0.01, &(x,y)\in\Omega_0\backslash (\Omega_1\cup\Omega_4\cup\Omega_5)
  \end{cases}
$
 and 
$\eta=
\begin{cases}
0.5,&(x,y)\in\Omega_2\\
0.6,&(x,y)\in\Omega_3\\
0.7,&(x,y)\in\Omega_6\\
  0.1,&(x,y)\in\Omega_0\backslash(\Omega_2\cup\Omega_3\cup\Omega_4)
\end{cases};
$
\item $\mu_{a,xf}=
  \begin{cases}
    0.02,&(x,y)\in\Omega_2\\
0.03,&(x,y)\in\Omega_3\\
0.04,&(x,y)\in\Omega_1\\
0.01,&(x,y)\in\Omega_0\backslash(\Omega_1\cup\Omega_2\cup\Omega_3)
  \end{cases}
$
 and 
$\eta=
\begin{cases}
 0.5,&(x,y)\in\Omega_2\\
0.6,&(x,y)\in\Omega_3\\
0.7,&(x,y)\in\Omega_1\\
0.1,&(x,y)\in\Omega_0\backslash(\Omega_1\cup\Omega_2\cup\Omega_3)
\end{cases}.
$
\end{enumerate}

\begin{figure}[htpb]
    \centering
    \begin{minipage}{0.35\linewidth}
    \includegraphics[width=\textwidth]{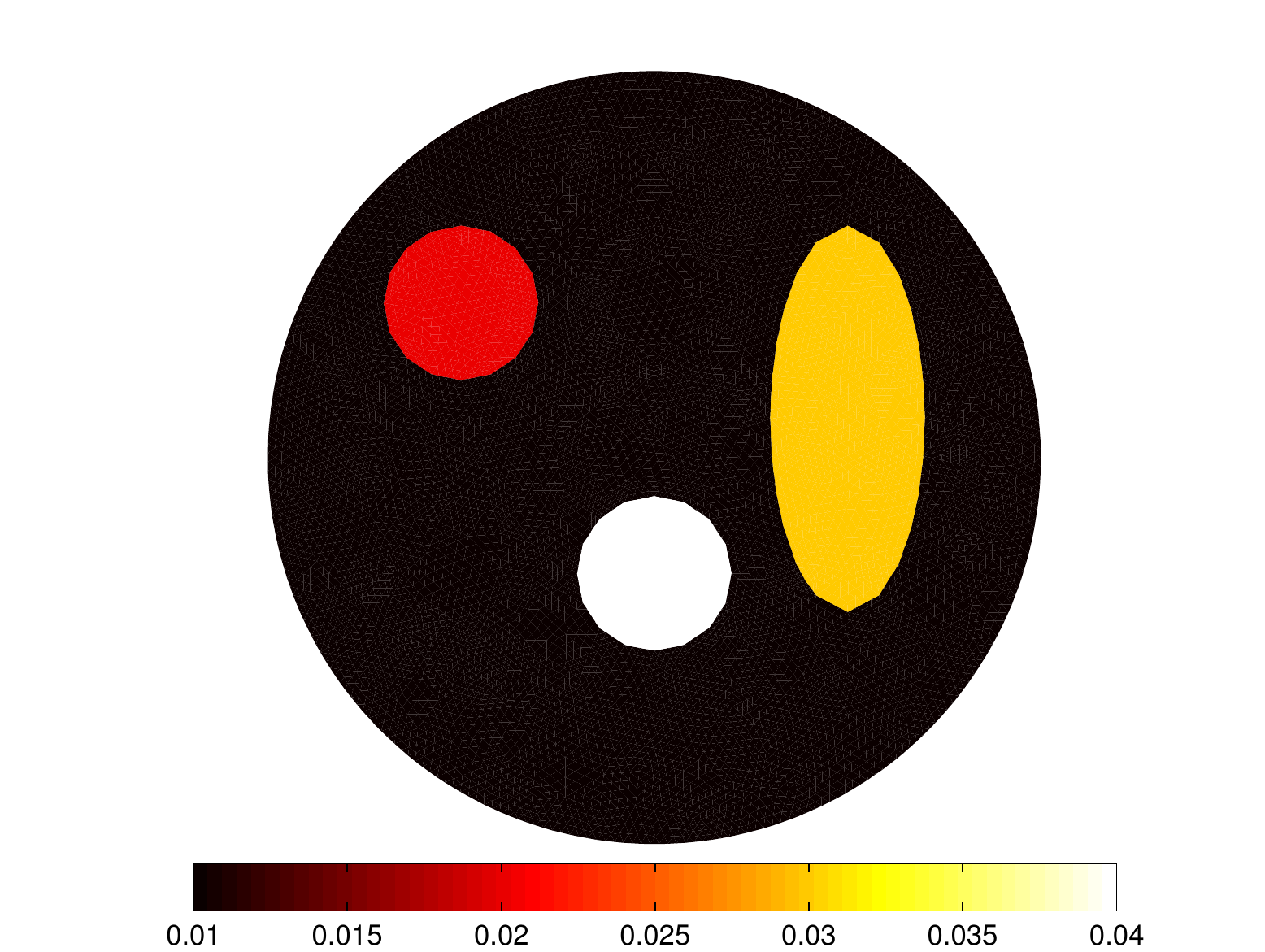}  
    \end{minipage}
   \begin{minipage}{0.35\linewidth}
    \includegraphics[width=\textwidth]{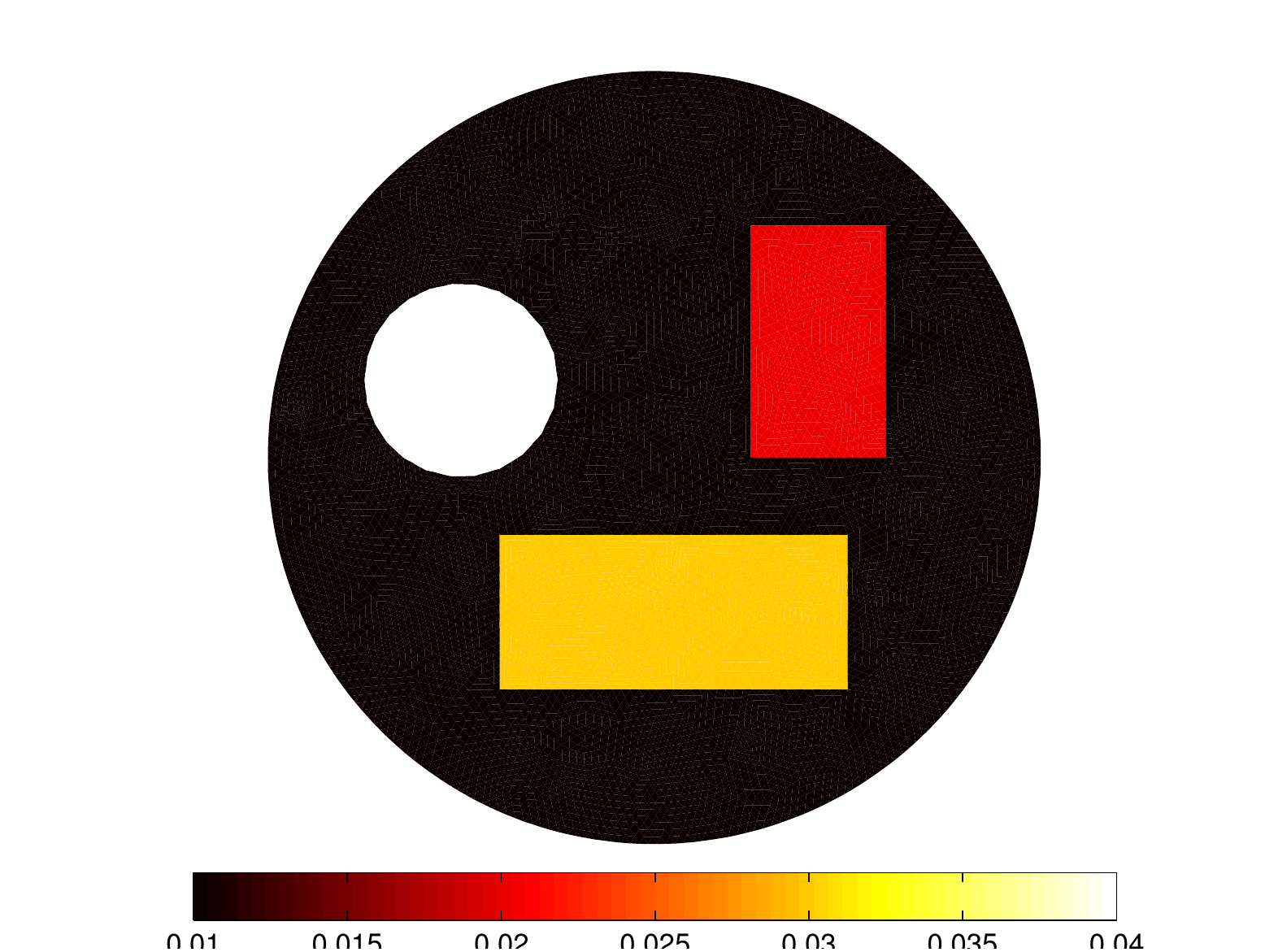}  
    \end{minipage}
\\
    \begin{minipage}{0.35\linewidth}
    \includegraphics[width=\textwidth]{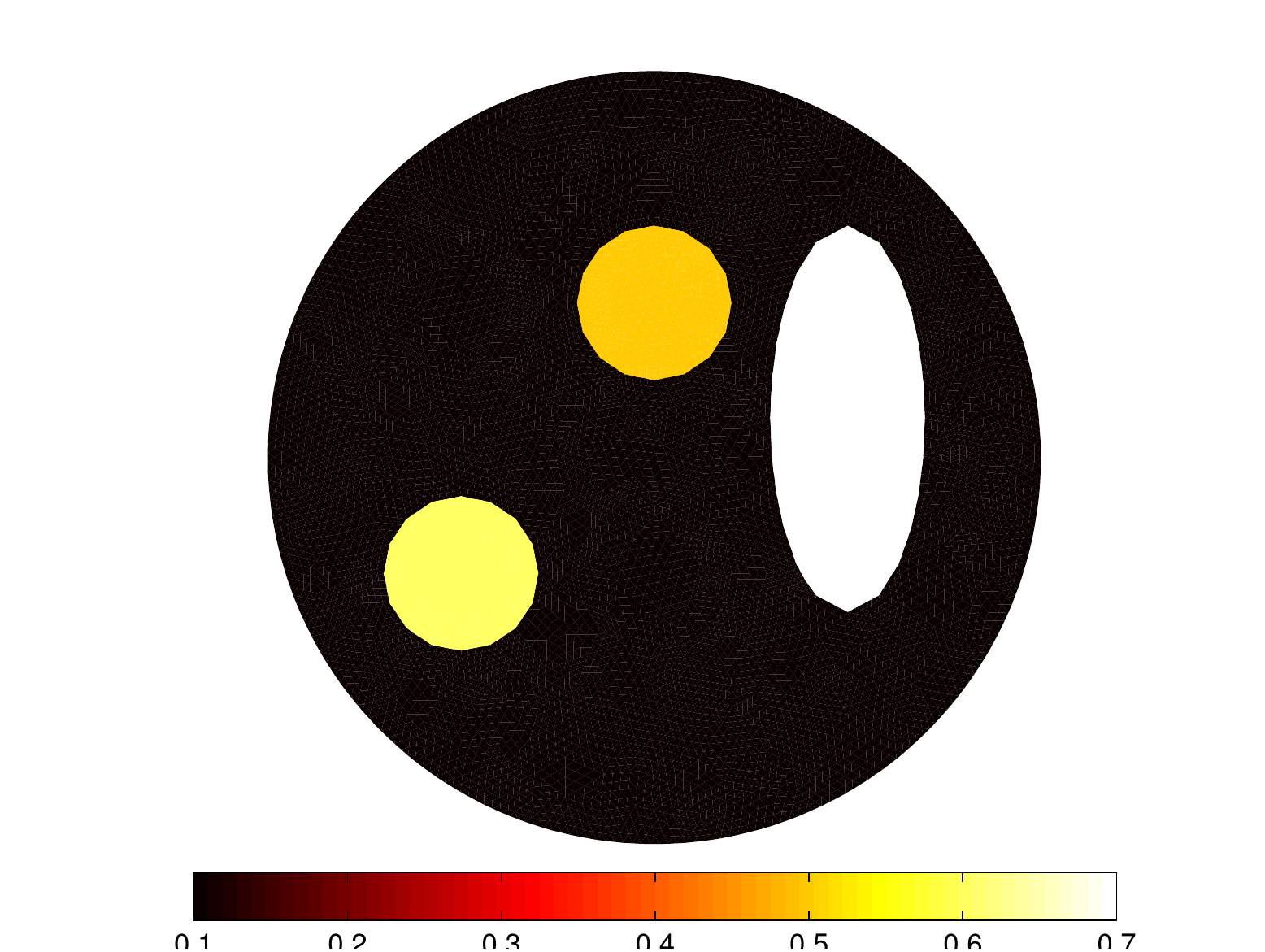}  
    \end{minipage}
    \begin{minipage}{0.35\linewidth}
    \includegraphics[width=\textwidth]{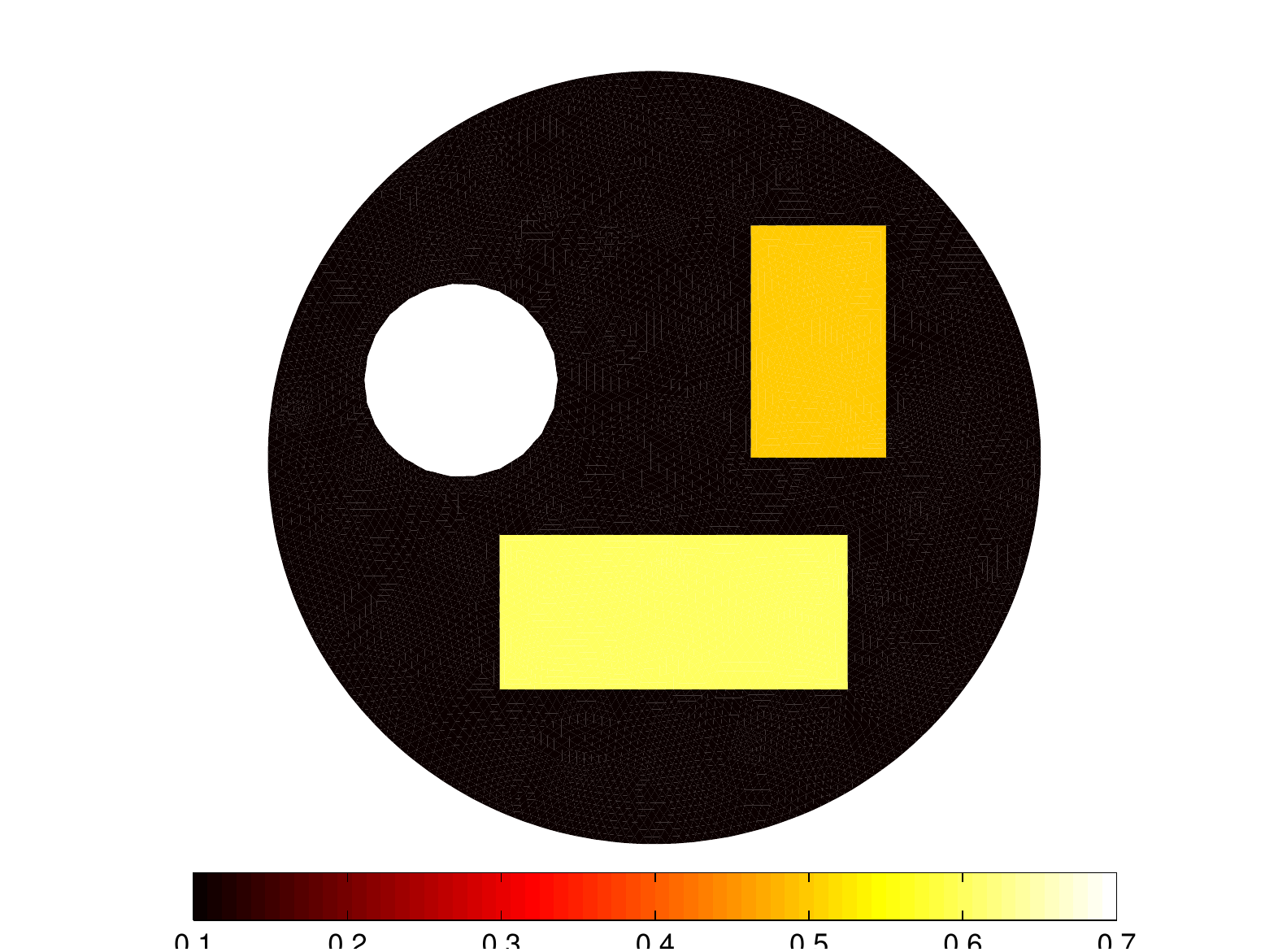}  
    \end{minipage}
 \caption{Original fluorescence yield. Top row: fluorescence absorption coefficient $\mu_{a,xf}$. Bottom row: quantum efficiency\label{fig:original}}
\end{figure}
From figure \ref{fig:original}, investigated templates contain inclusions with smooth and sharp edges and their fluorescent absorption coefficients are piecewise constant. We apply discontinuous Galerkin (DG) method combined with multigrid method to solve RTE system \eqref{eq:1}, and the details about algorithm and its convergence refer to~\cite{Gao,Gao2013}. As for adjoint RTE \eqref{eq:34} and \eqref{eq:35}, similar algorithm and corresponding convergence are presented in~\cite{Wang2017}. Compared with other finite element methods, such as streamline diffusion modification, DG not only admits jumps or smooth borders, but also it reduces the problem to a sparse $3\times 3$ block diagonal system, which means we can attain the solution by solving $3\times 3$ linear system one by one. Using two templates illustrated in figure \ref{fig:original}, we solve forward problem on unstructured mesh with 16640 and 17376 triangles respectively. There are four available measurements in the position of $(20,0)$, $(0,20)$, $(-20,0)$ and $(0,-20)$. The discrete data is still denoted by $h$. To test the stability of algorithms with respect to noise, we add Gaussian noise to the data of the form
\begin{equation*}
  \widetilde{h}=h(1+\epsilon \mathcal{N}),
\end{equation*}
where $\mathcal{N}$ is a standard Gaussian random matrix with the same size as $h^*$ and $\epsilon$ represents the level of noise. We use $\epsilon_f$ to measure the relative distance between estimating $\mu_{a,xf}$ and $\mu_{a,xf}^*$, which is defined by
\begin{equation*}
  \epsilon_f:=\frac{\norm{\mu_{a,xf}-\mu_{a,xf}^*}_2}{\norm{\mu_{a,xf}^*}_2}
\end{equation*}

\subsection{The effect of different mesh on SIM algorithm}
\label{sec:4.1}
To test the effect of different meshes on algorithm SIM, we apply the algorithm \ref{alg:2} to the first template. There are five unstructured triangular mesh $T_1$, $T_2$ and $T_3$ containing 7392, 8074, 11872 triangles respectively. Forward problem is solved in triangulation $T_0$ with 16640 triangles. And one, two, three and four measurements are applied respectively to test the effect of multi-measurement. The specific relative error $\epsilon_f$ are shown in figure \ref{fig:invcri}. We find that $\epsilon_f$ decreases at first steps, then increases quickly after arriving minimum. So some stabilization scheme need to be incorporated into our SIM method. 

\begin{figure}[htpb]
    \centering
\subfloat[one measurement\label{fig:invcri1}]{\includegraphics[width=0.45\textwidth]{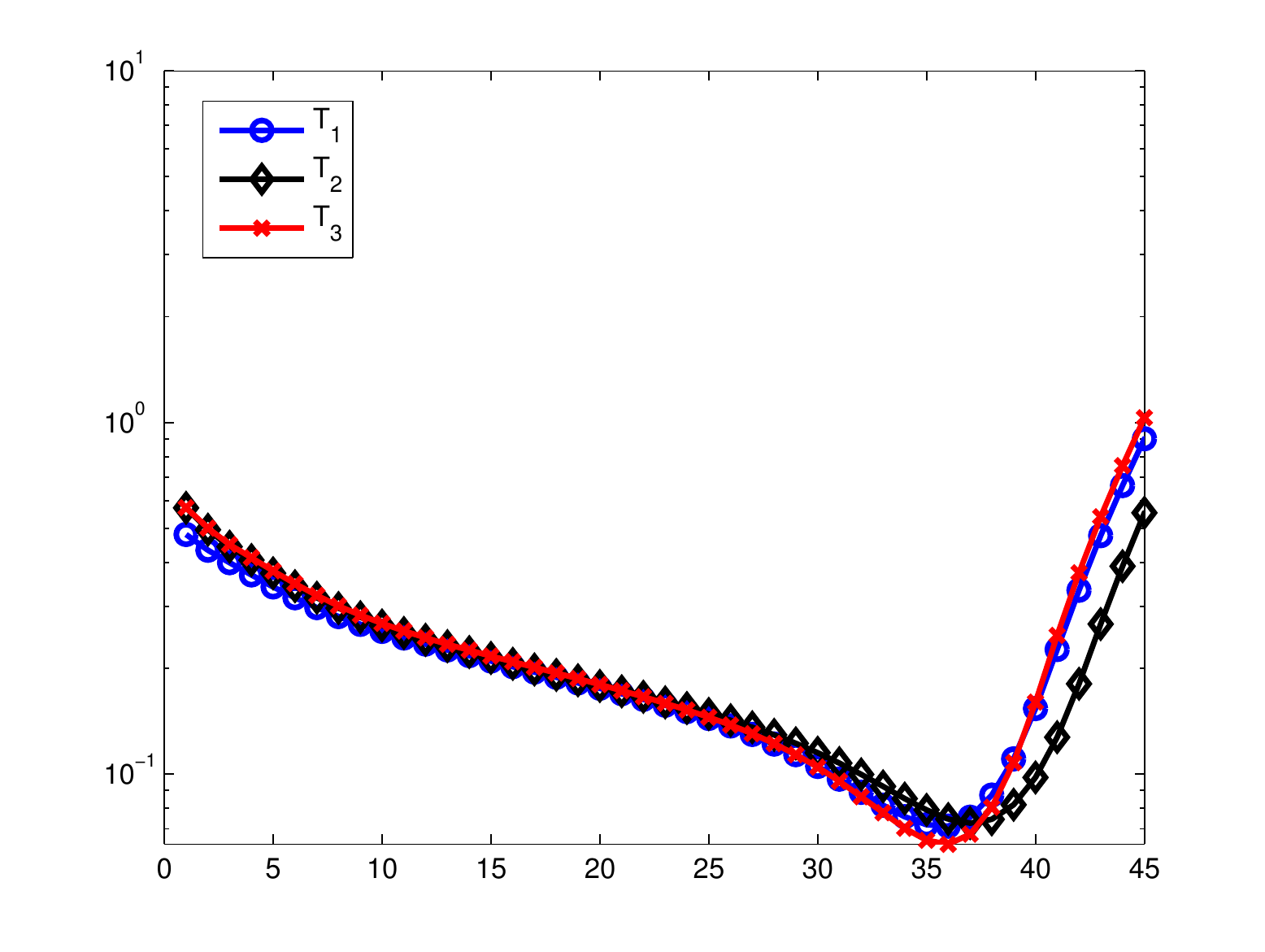}}
\subfloat[two measurements\label{fig:invcri2}]
 {\includegraphics[width=0.45\textwidth]{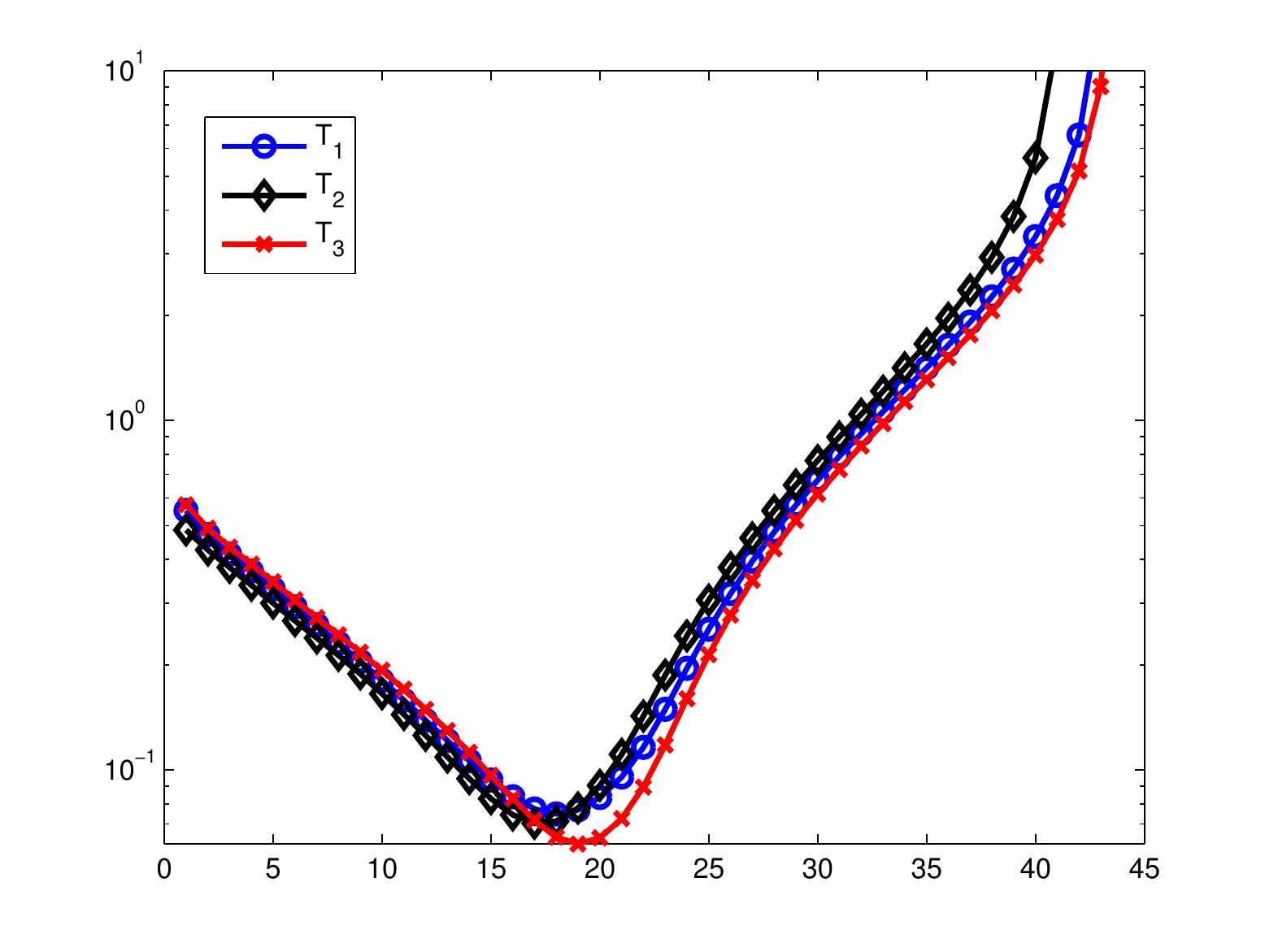}}\\
\subfloat[three measurements\label{fig:invcri3}]
 {\includegraphics[width=0.45\textwidth]{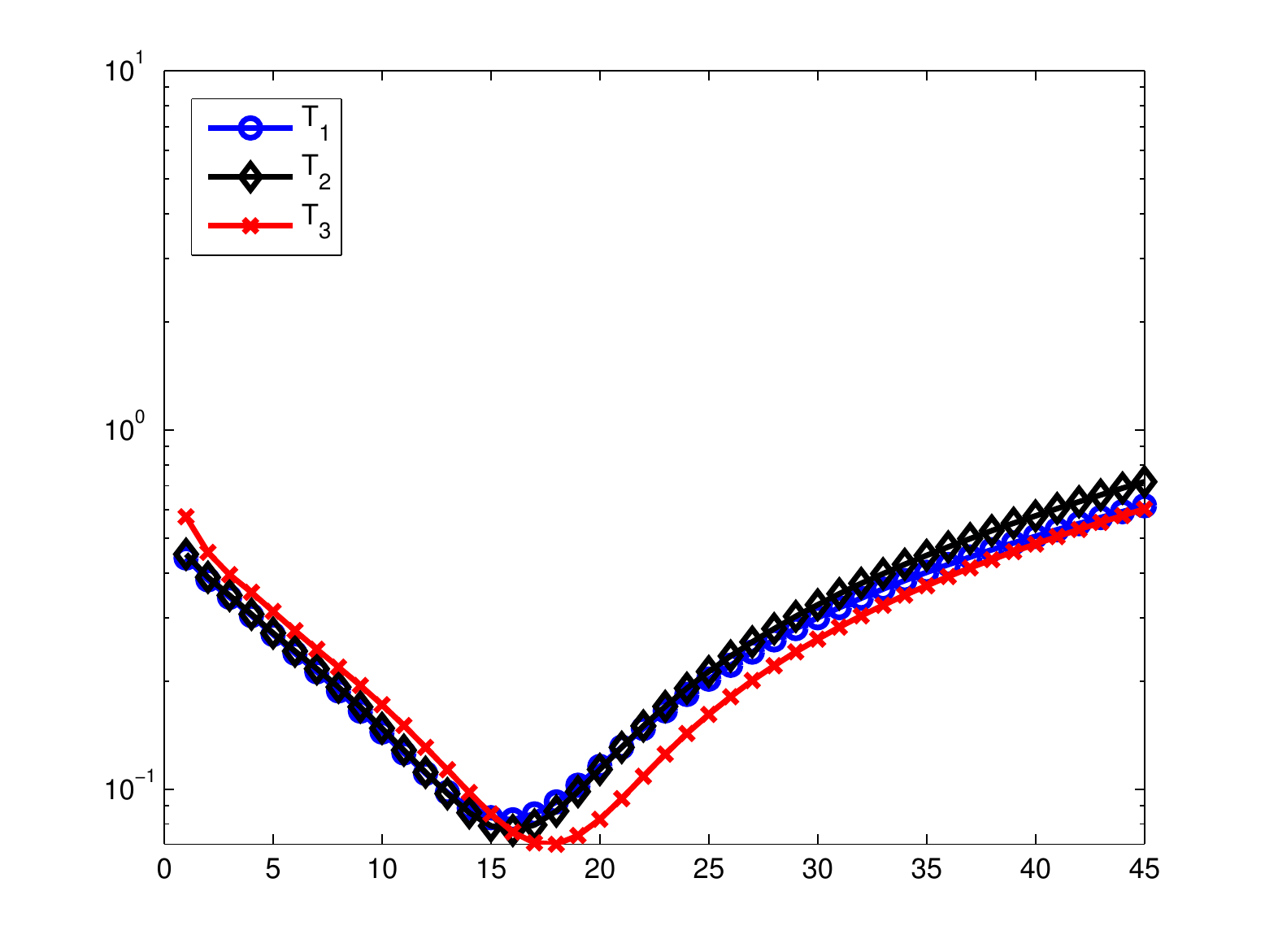}}
\subfloat[four measurements\label{fig:invcri4}]
 {\includegraphics[width=0.45\textwidth]{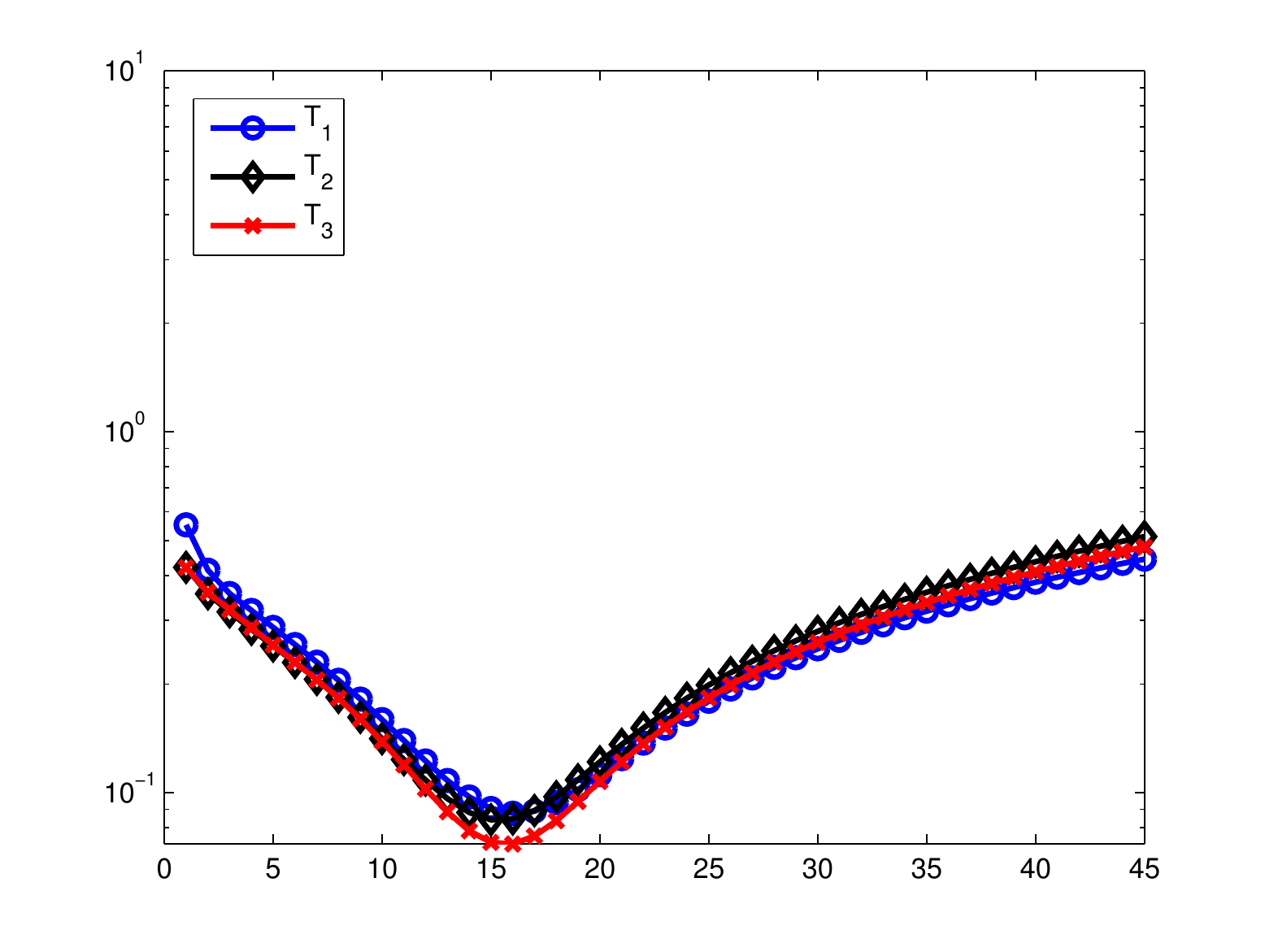}}\\
 \caption{The specific iterative relative error $\epsilon_f$ for three triangulation $T_1$, $T_2$, and $T_3$ \label{fig:invcri}}
\end{figure}

\subsection{Comparison of the hybrid method and the nonlinear optimization method}
\label{sec:4.2}
From section \ref{sec:4.1}, SIM remarkably linearly convergence before arriving minimum from \ref{fig:invcri2}, \ref{fig:invcri3} and \ref{fig:invcri4} of figure \ref{fig:invcri}. Using this feature of SIM, hybrid method is expected to improve the stability of SIM. Considering one, two, three, and four measurements, we apply hybrid and nonlinear optimization method respectively on noise-free, 2$\%$ noise and 5$\%$ noise data. For two templates illustrated in figure \ref{fig:original}, their reconstruction results are showed in figure \ref{fig:noise-free}, \ref{fig:noise2} and \ref{fig:noise5}, the specific relative error $\epsilon_f$ are showed in figure \ref{fig:template 1} and \ref{fig:template 2}, and their relative error are listed in the table \ref{tab:1} after 50 steps. 

From figure\ref{fig:noise-free}, \ref{fig:noise2} and \ref{fig:noise5}, hybrid method performs better in one-measurement case. Even for 5$\%$ noise data, optimization method in one-measurement can only obtain a figure almost without any edges, see the first figure on the fourth row of figure \ref{fig:noise5}. Similarly, from the figure \ref{fig:template 1} and \ref{fig:template 2}, in one-measurement case, hybrid method gets smaller relative error. Even for noise-free data, in one-measurement case, optimization can not control the relative error to less than 10$\%$. In more measurements cases such as three-measurement or four-measurement, the two methods almost can get the same accuracy, but the hybrid method converges more rapidly in most cases. From table \ref{tab:1}, we can see that no matter which method is used, the more measurements, the smaller the reconstruction error. When the number of measurement is small, hybrid method is more advantageous.

\begin{figure}[htpb]
    \centering
\begin{minipage}{0.24\linewidth}
\includegraphics[width=\textwidth]{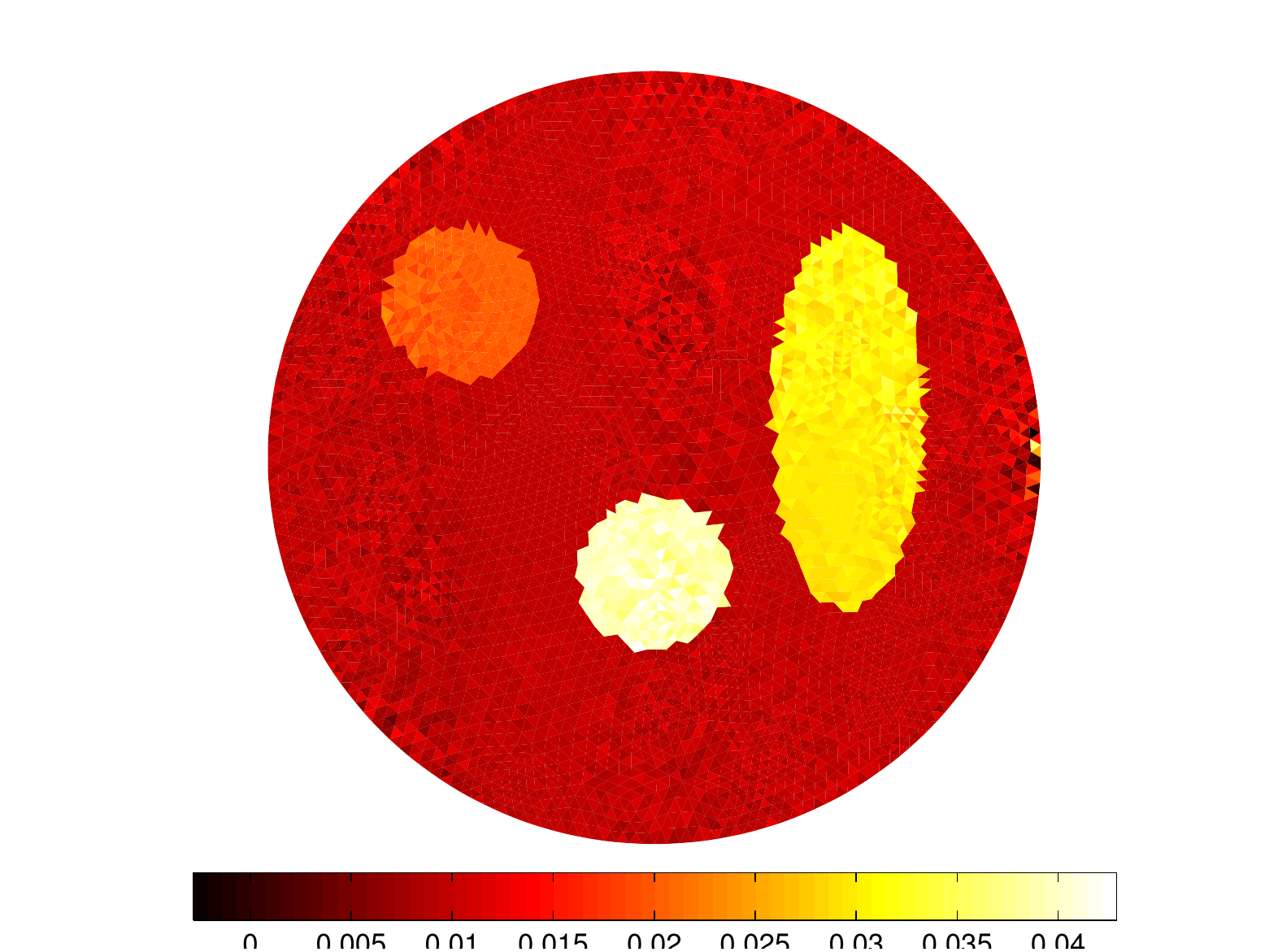}
\end{minipage}
\begin{minipage}{0.24\linewidth}
  \includegraphics[width=\textwidth]{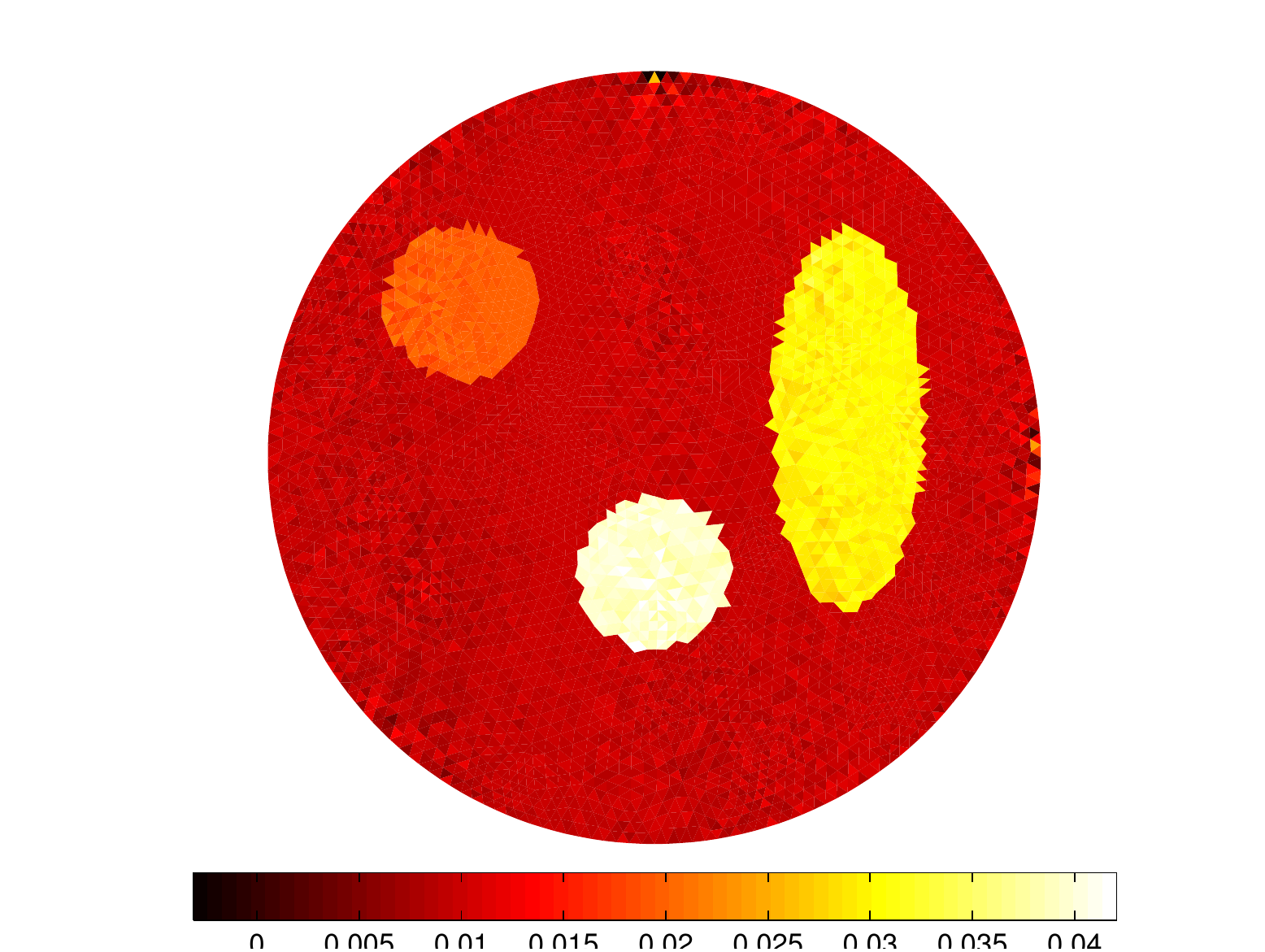}
\end{minipage}
\begin{minipage}{0.24\linewidth}
  \includegraphics[width=\textwidth]{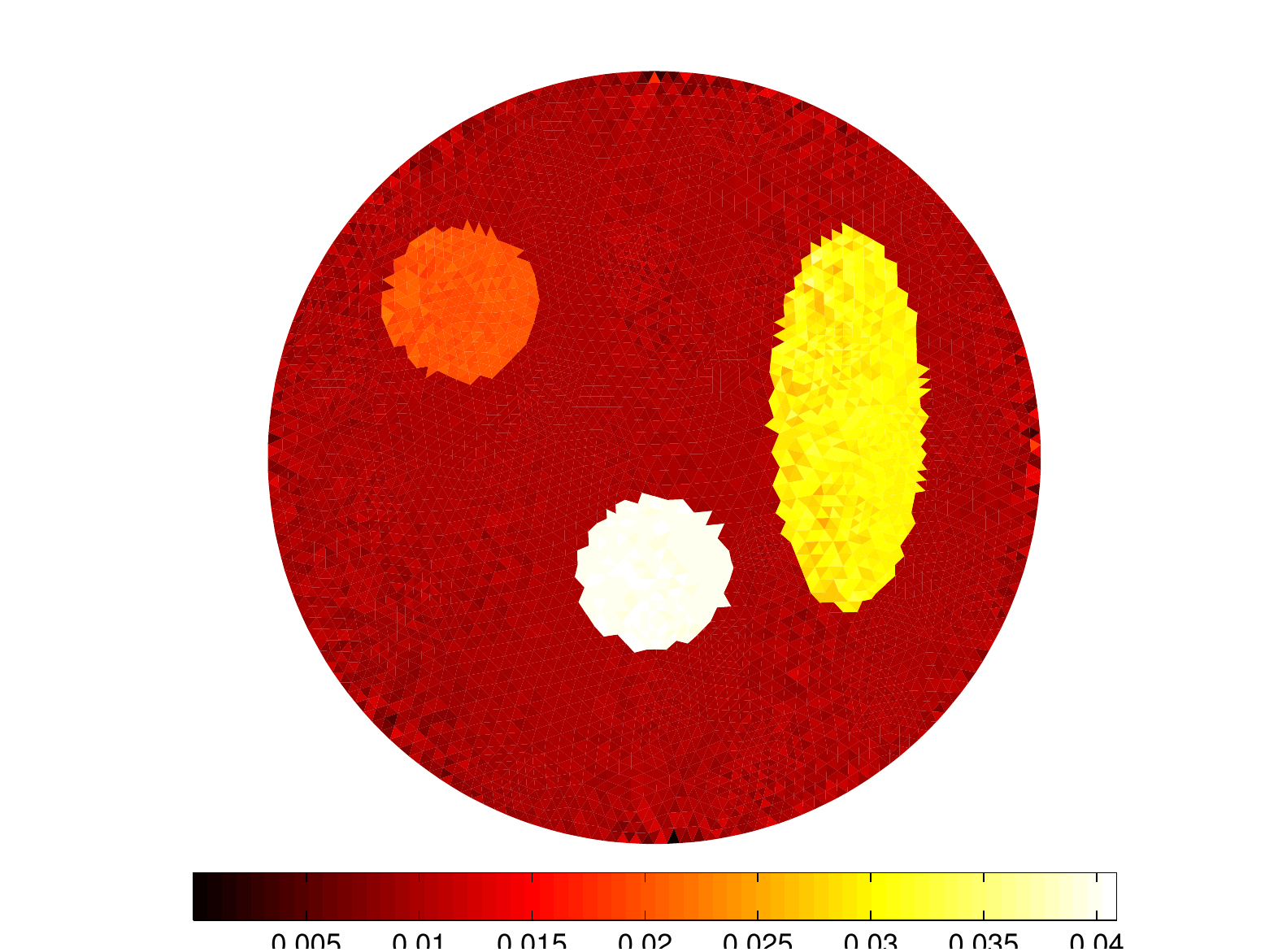}
\end{minipage}
\begin{minipage}{0.24\linewidth}
  \includegraphics[width=\textwidth]{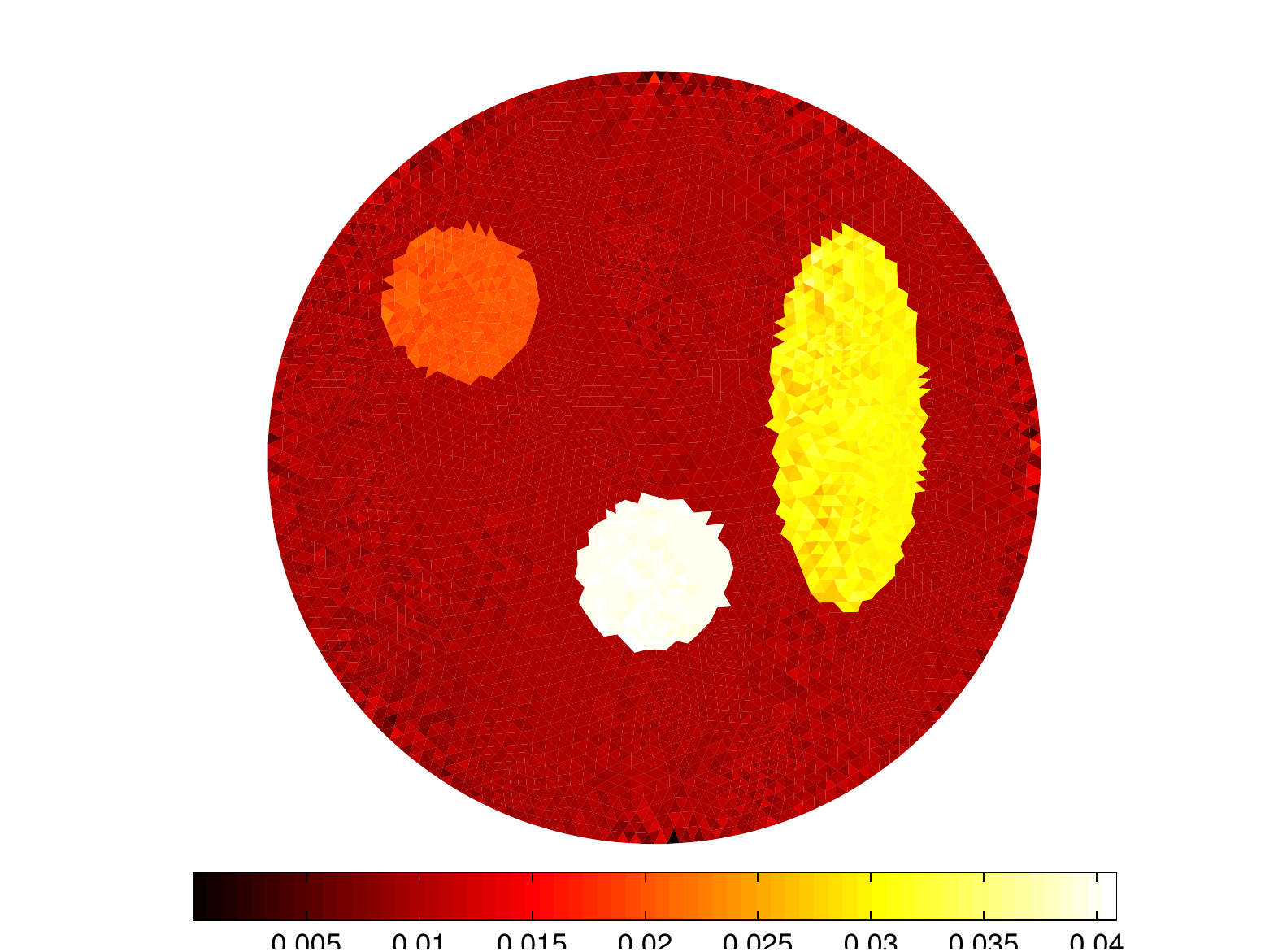}
\end{minipage}
\\
\begin{minipage}{0.24\linewidth}
\includegraphics[width=\textwidth]{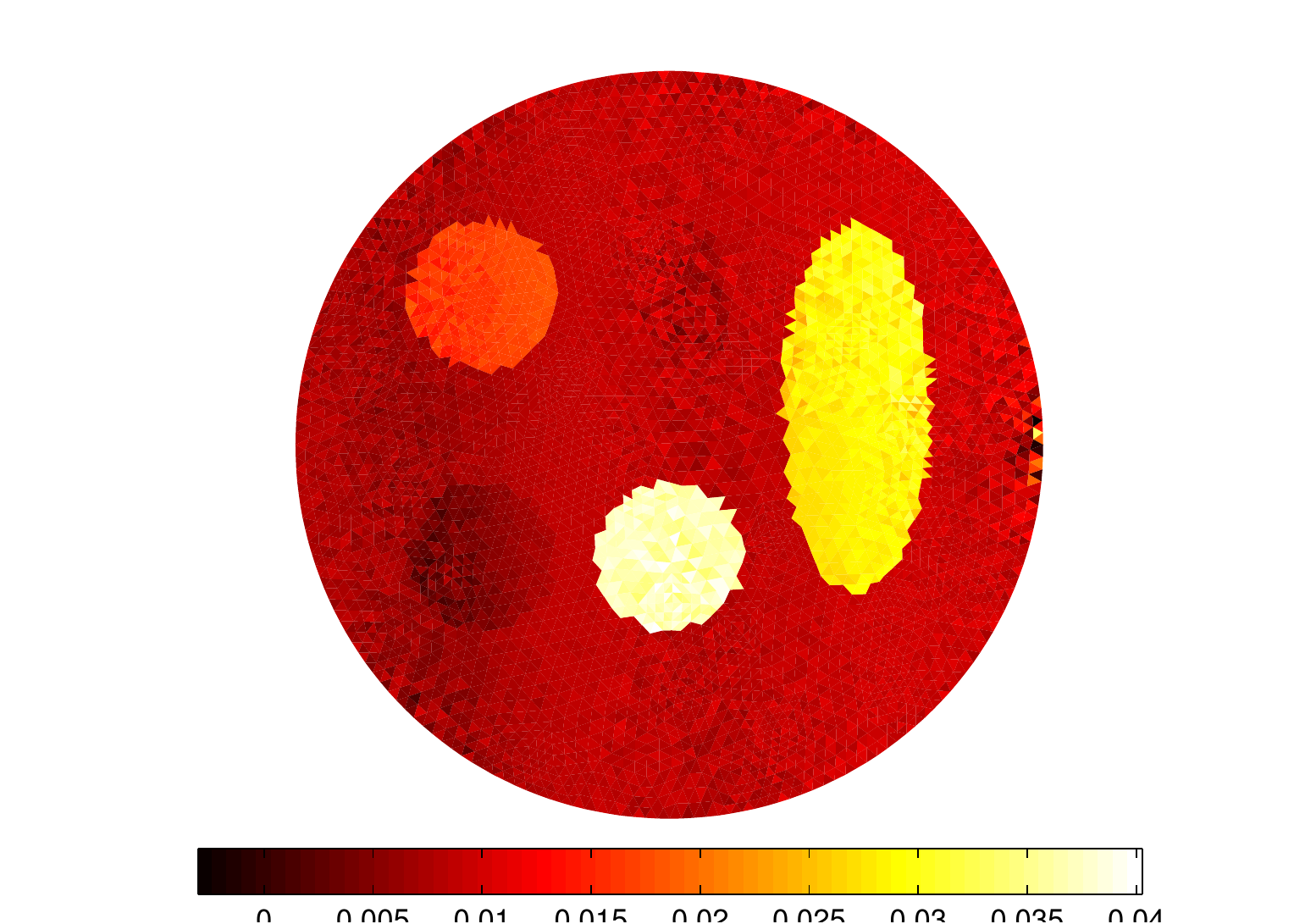}
\end{minipage}
\begin{minipage}{0.24\linewidth}
  \includegraphics[width=\textwidth]{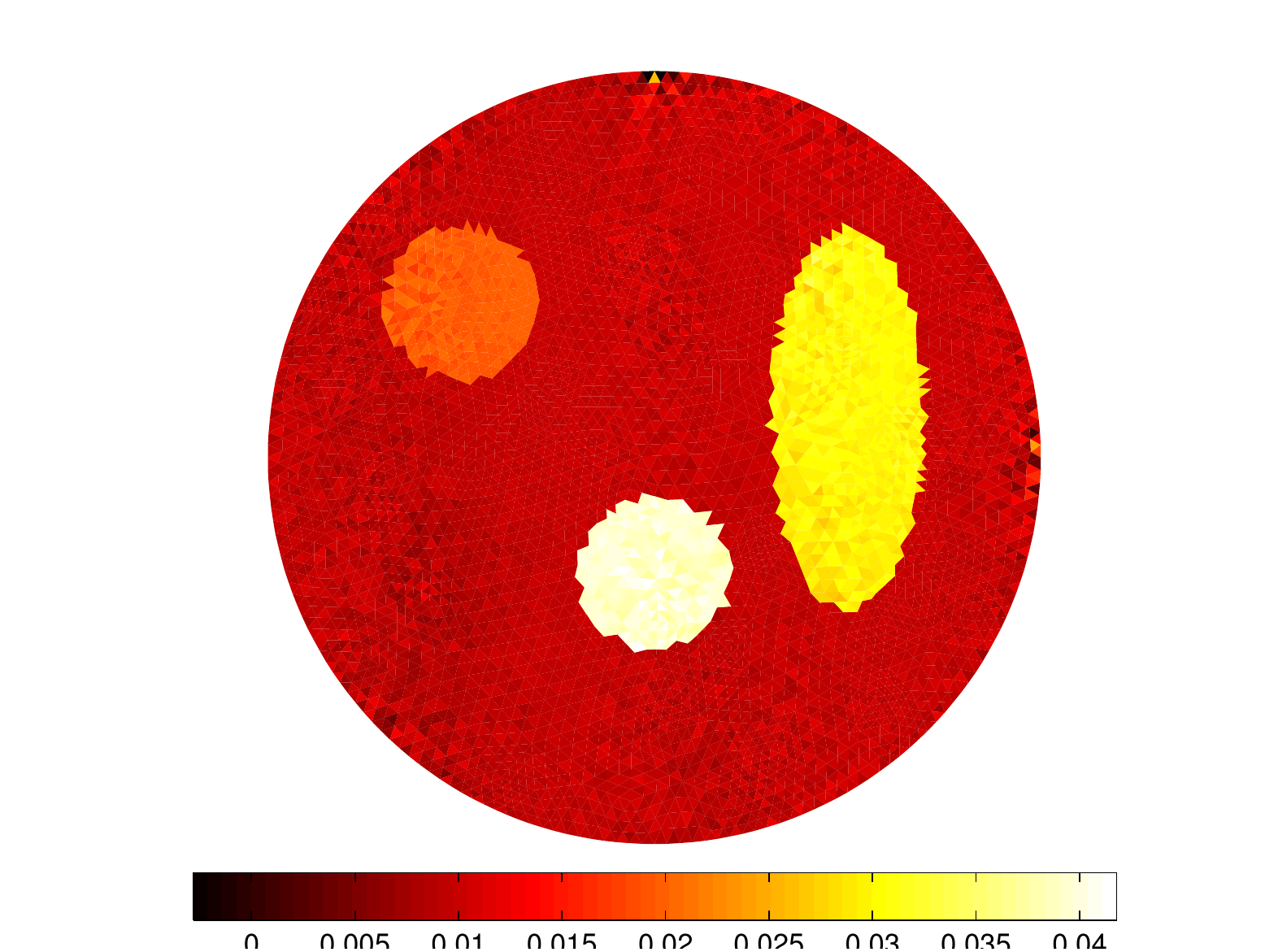}
\end{minipage}
\begin{minipage}{0.24\linewidth}
  \includegraphics[width=\textwidth]{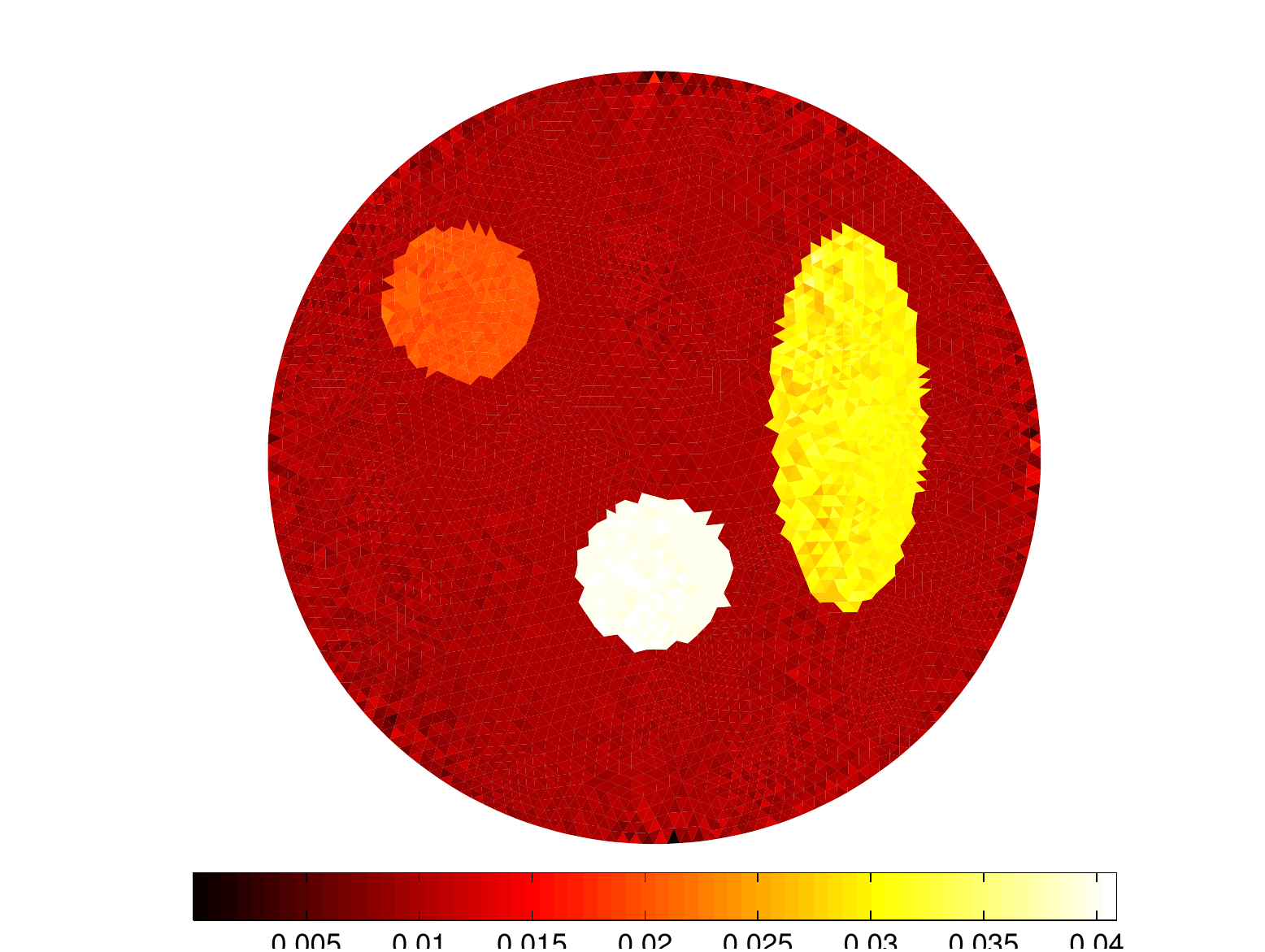}
\end{minipage}
\begin{minipage}{0.24\linewidth}
  \includegraphics[width=\textwidth]{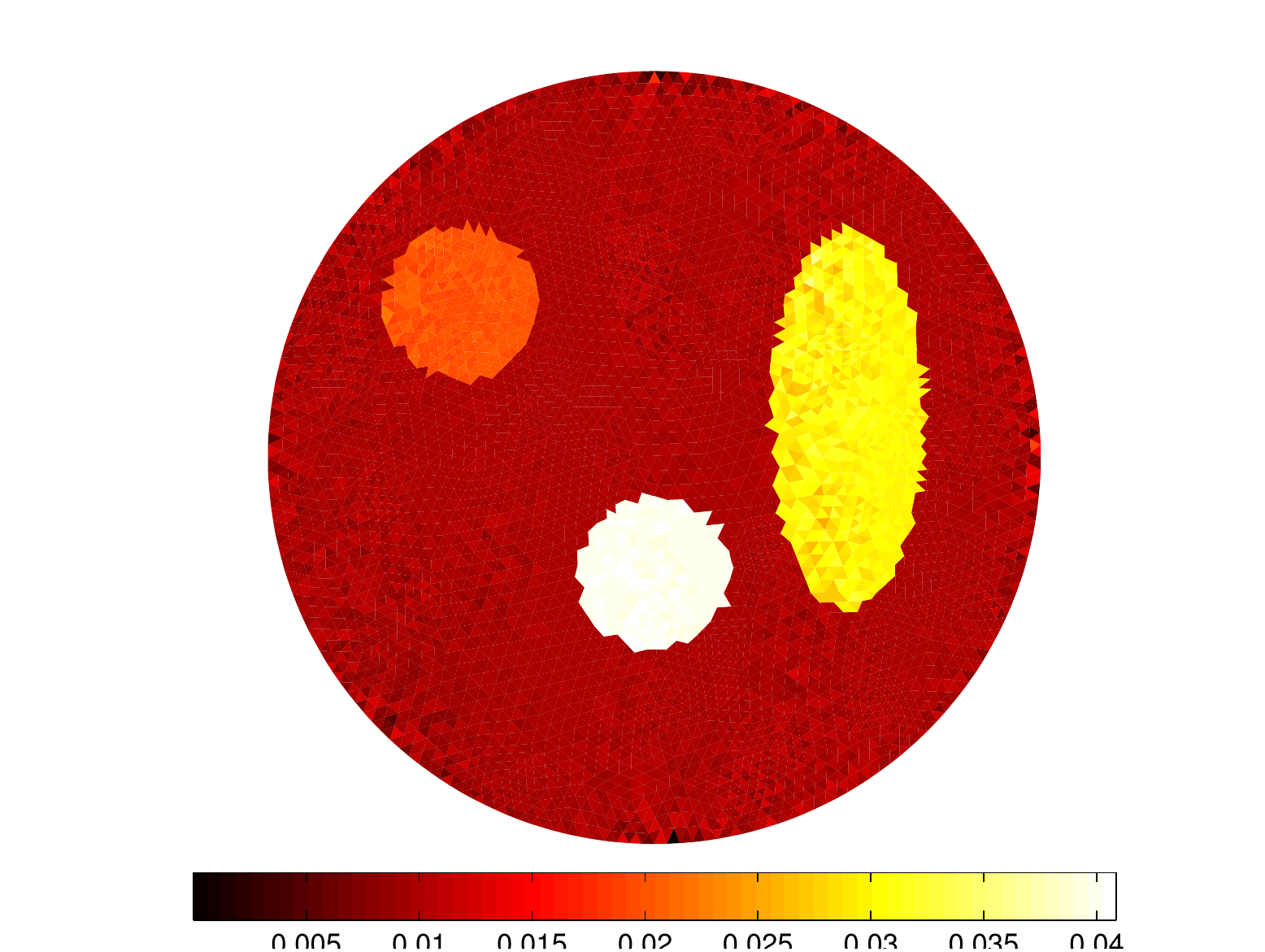}
\end{minipage}
\\
\begin{minipage}{0.24\linewidth}
\includegraphics[width=\textwidth]{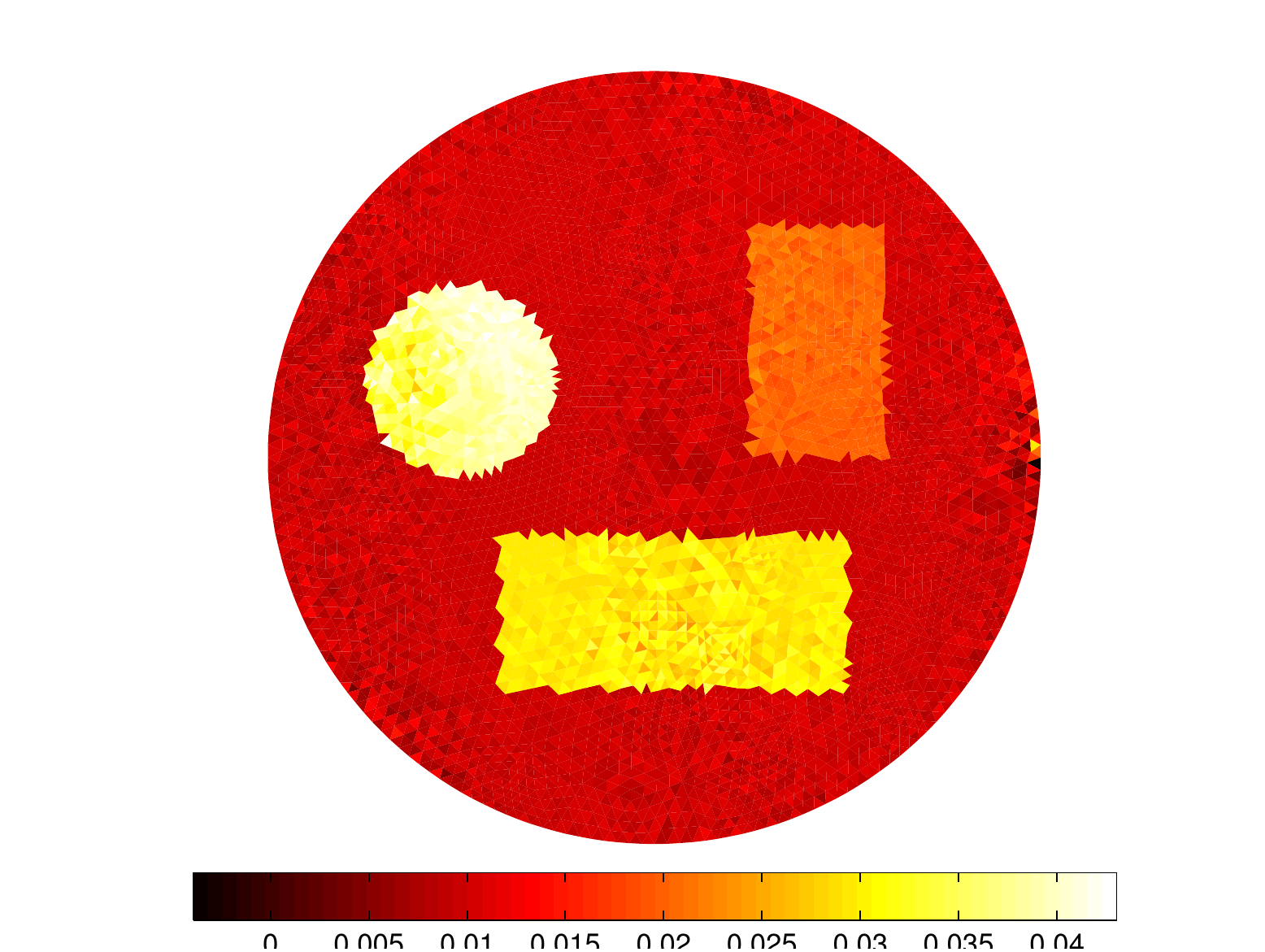}
\end{minipage}
\begin{minipage}{0.24\linewidth}
  \includegraphics[width=\textwidth]{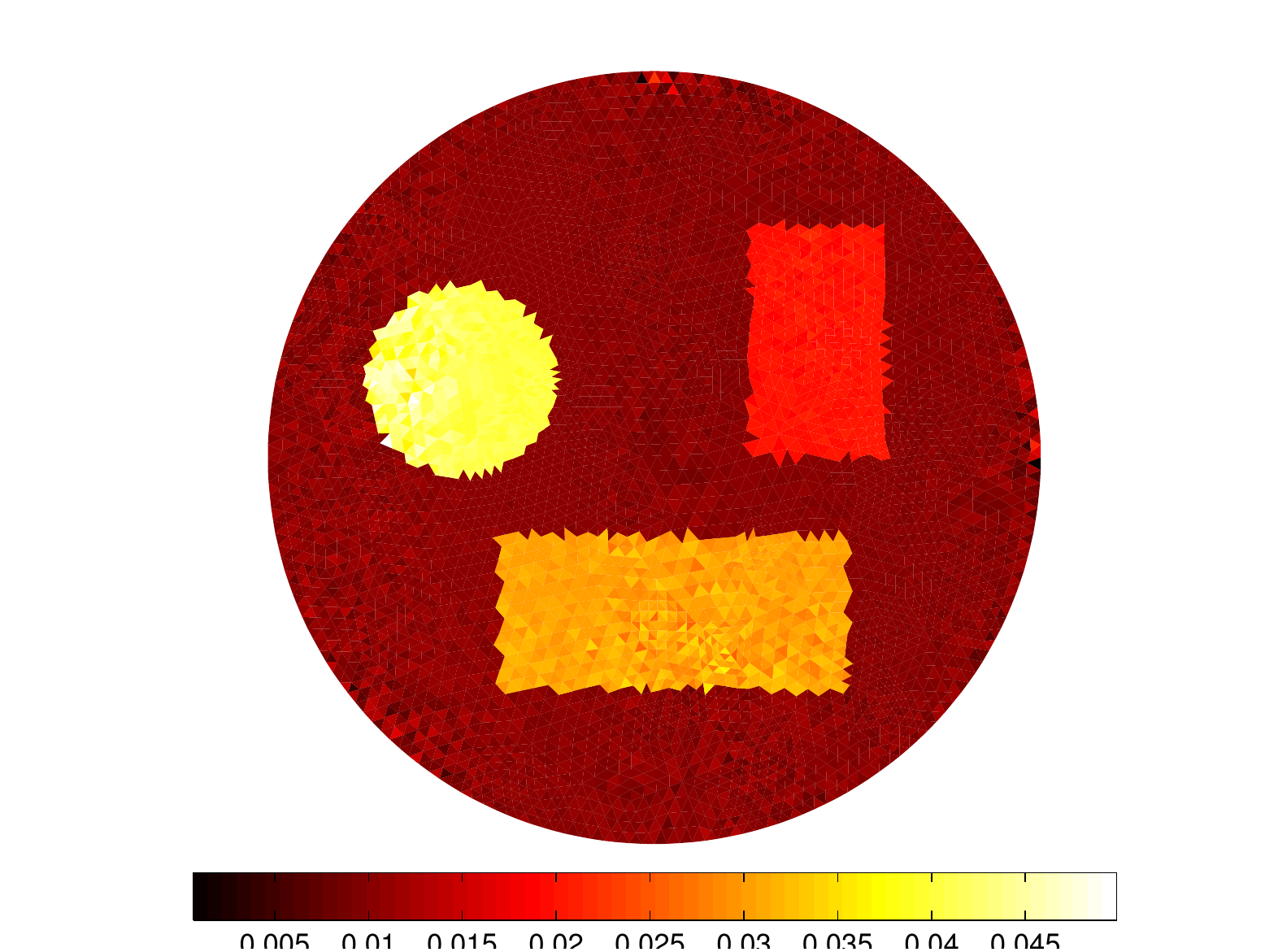}
\end{minipage}
\begin{minipage}{0.24\linewidth}
  \includegraphics[width=\textwidth]{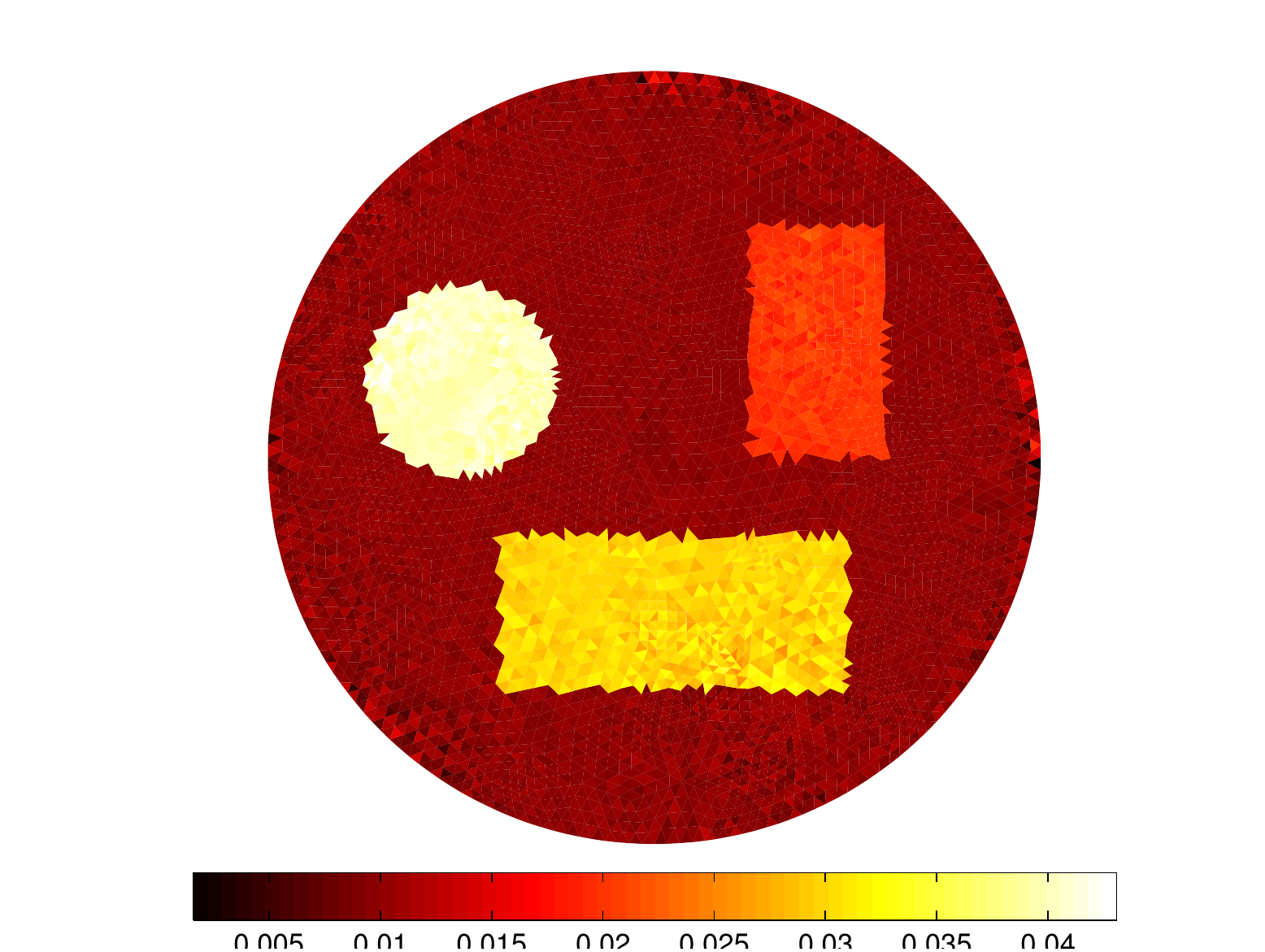}
\end{minipage}
\begin{minipage}{0.24\linewidth}
  \includegraphics[width=\textwidth]{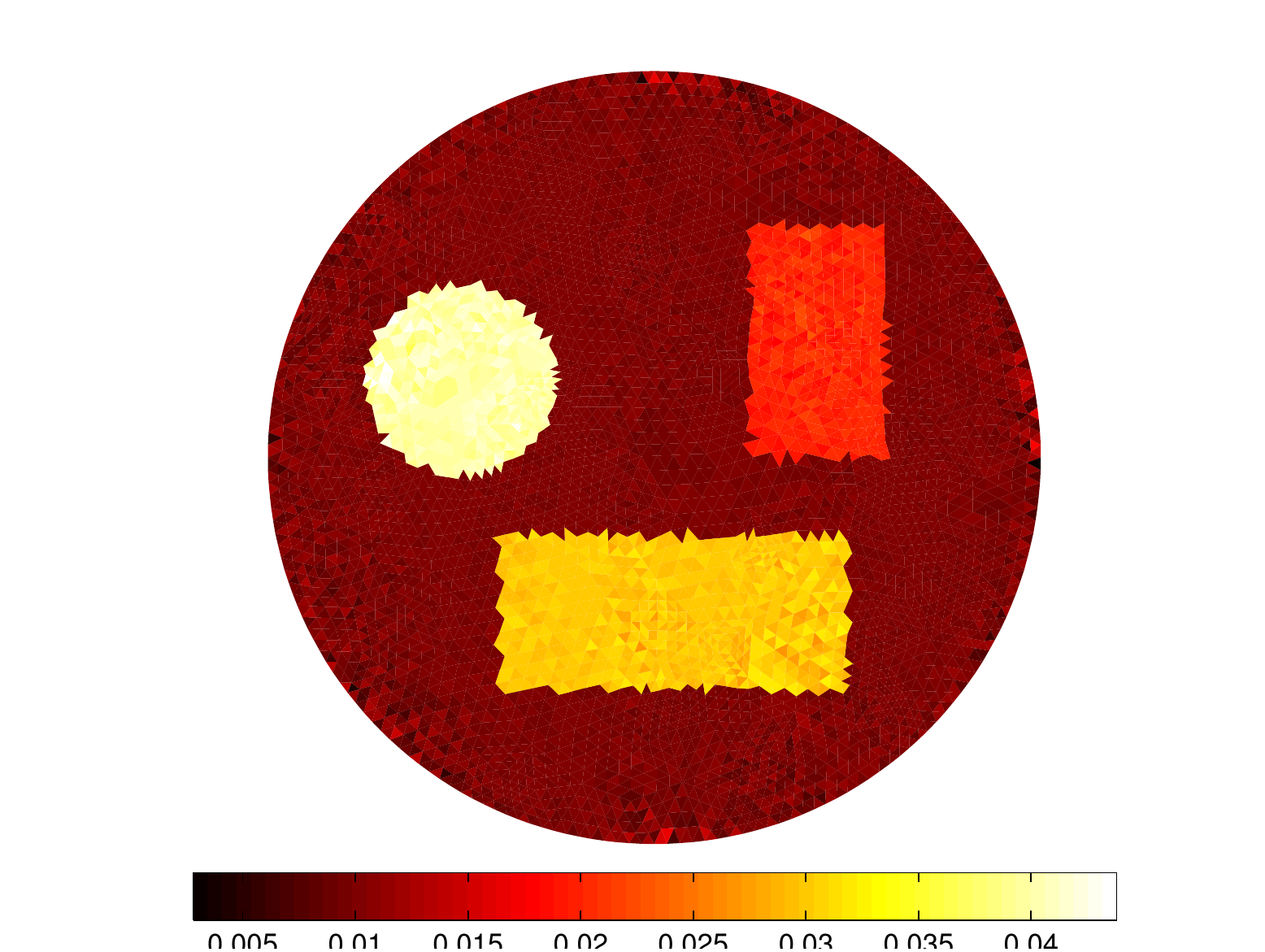}
\end{minipage}
\\
\begin{minipage}{0.24\linewidth}
\includegraphics[width=\textwidth]{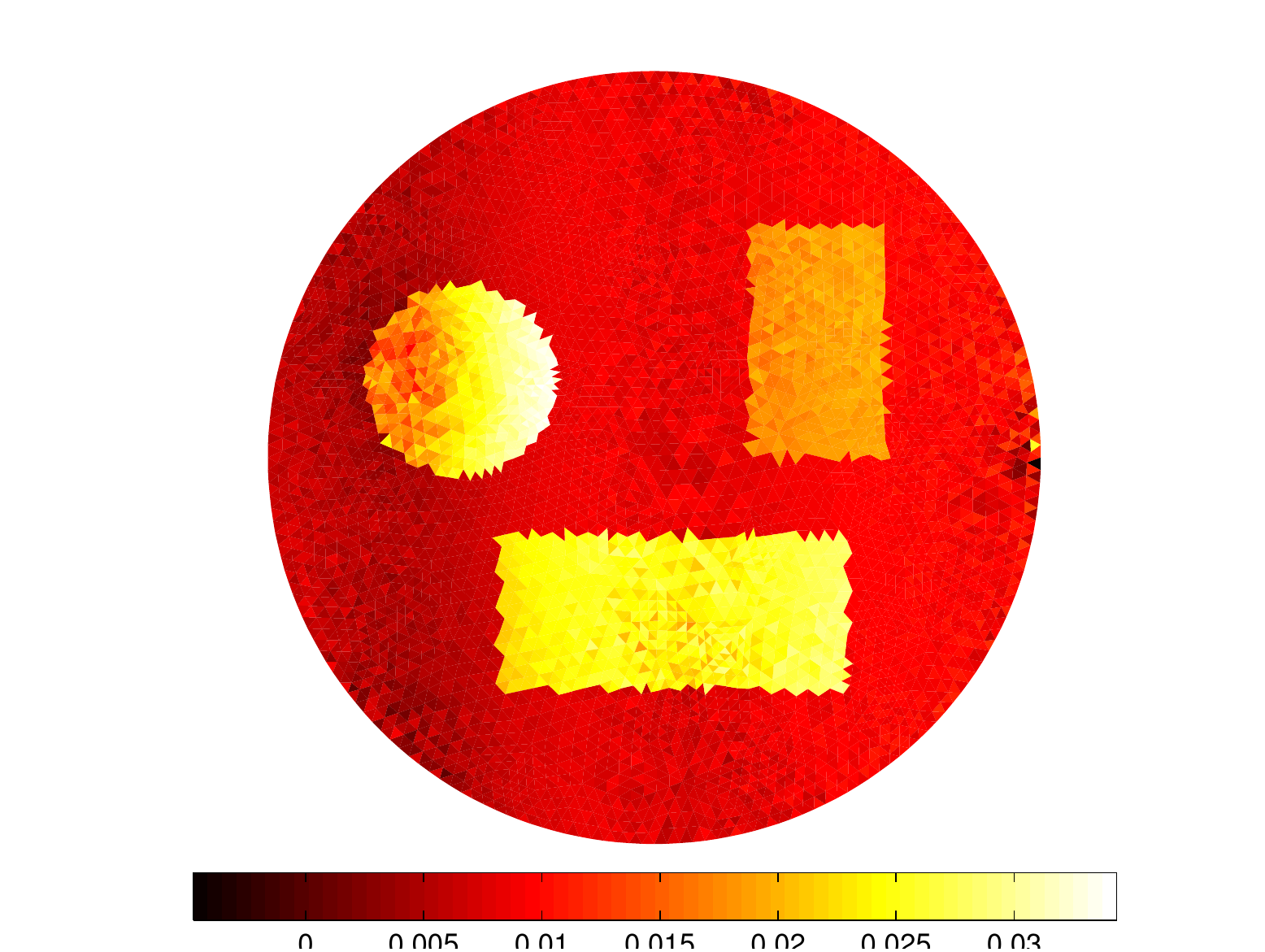}
\end{minipage}
\begin{minipage}{0.24\linewidth}
  \includegraphics[width=\textwidth]{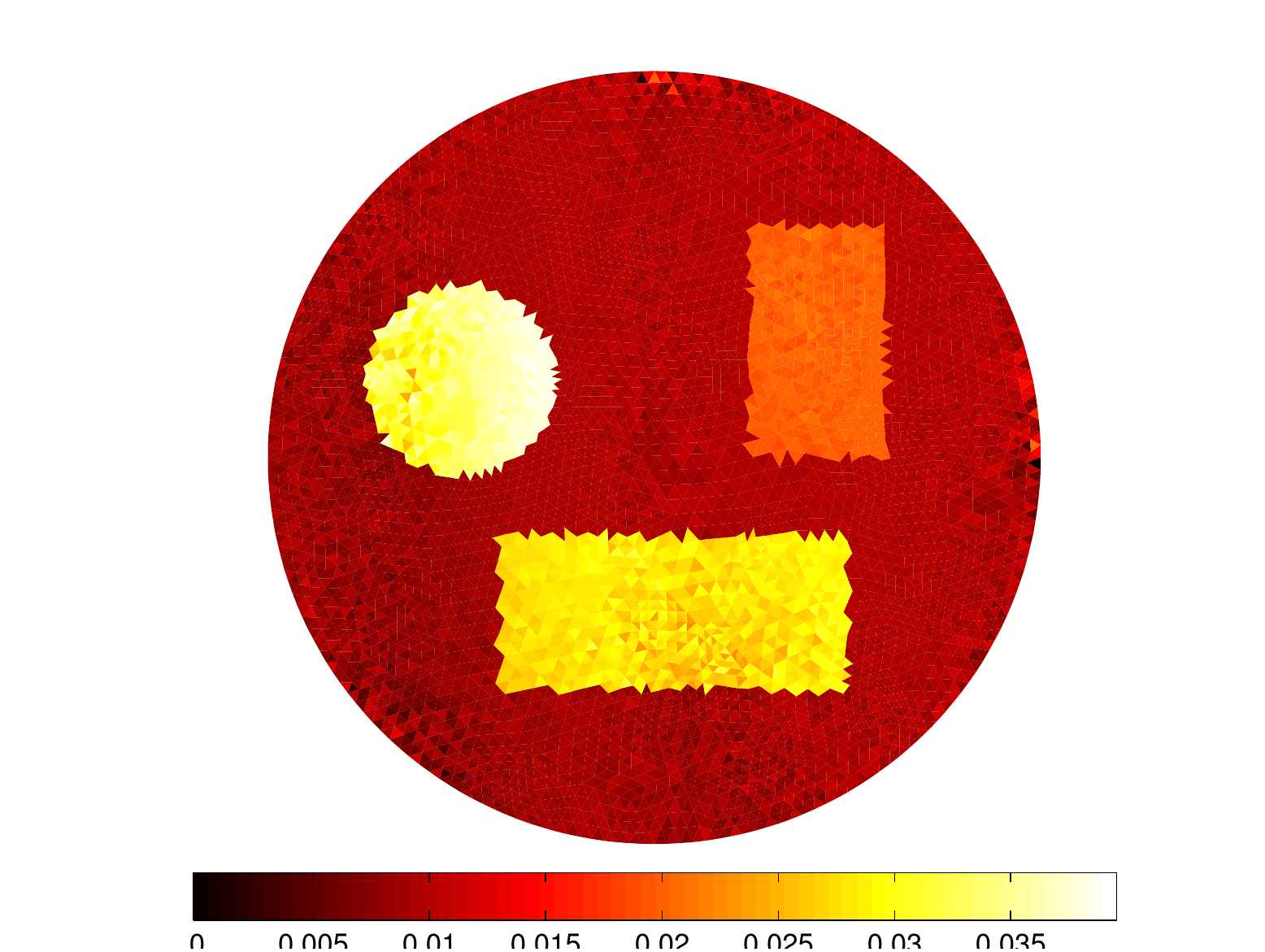}
\end{minipage}
\begin{minipage}{0.24\linewidth}
  \includegraphics[width=\textwidth]{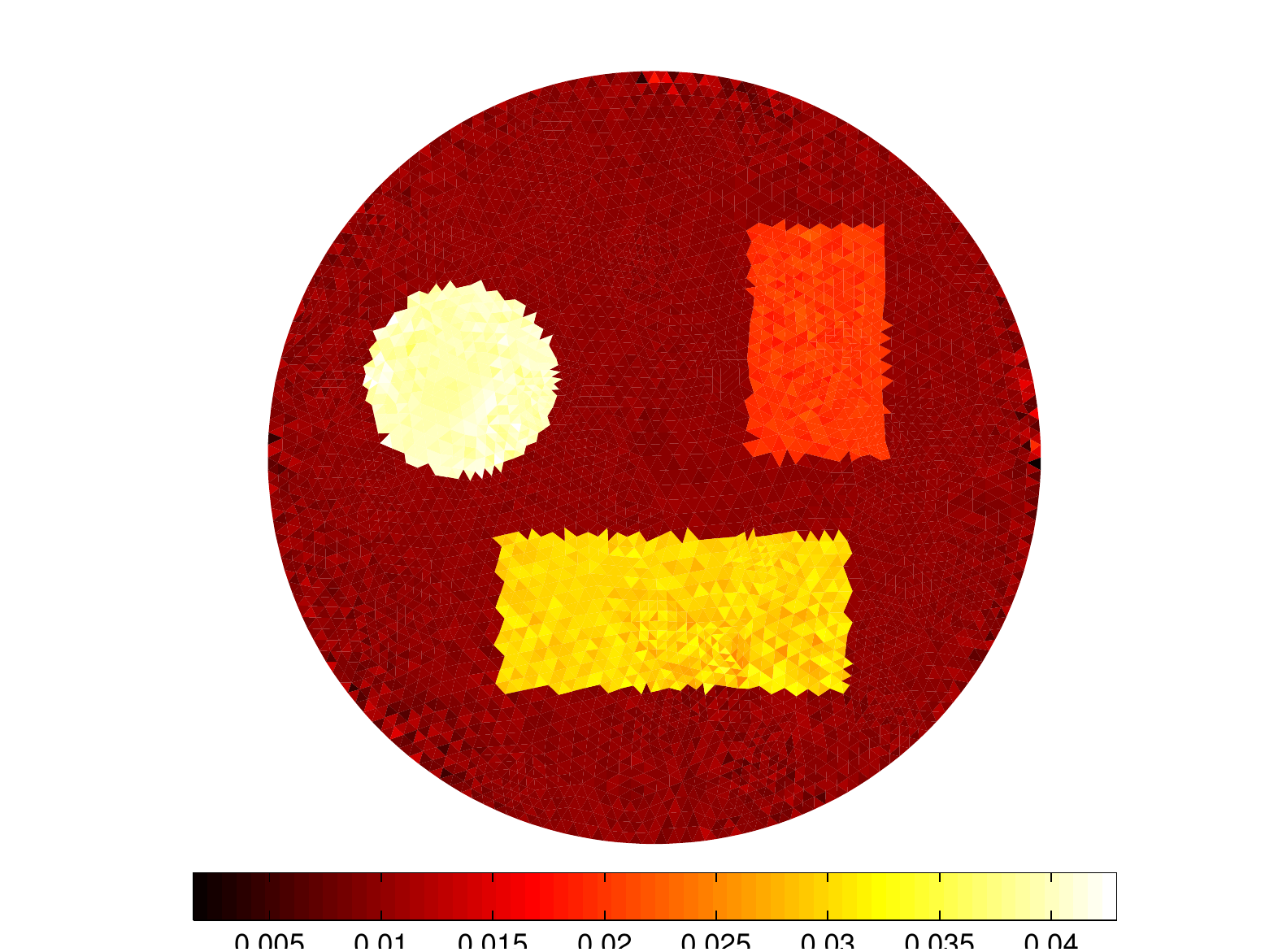}
\end{minipage}
\begin{minipage}{0.24\linewidth}
  \includegraphics[width=\textwidth]{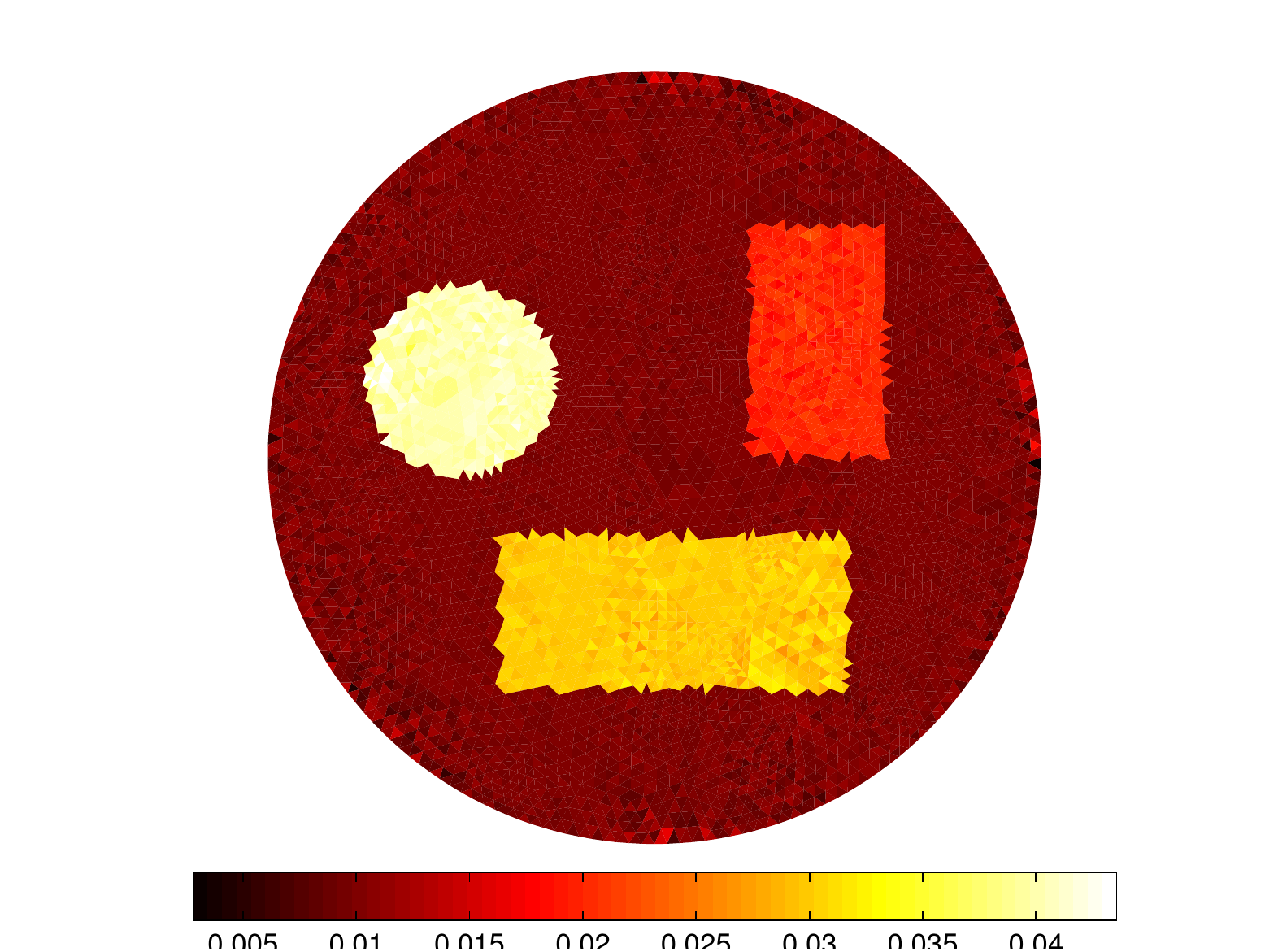}
\end{minipage}
 \caption{\label{fig:noise-free}Reconstruction of $\mu_{a,xf}$ for noise-free data by the hybrid method and the nonlinear optimization method. First, second row: first template. Third, fourth row: second template. First, third row: hybrid method. Second, fourth row: nonlinear optimization method. First, second, third, and fourth column: one-measurement, two-measurement, three-measurement and four-measurement.}
\end{figure}

\begin{figure}[htpb]
    \centering
\begin{minipage}{0.24\linewidth}
\includegraphics[width=\textwidth]{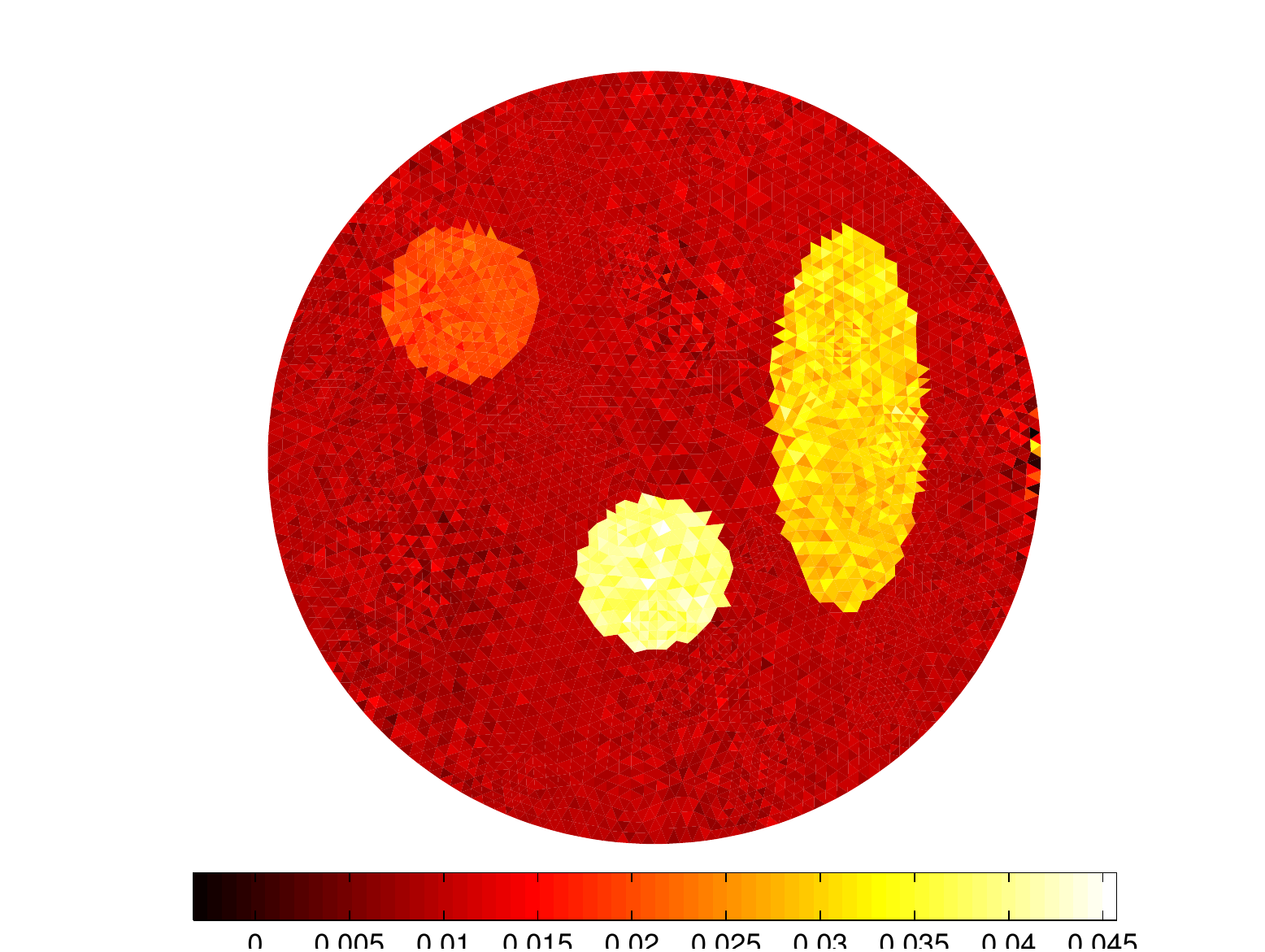}
\end{minipage}
\begin{minipage}{0.24\linewidth}
  \includegraphics[width=\textwidth]{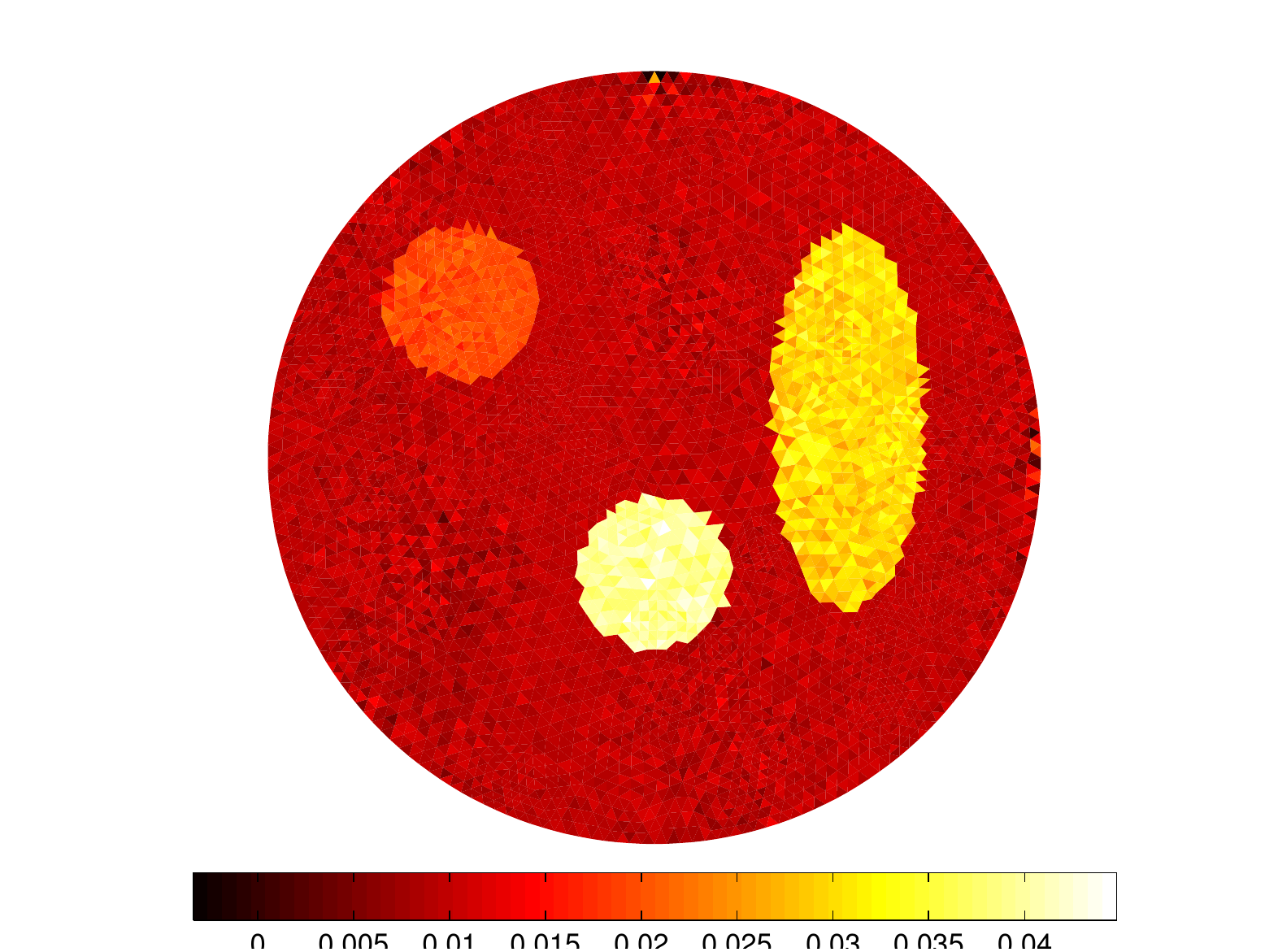}
\end{minipage}
\begin{minipage}{0.24\linewidth}
  \includegraphics[width=\textwidth]{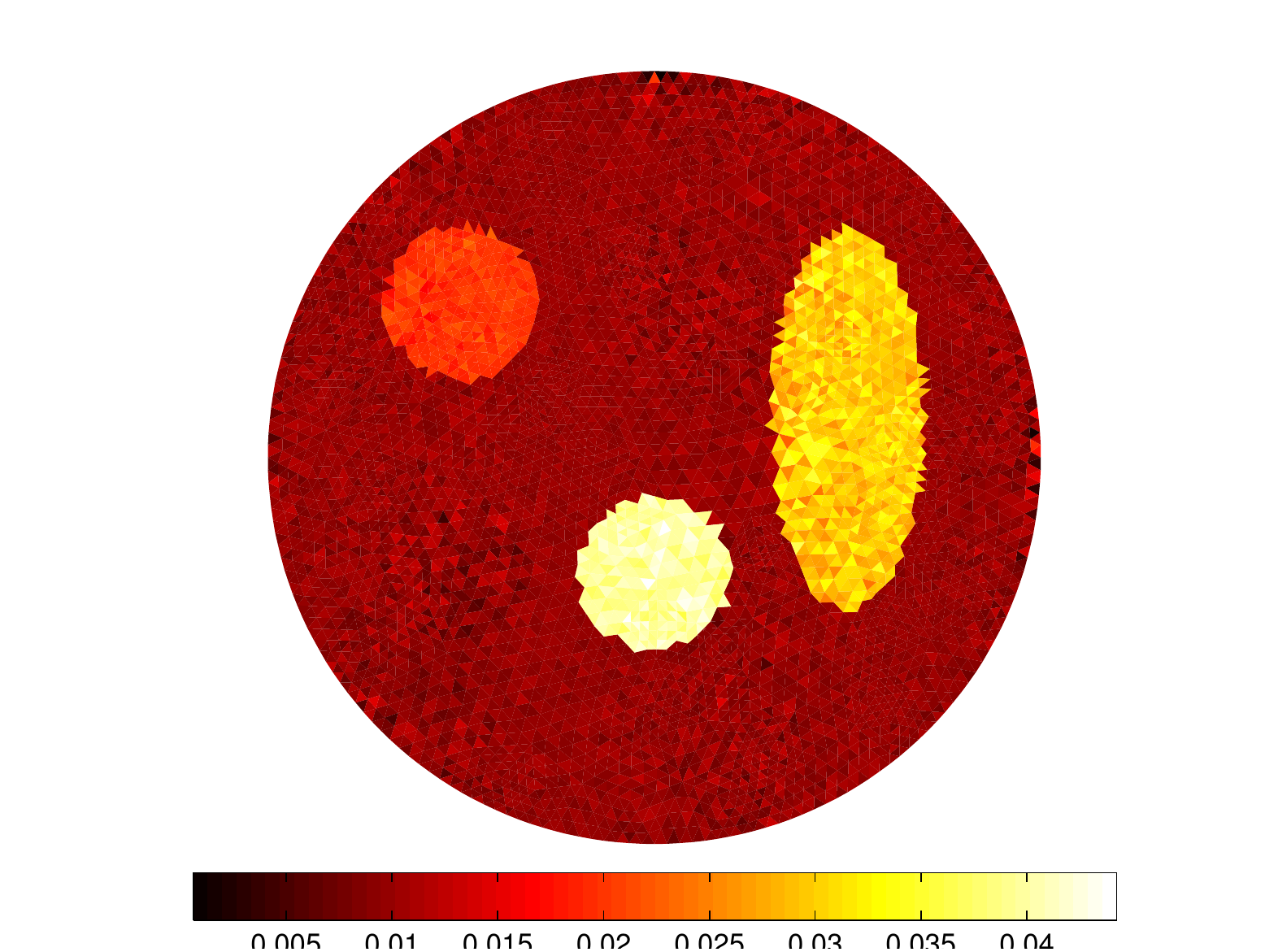}
\end{minipage}
\begin{minipage}{0.24\linewidth}
  \includegraphics[width=\textwidth]{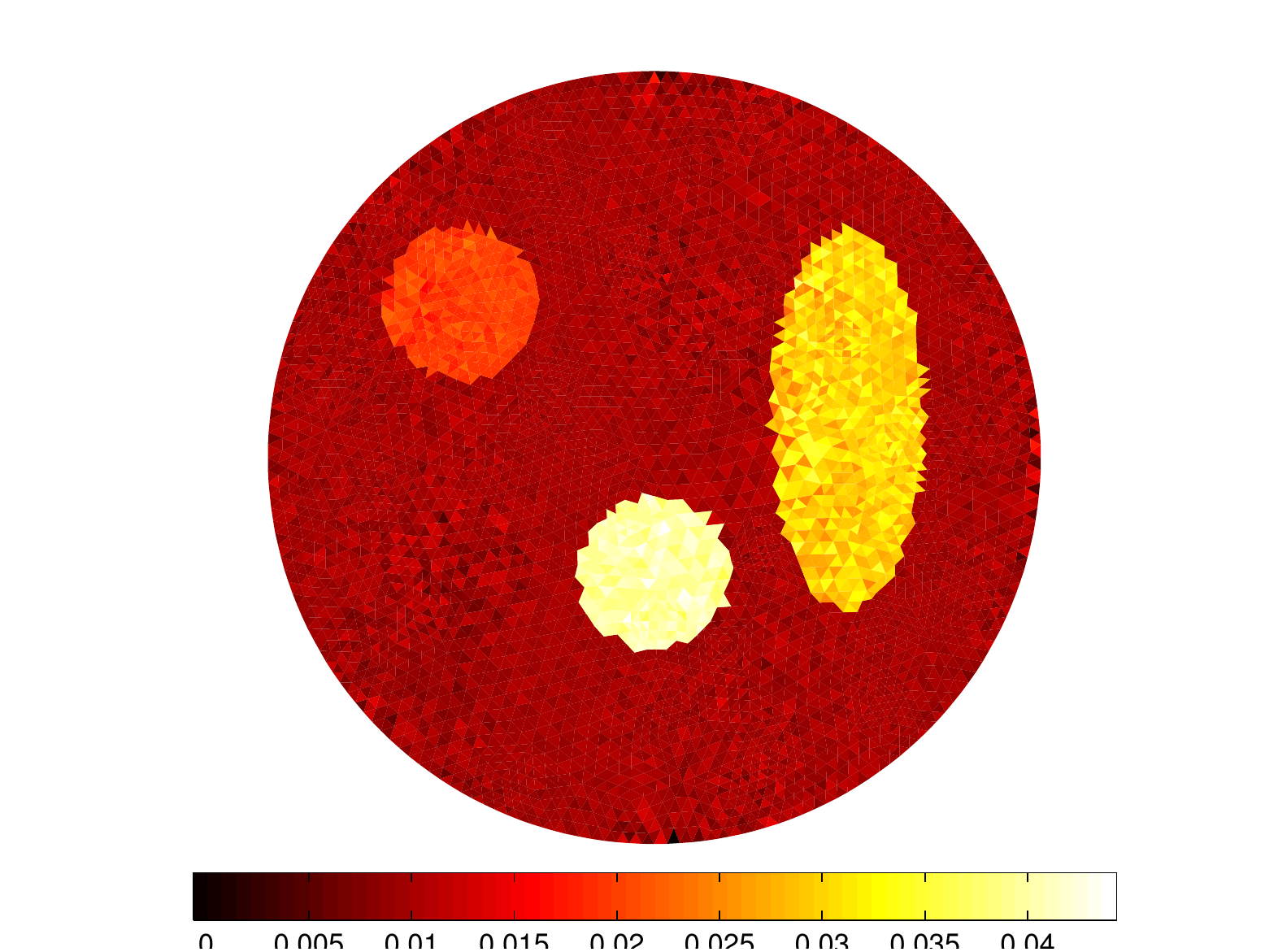}
\end{minipage}
\\
\begin{minipage}{0.24\linewidth}
\includegraphics[width=\textwidth]{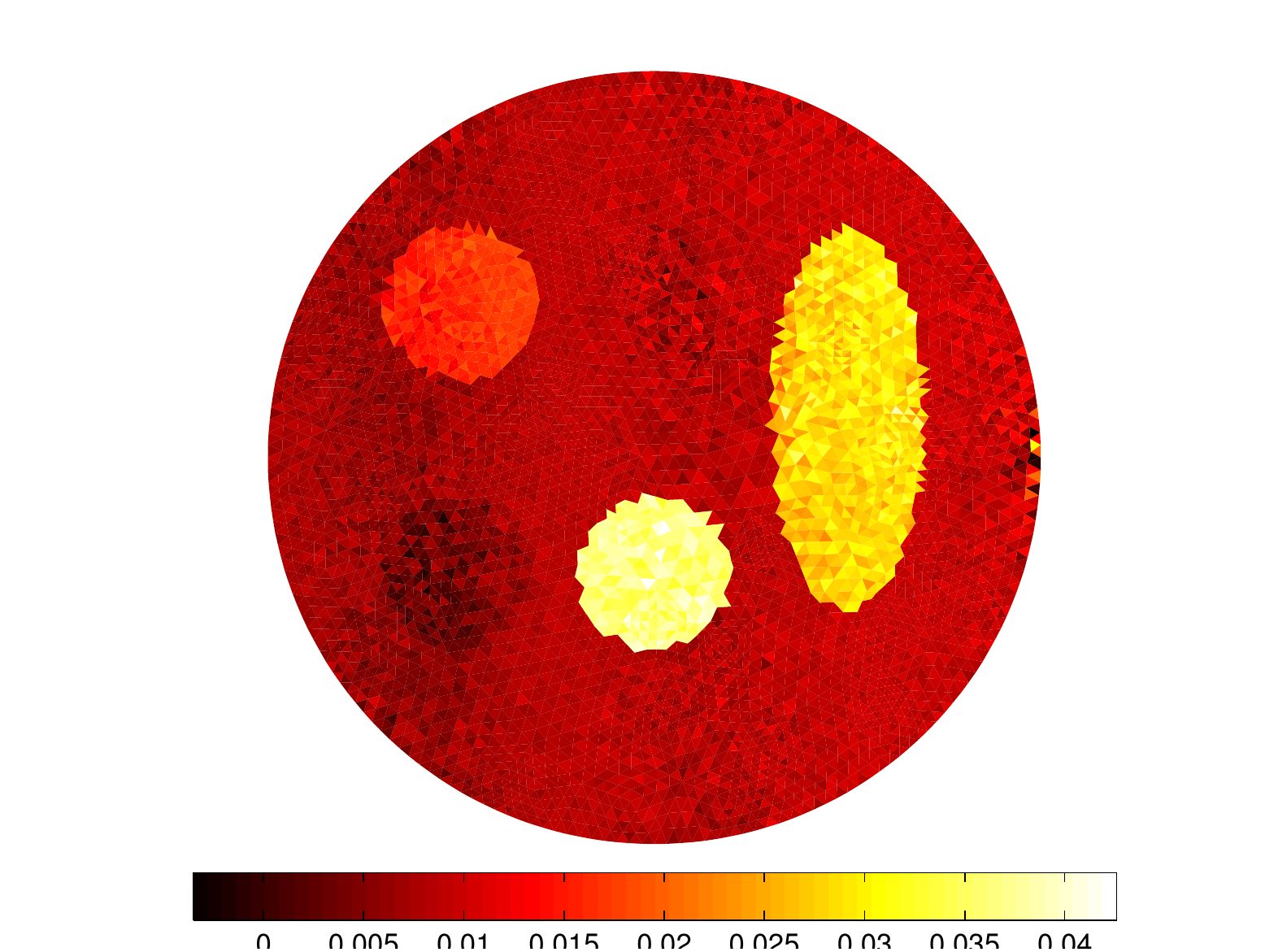}
\end{minipage}
\begin{minipage}{0.24\linewidth}
  \includegraphics[width=\textwidth]{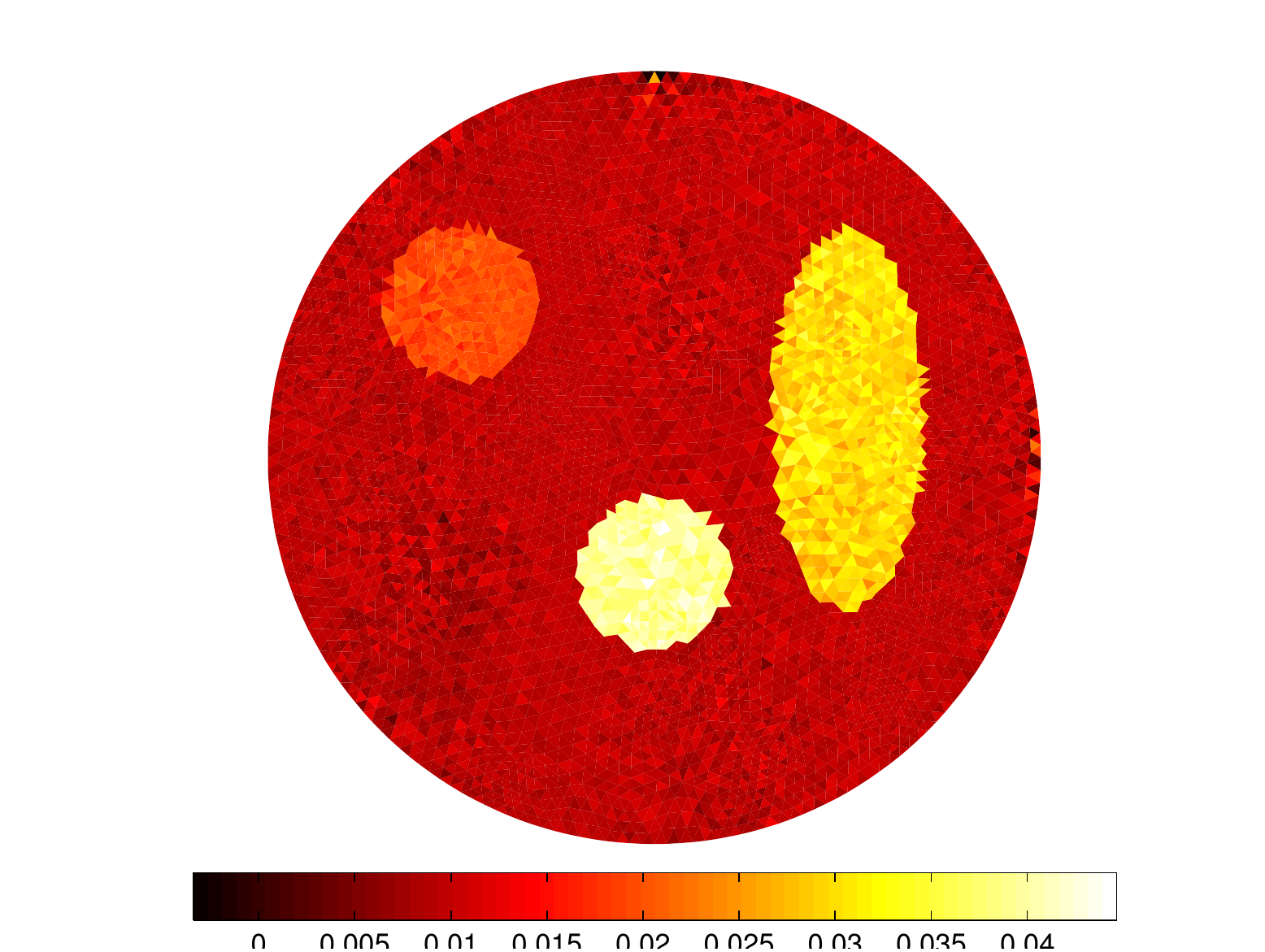}
\end{minipage}
\begin{minipage}{0.24\linewidth}
  \includegraphics[width=\textwidth]{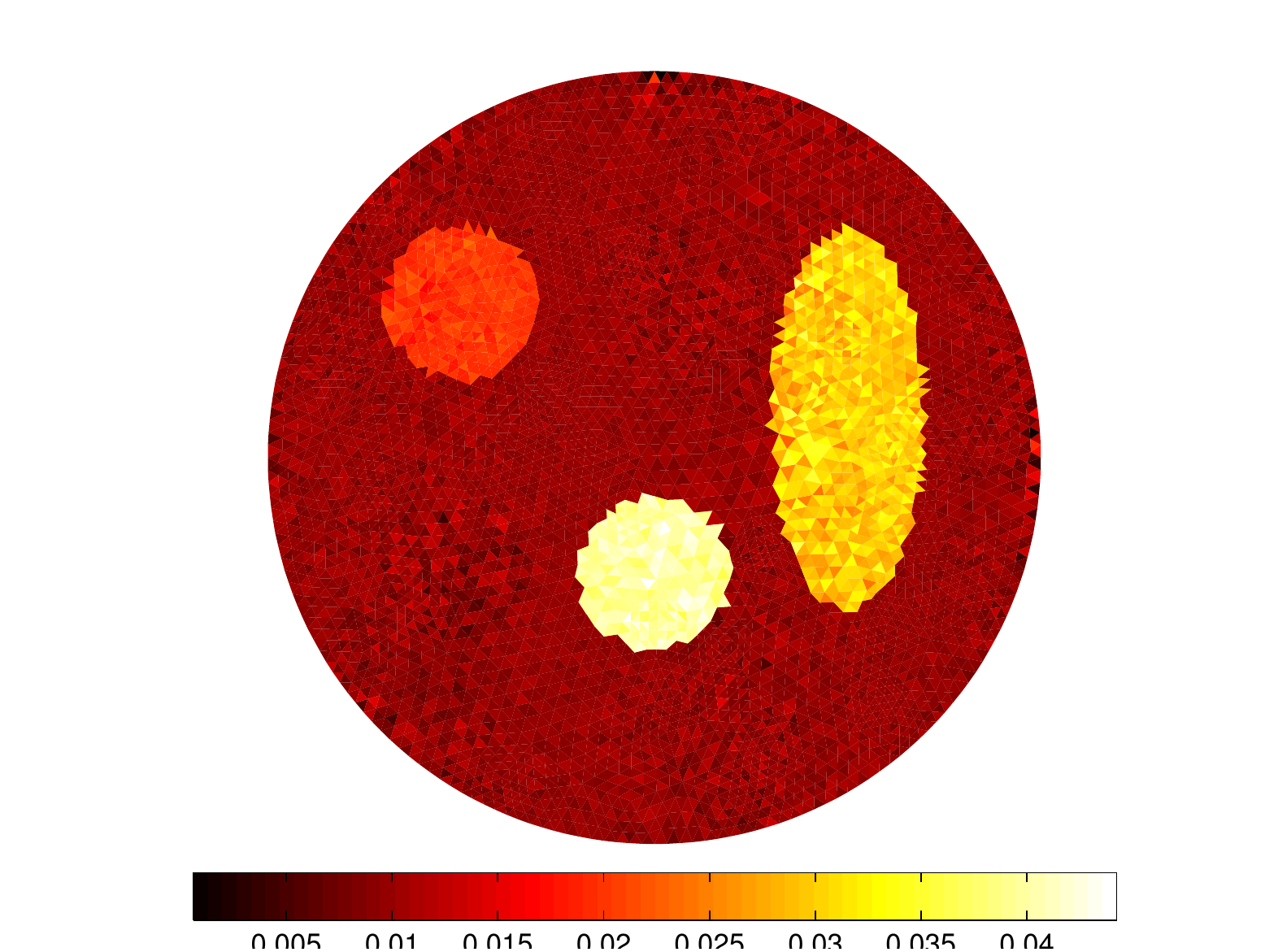}
\end{minipage}
\begin{minipage}{0.24\linewidth}
  \includegraphics[width=\textwidth]{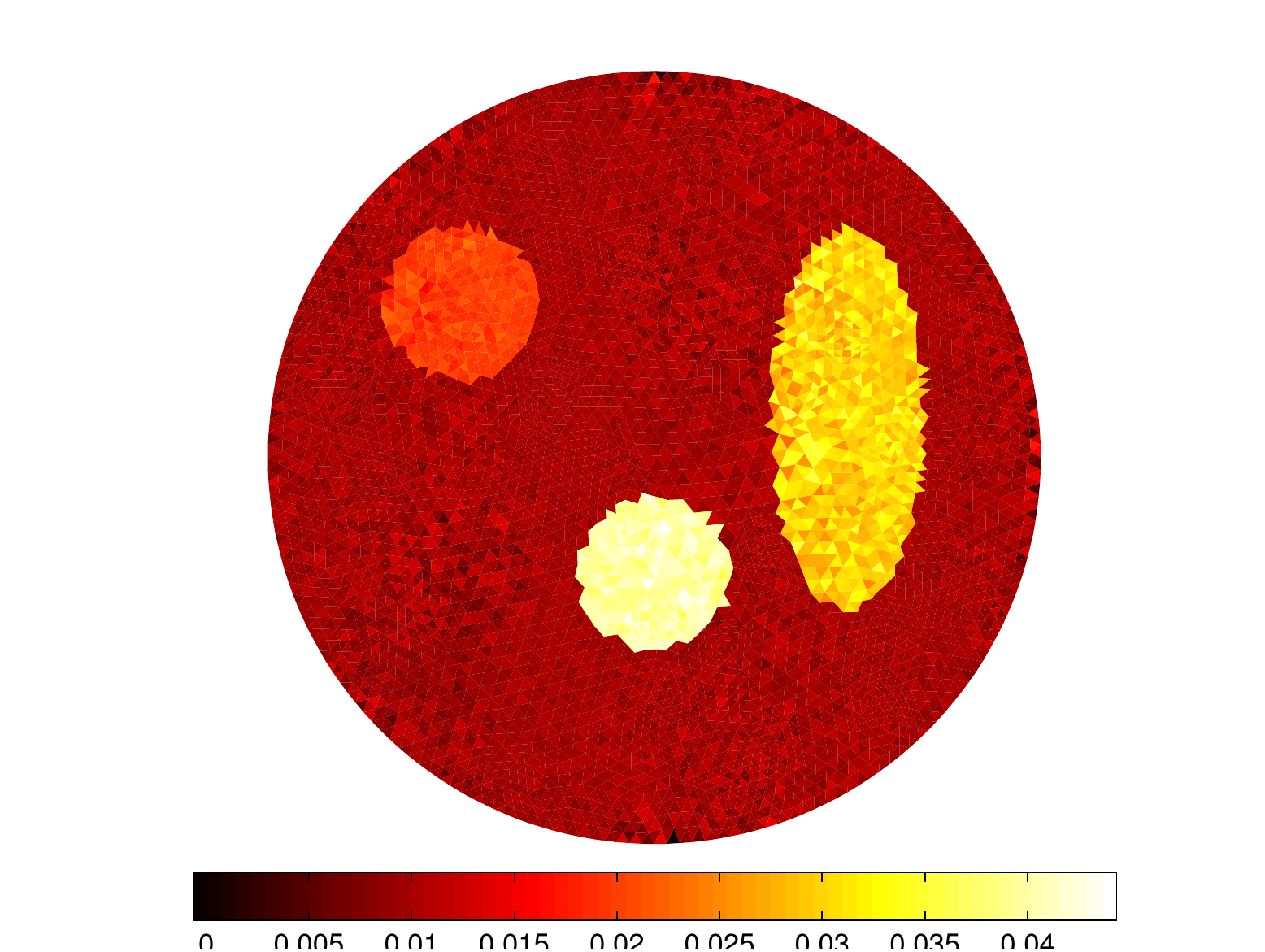}
\end{minipage}
\\
\begin{minipage}{0.24\linewidth}
\includegraphics[width=\textwidth]{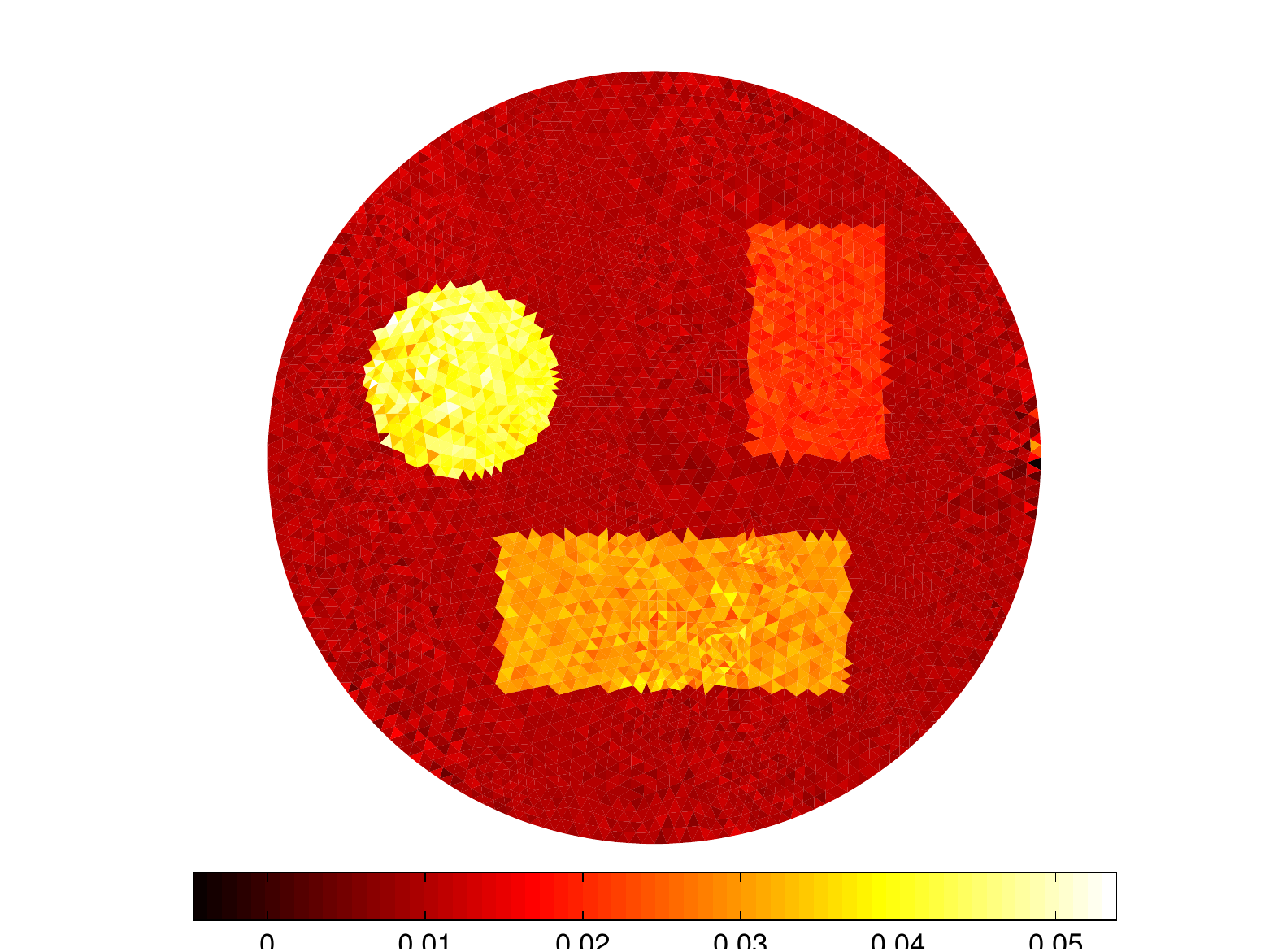}
\end{minipage}
\begin{minipage}{0.24\linewidth}
  \includegraphics[width=\textwidth]{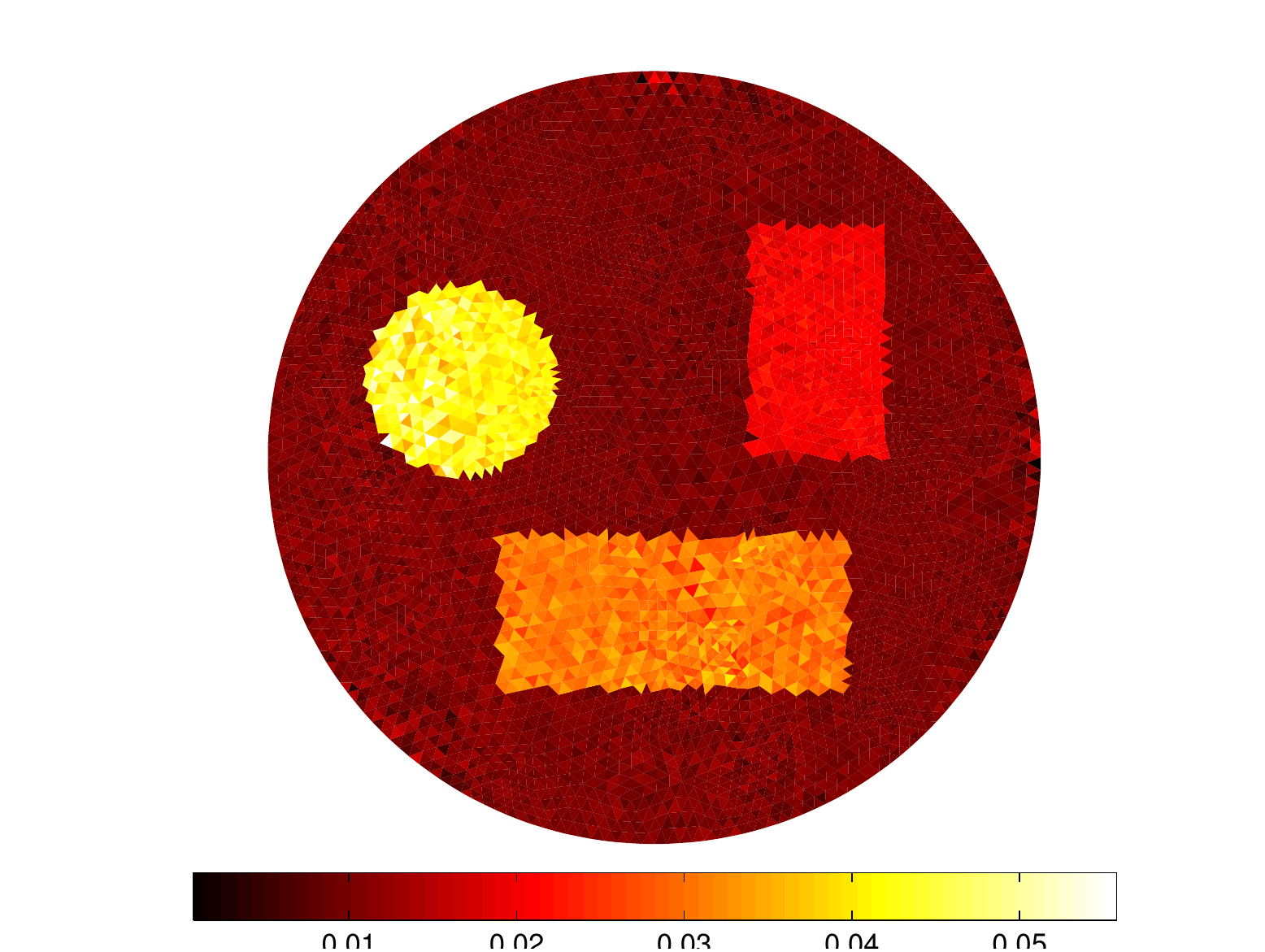}
\end{minipage}
\begin{minipage}{0.24\linewidth}
  \includegraphics[width=\textwidth]{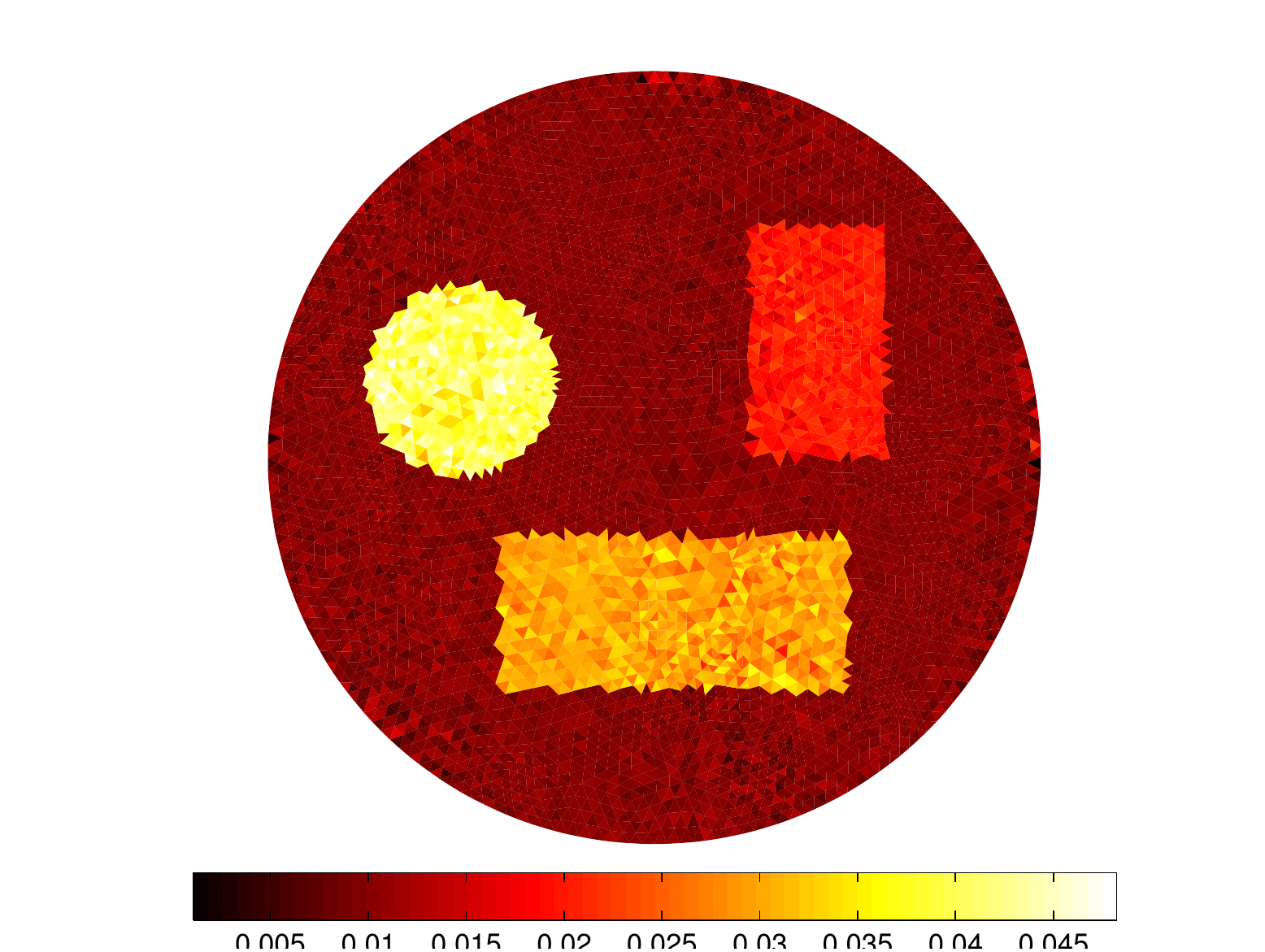}
\end{minipage}
\begin{minipage}{0.24\linewidth}
  \includegraphics[width=\textwidth]{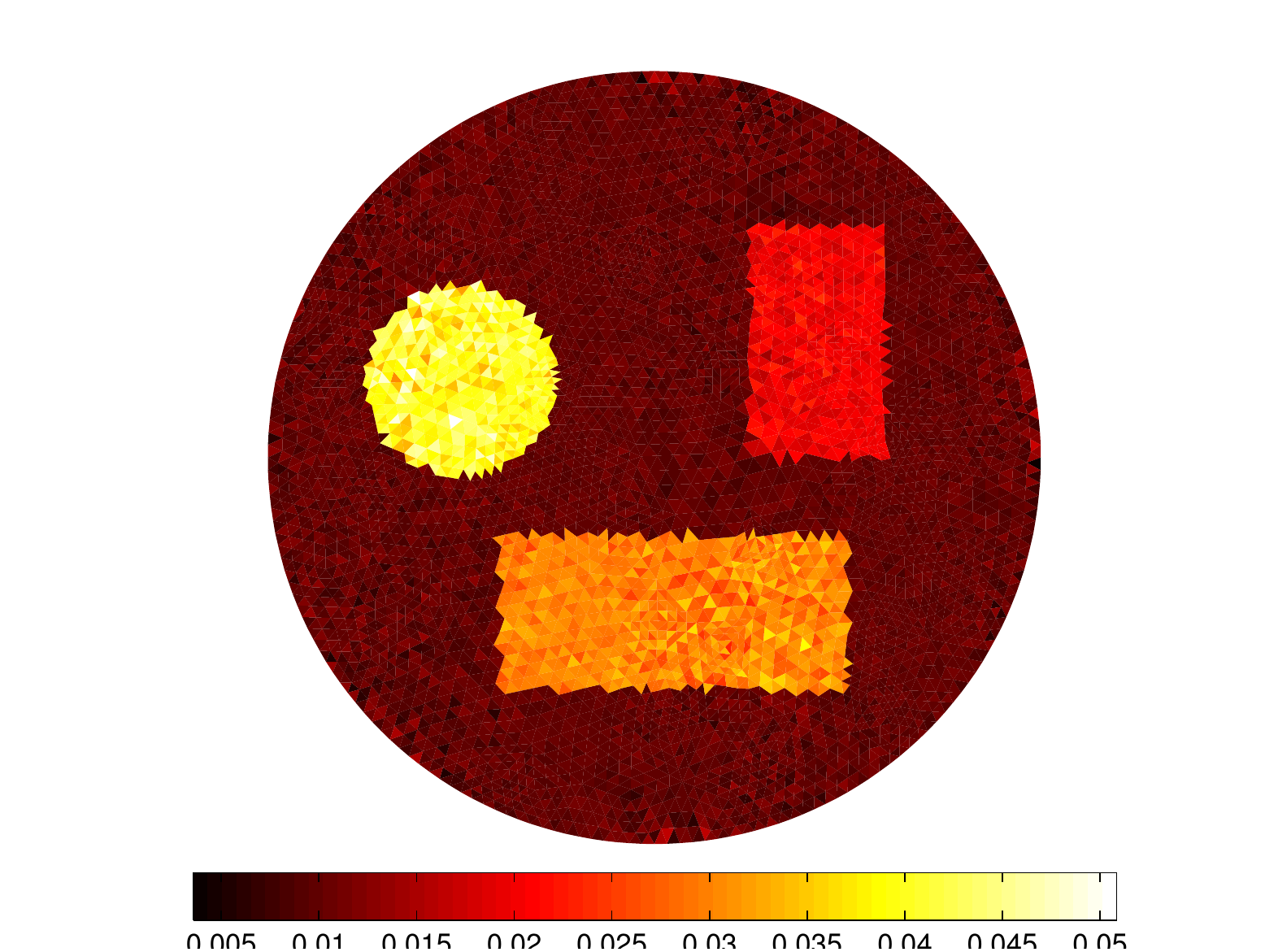}
\end{minipage}
\\
\begin{minipage}{0.24\linewidth}
\includegraphics[width=\textwidth]{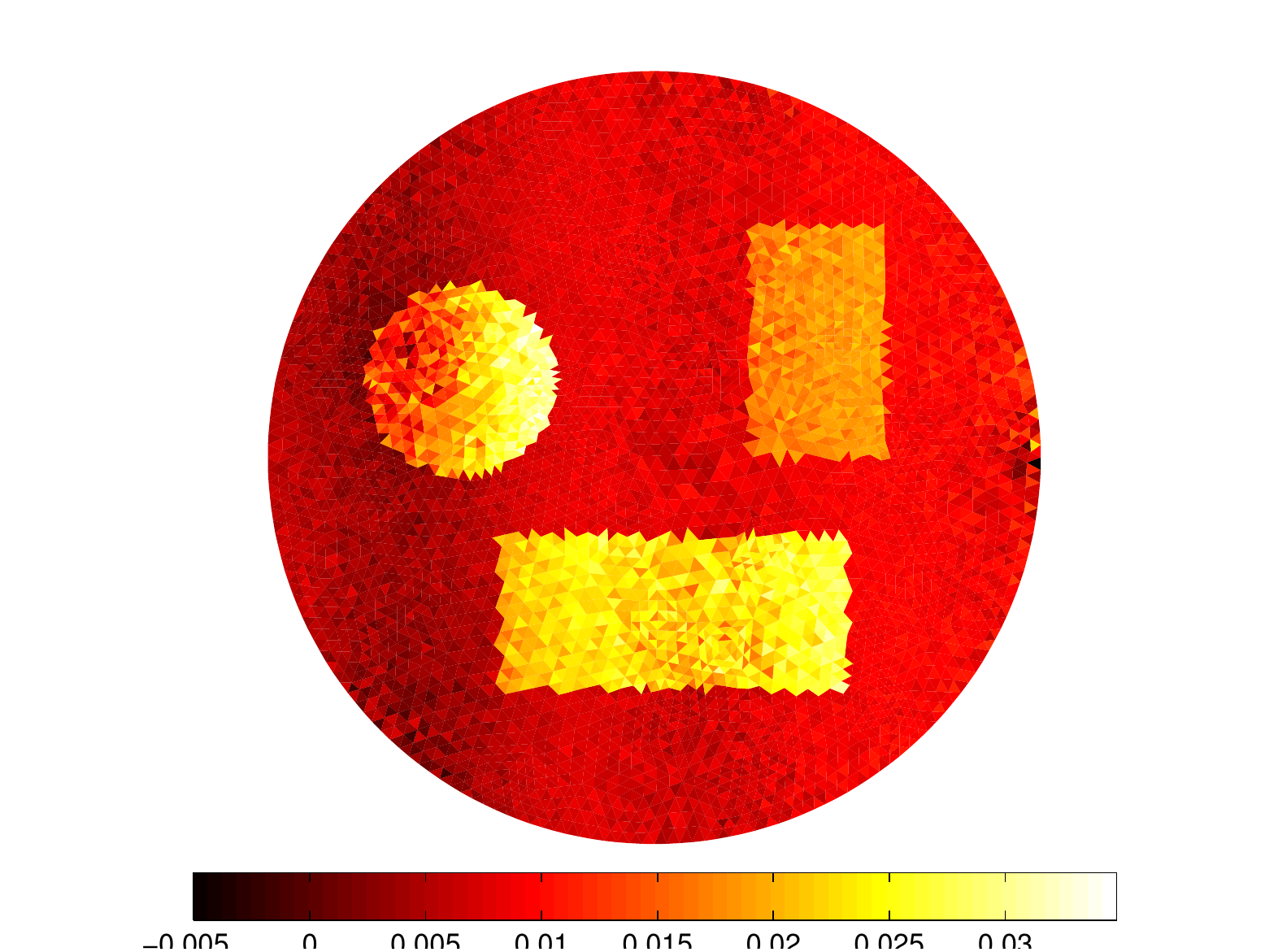}
\end{minipage}
\begin{minipage}{0.24\linewidth}
  \includegraphics[width=\textwidth]{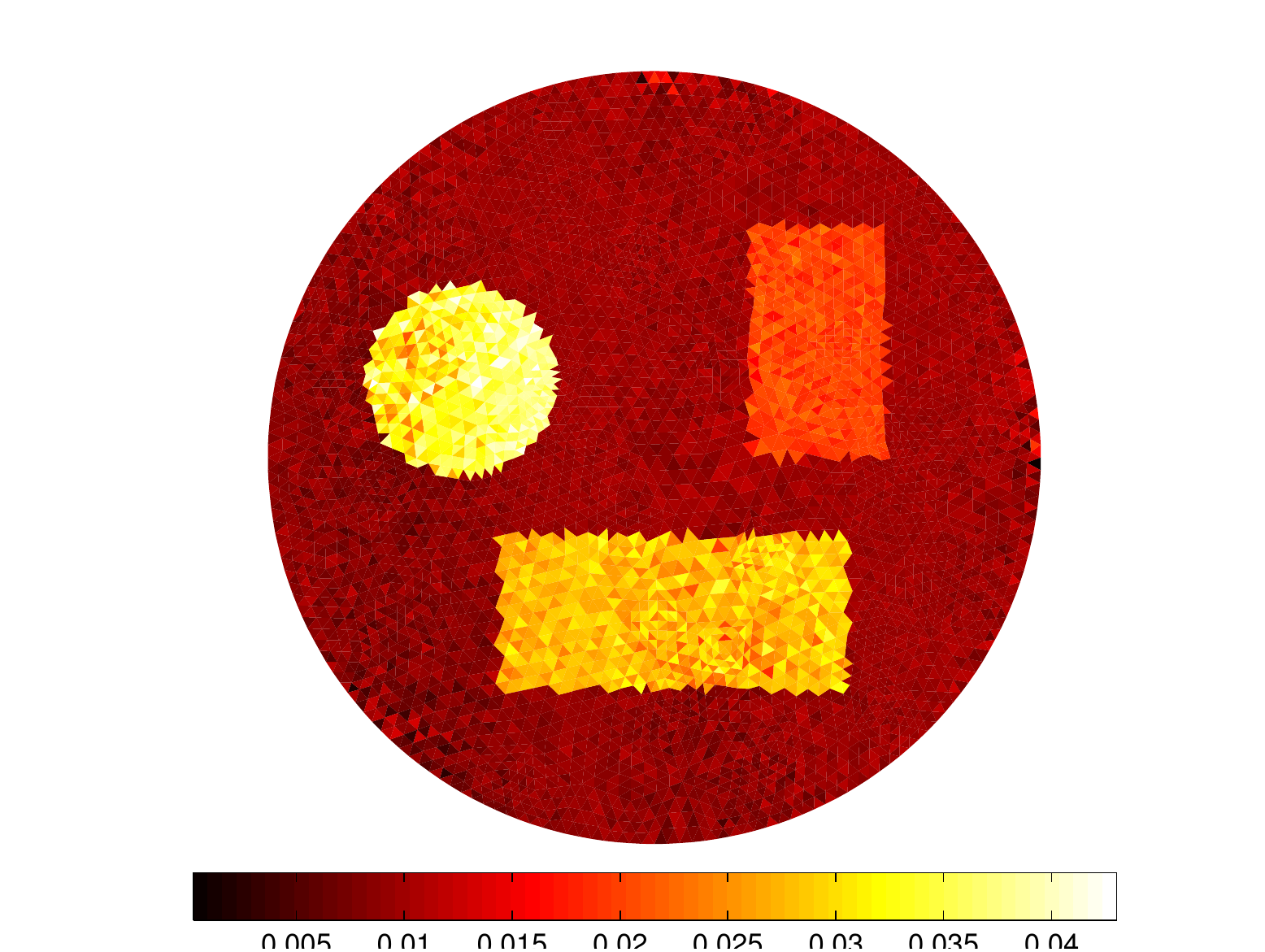}
\end{minipage}
\begin{minipage}{0.24\linewidth}
  \includegraphics[width=\textwidth]{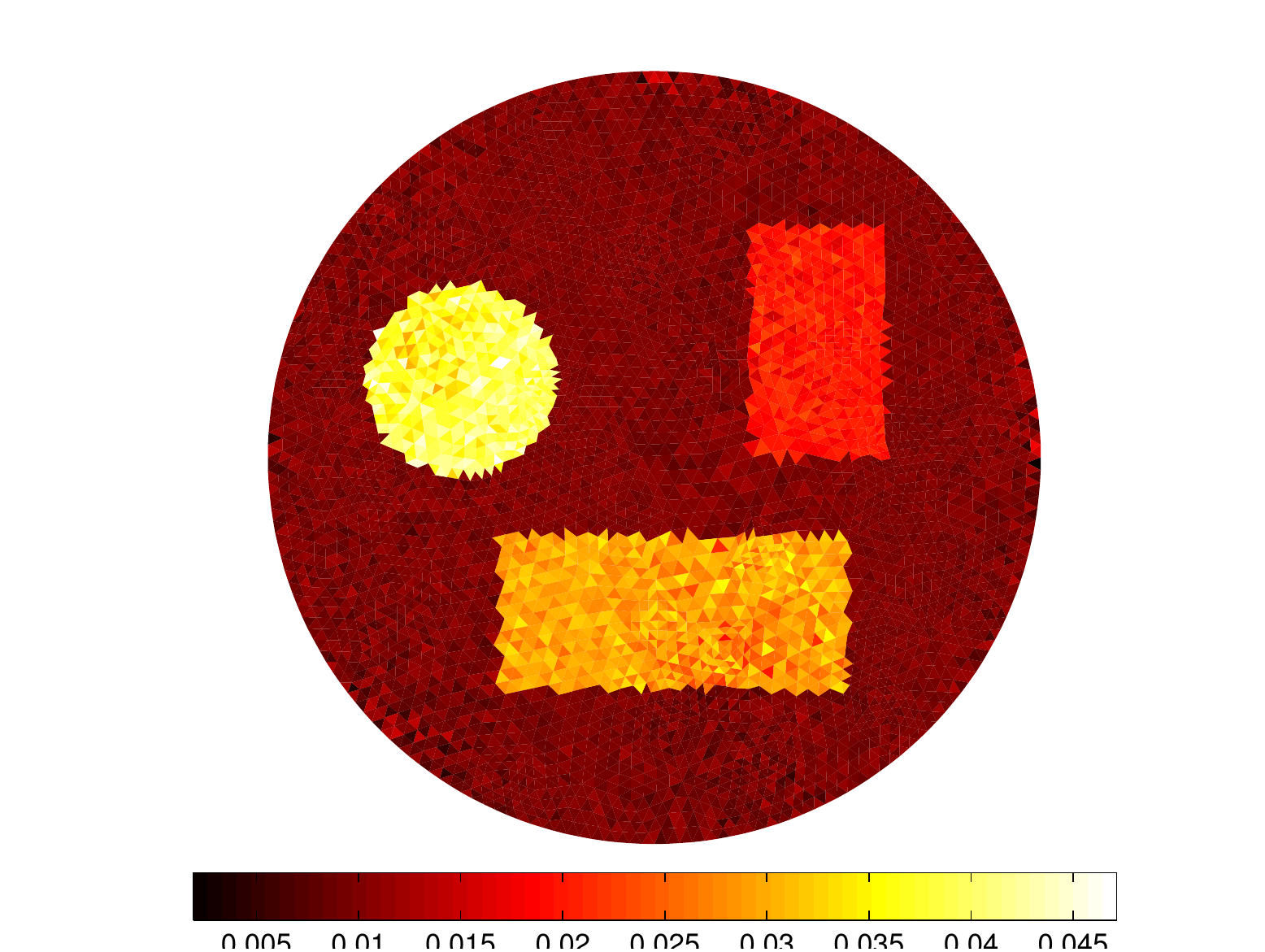}
\end{minipage}
\begin{minipage}{0.24\linewidth}
  \includegraphics[width=\textwidth]{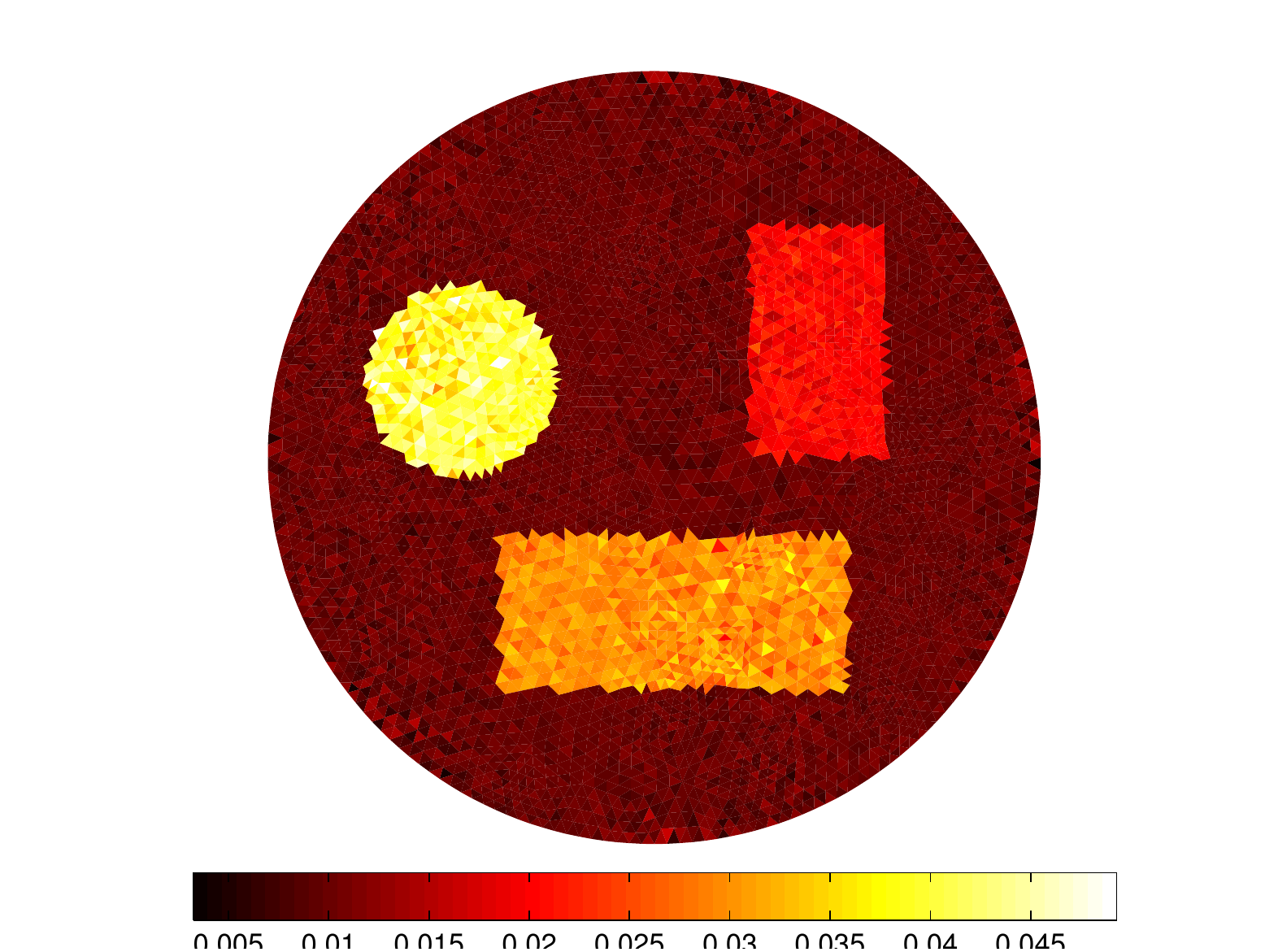}
\end{minipage}
 \caption{\label{fig:noise2}Reconstruction of $\mu_{a,xf}$ for 2$\%$ noise data by the hybrid method and the nonlinear optimization method. First, second row: first template. Third, fourth row: second template. First, third row: hybrid method. Second, fourth row: nonlinear optimization method. First, second, third, and fourth column: one-measurement, two-measurement, three-measurement and four-measurement.}
\end{figure}

\begin{figure}[htpb]
    \centering
\begin{minipage}{0.24\linewidth}
\includegraphics[width=\textwidth]{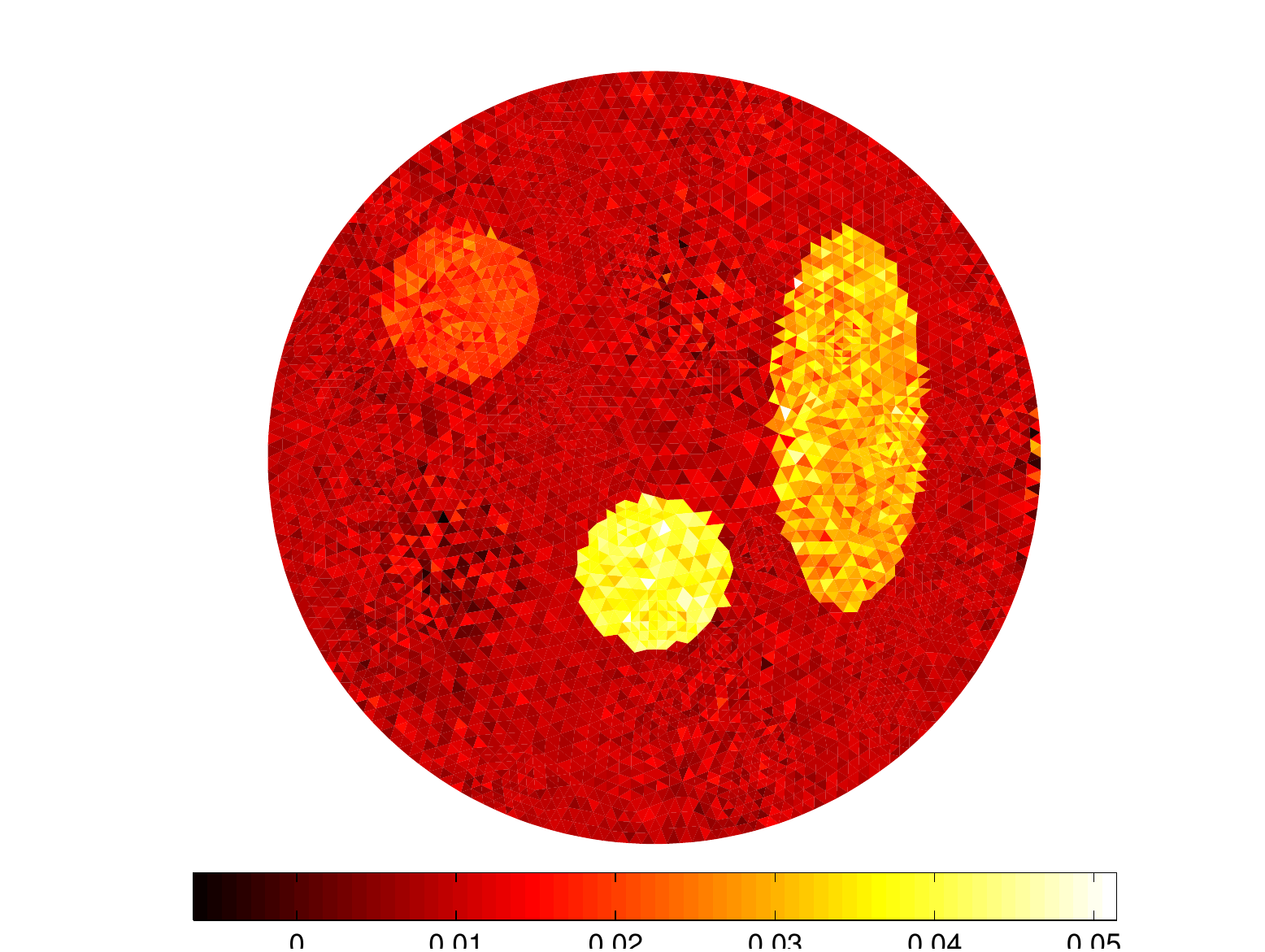}
\end{minipage}
\begin{minipage}{0.24\linewidth}
  \includegraphics[width=\textwidth]{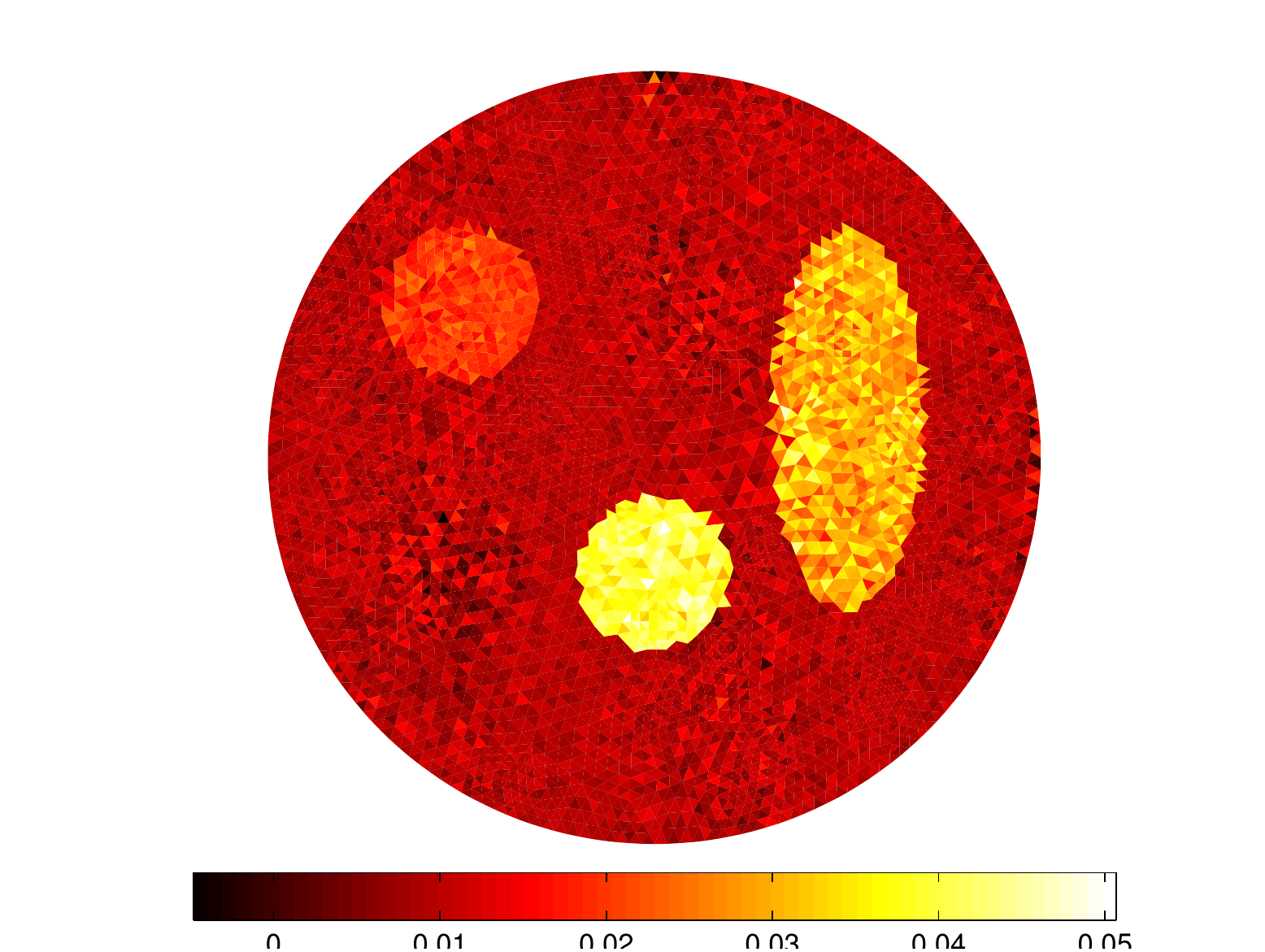}
\end{minipage}
\begin{minipage}{0.24\linewidth}
  \includegraphics[width=\textwidth]{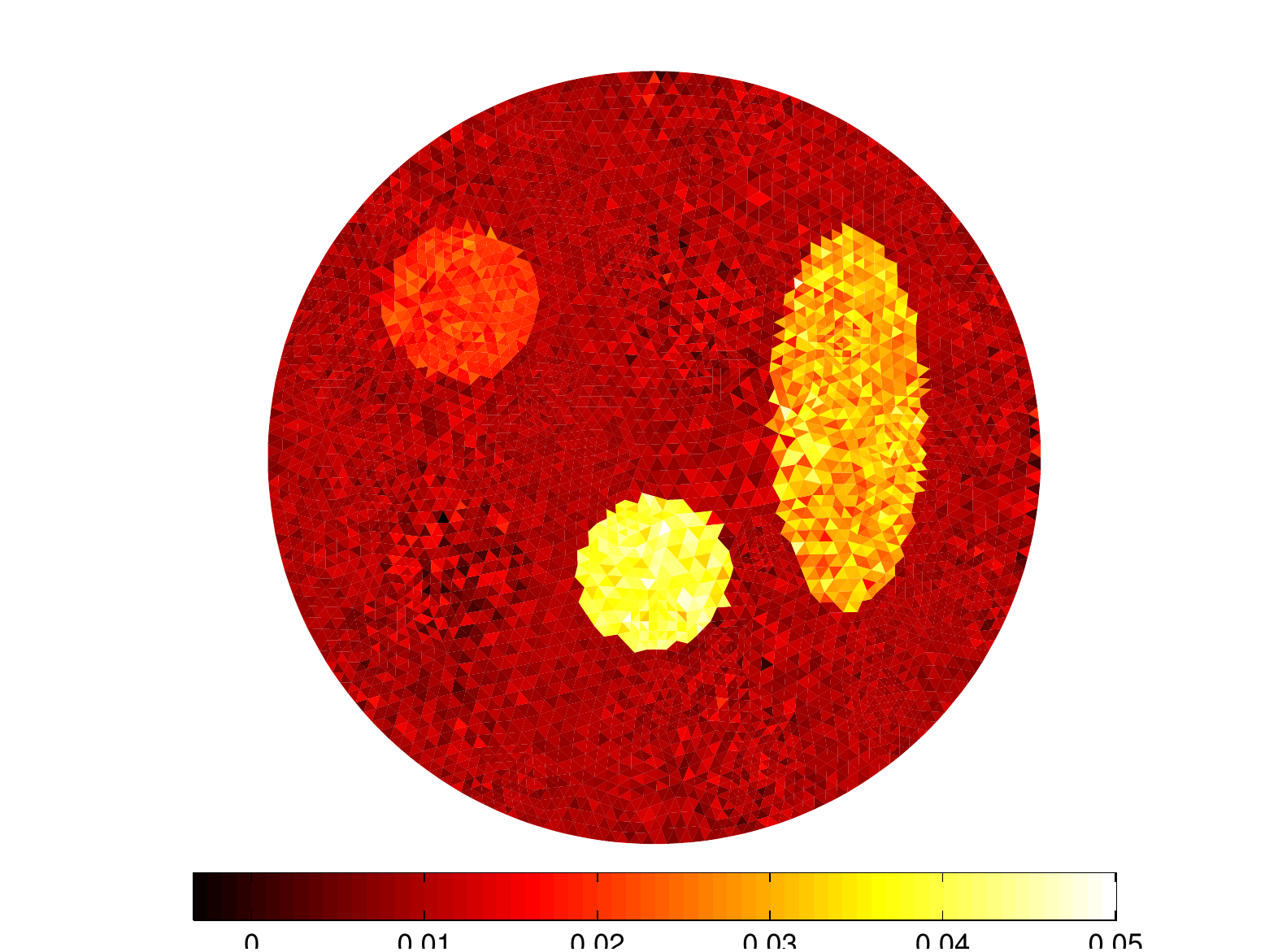}
\end{minipage}
\begin{minipage}{0.24\linewidth}
  \includegraphics[width=\textwidth]{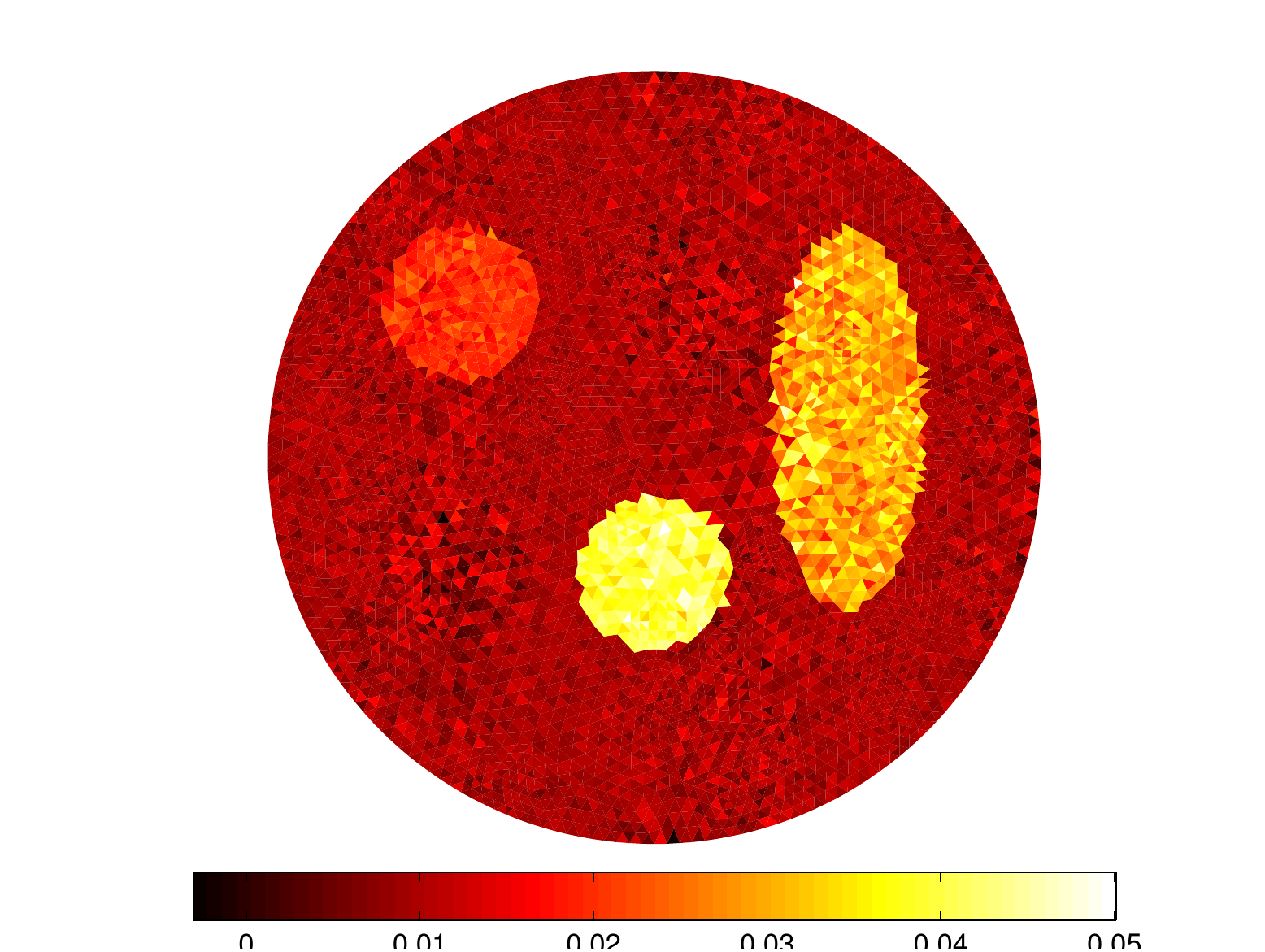}
\end{minipage}
\\
\begin{minipage}{0.24\linewidth}
\includegraphics[width=\textwidth]{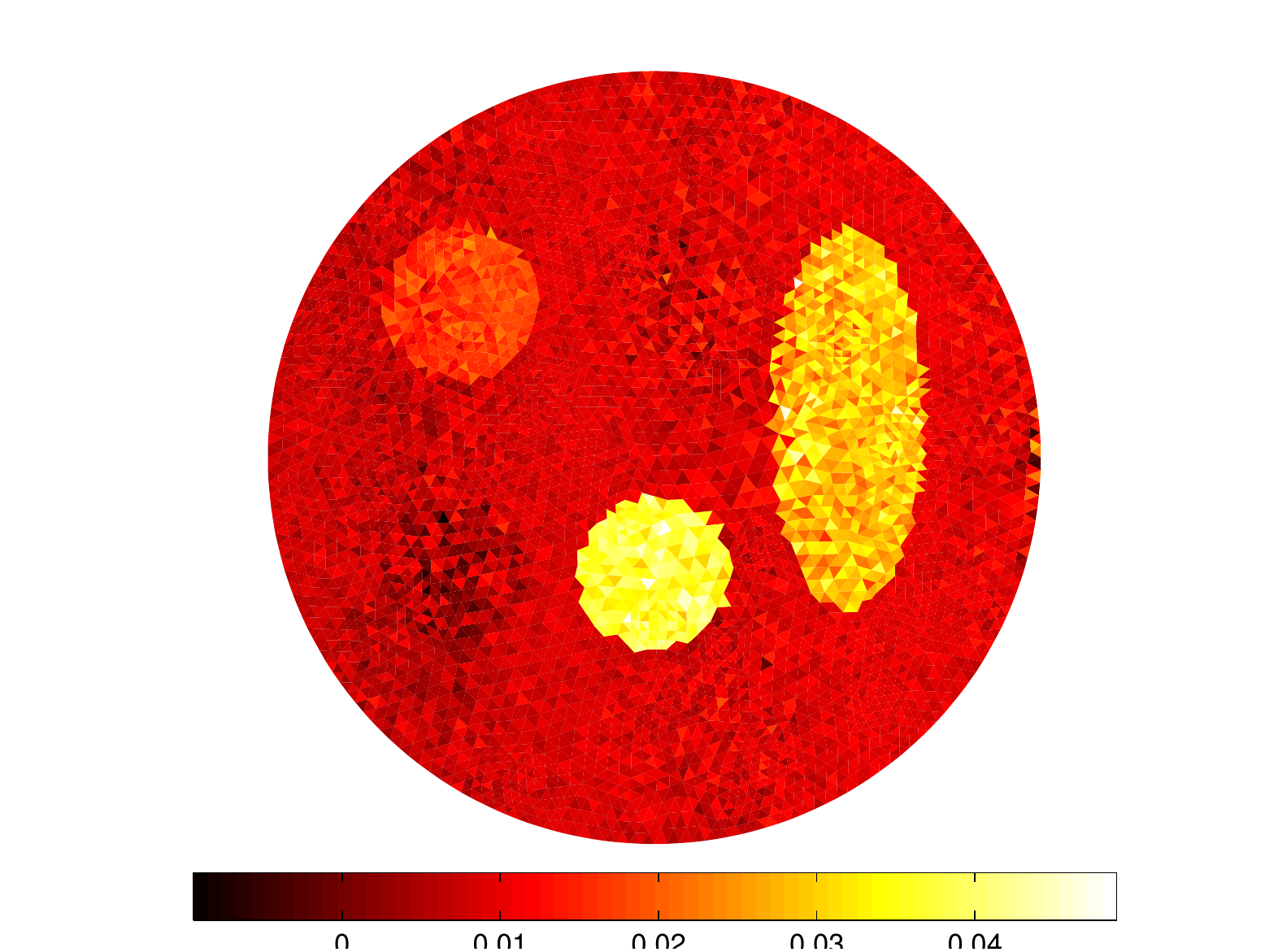}
\end{minipage}
\begin{minipage}{0.24\linewidth}
  \includegraphics[width=\textwidth]{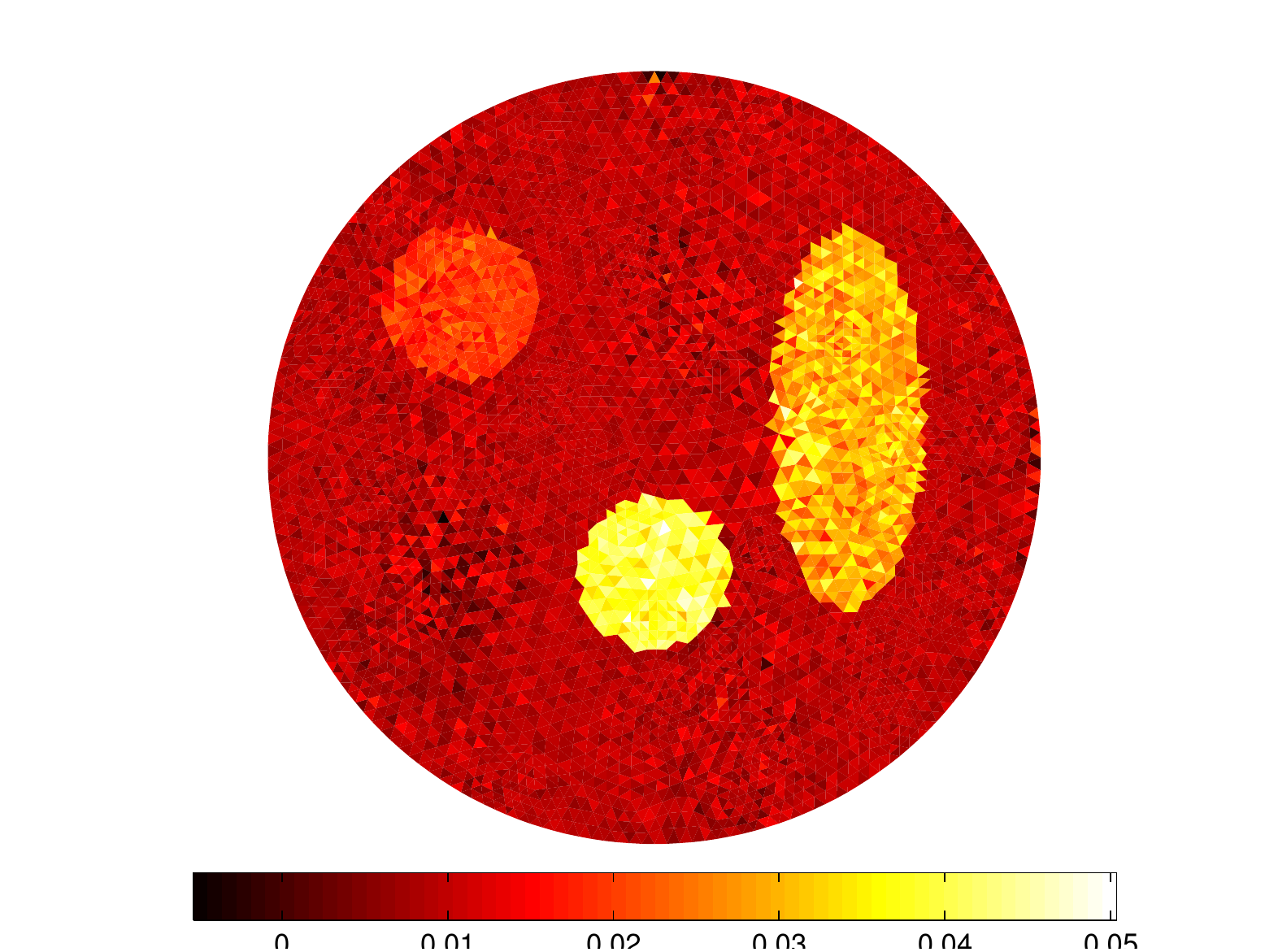}
\end{minipage}
\begin{minipage}{0.24\linewidth}
  \includegraphics[width=\textwidth]{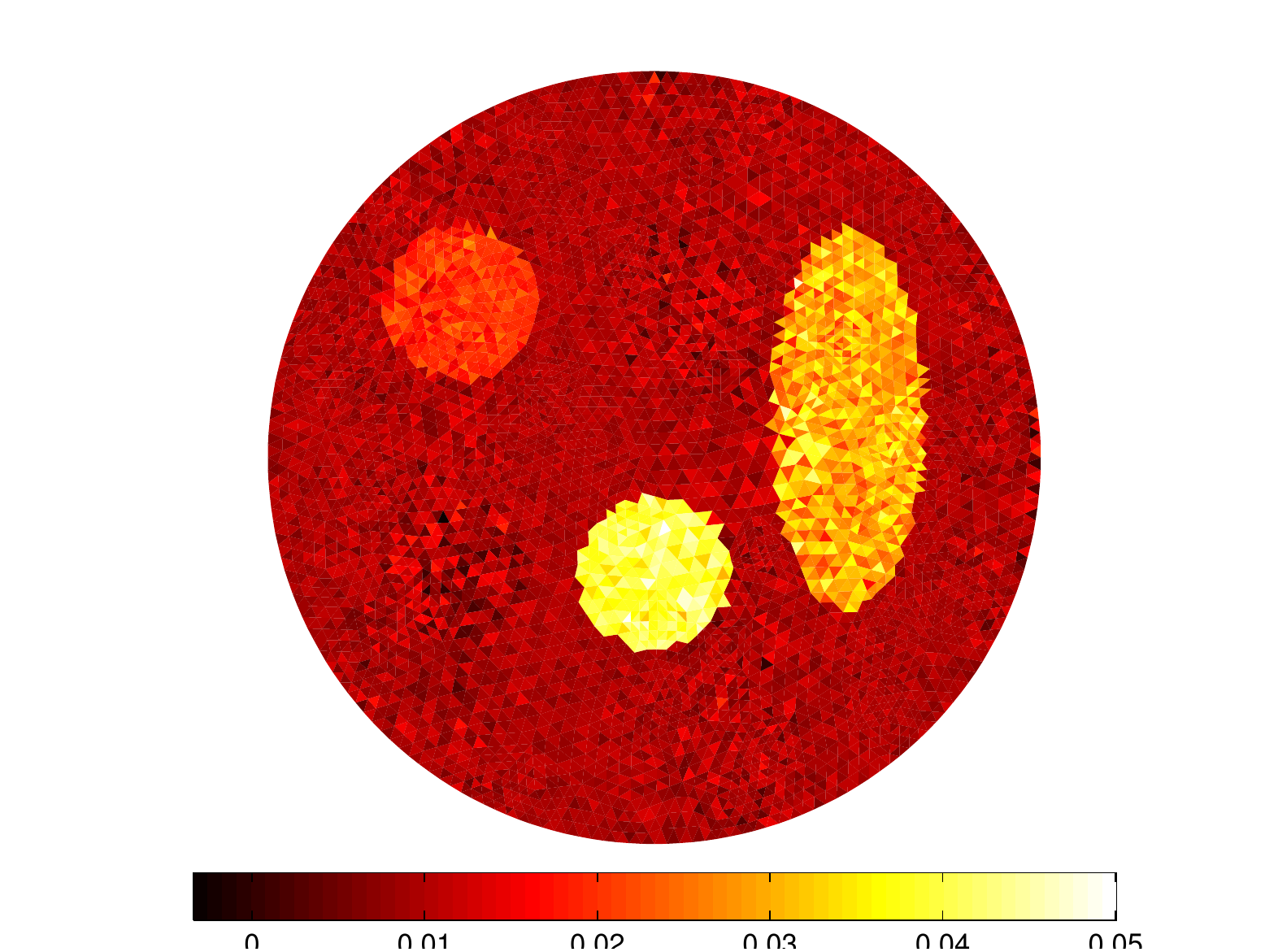}
\end{minipage}
\begin{minipage}{0.24\linewidth}
  \includegraphics[width=\textwidth]{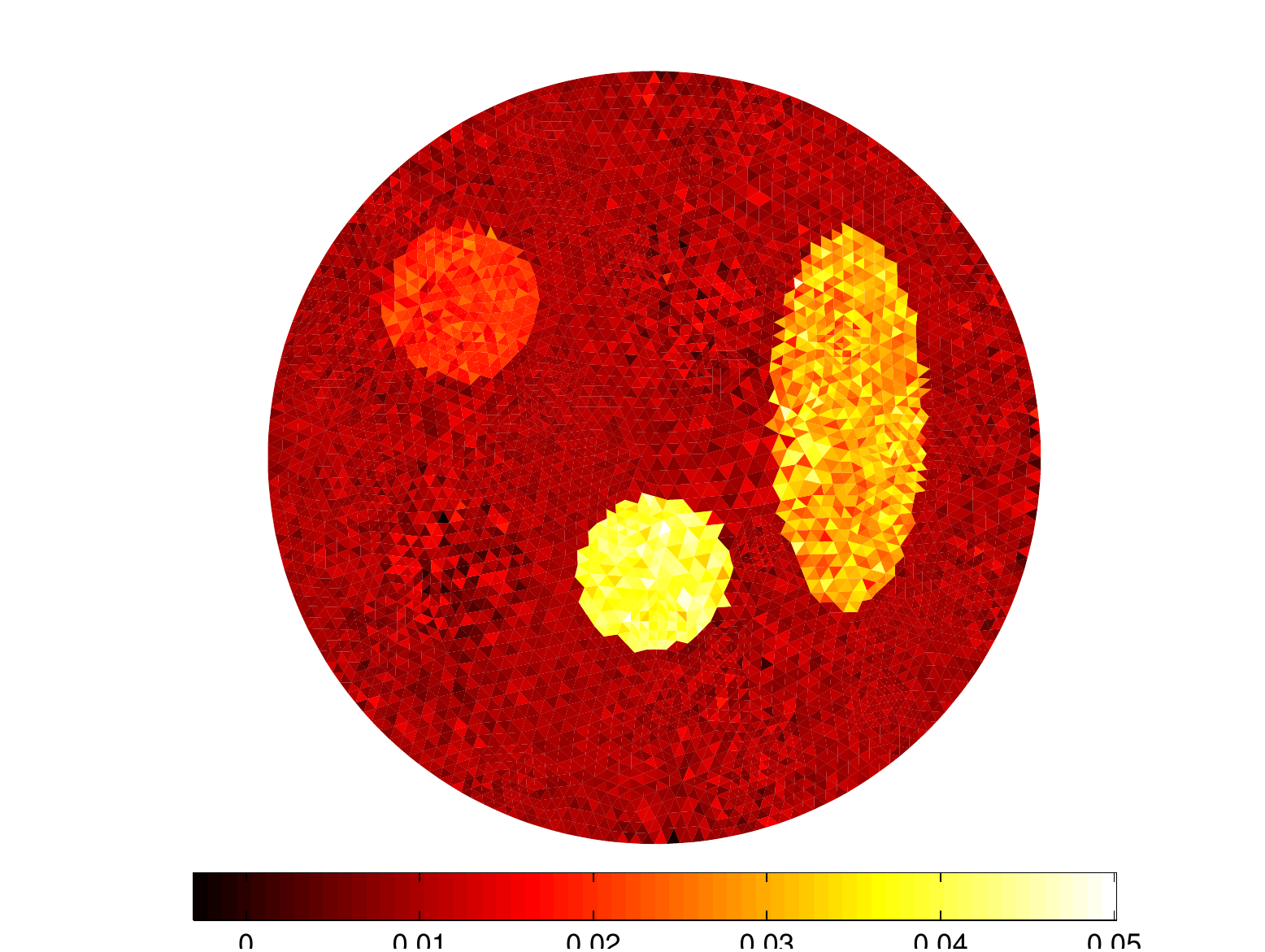}
\end{minipage}
\\
\begin{minipage}{0.24\linewidth}
\includegraphics[width=\textwidth]{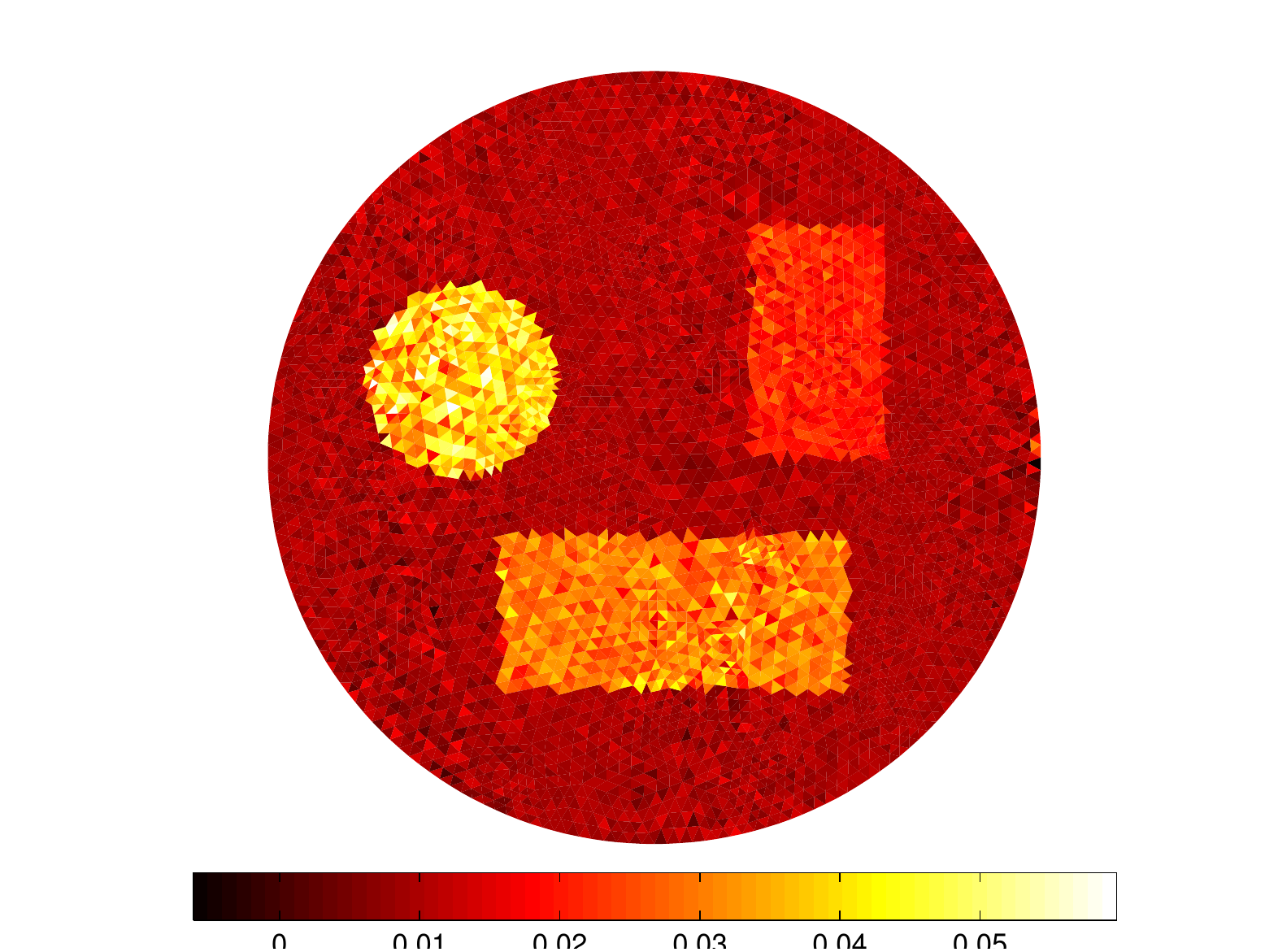}
\end{minipage}
\begin{minipage}{0.24\linewidth}
  \includegraphics[width=\textwidth]{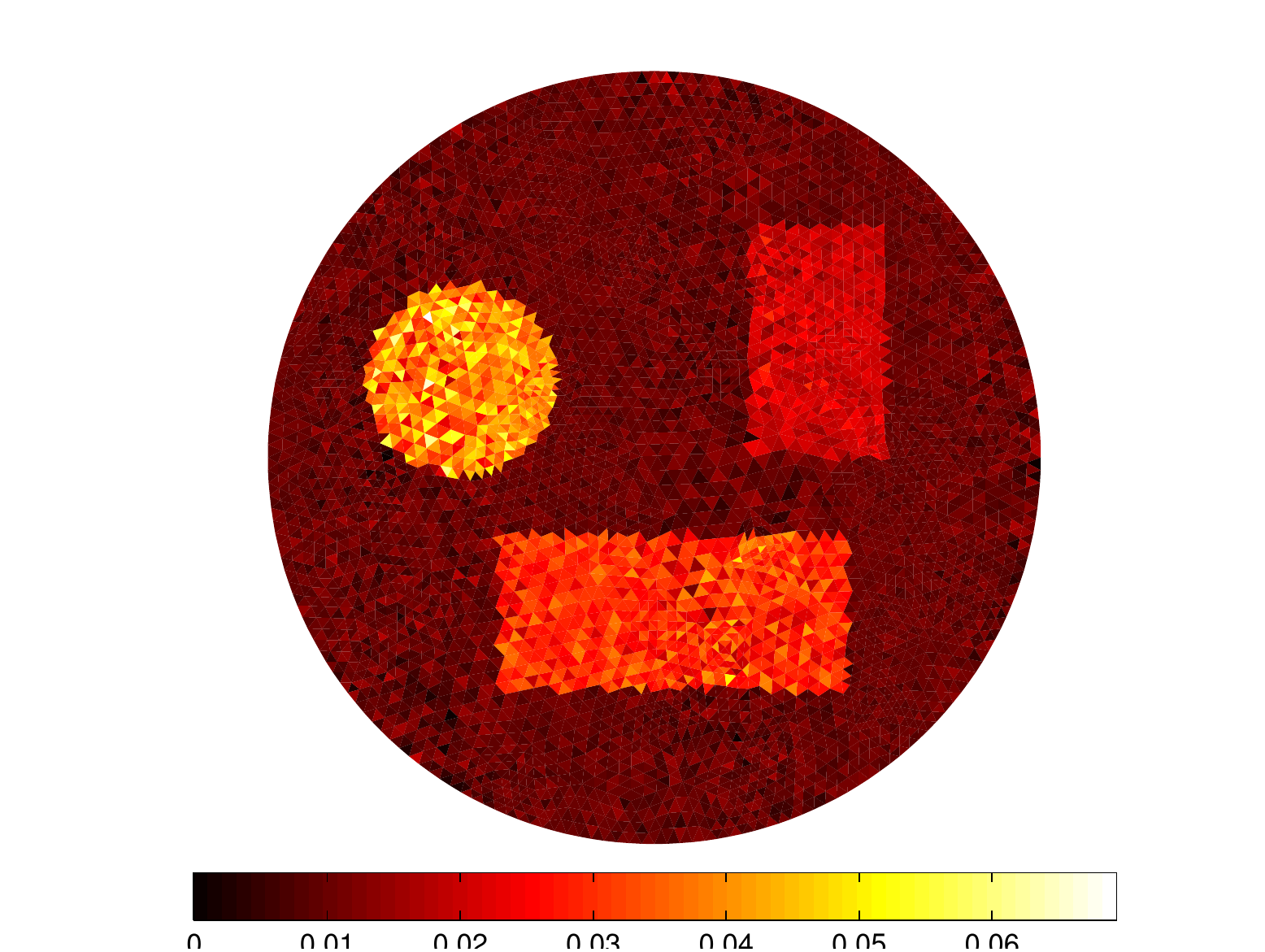}
\end{minipage}
\begin{minipage}{0.24\linewidth}
  \includegraphics[width=\textwidth]{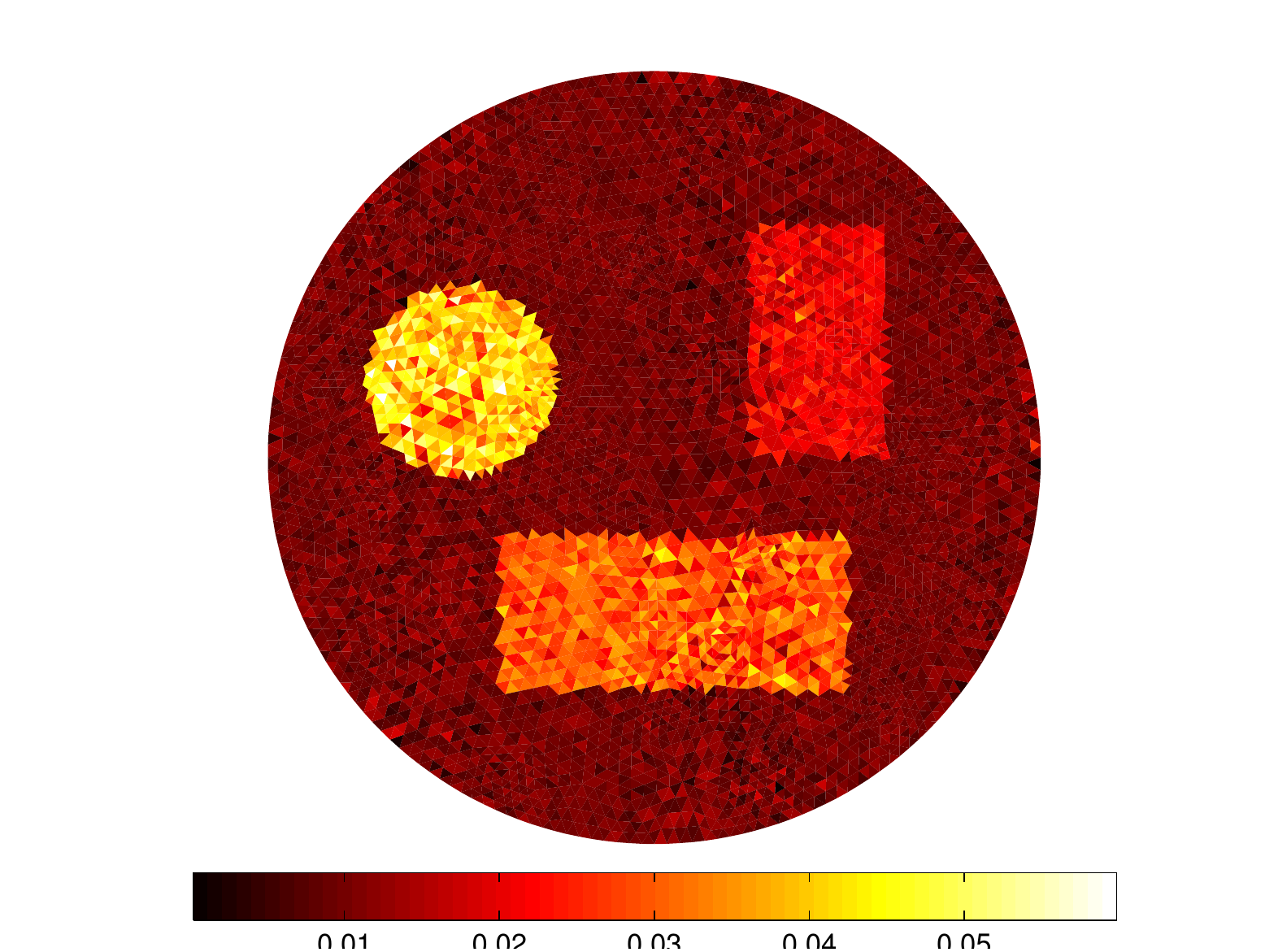}
\end{minipage}
\begin{minipage}{0.24\linewidth}
  \includegraphics[width=\textwidth]{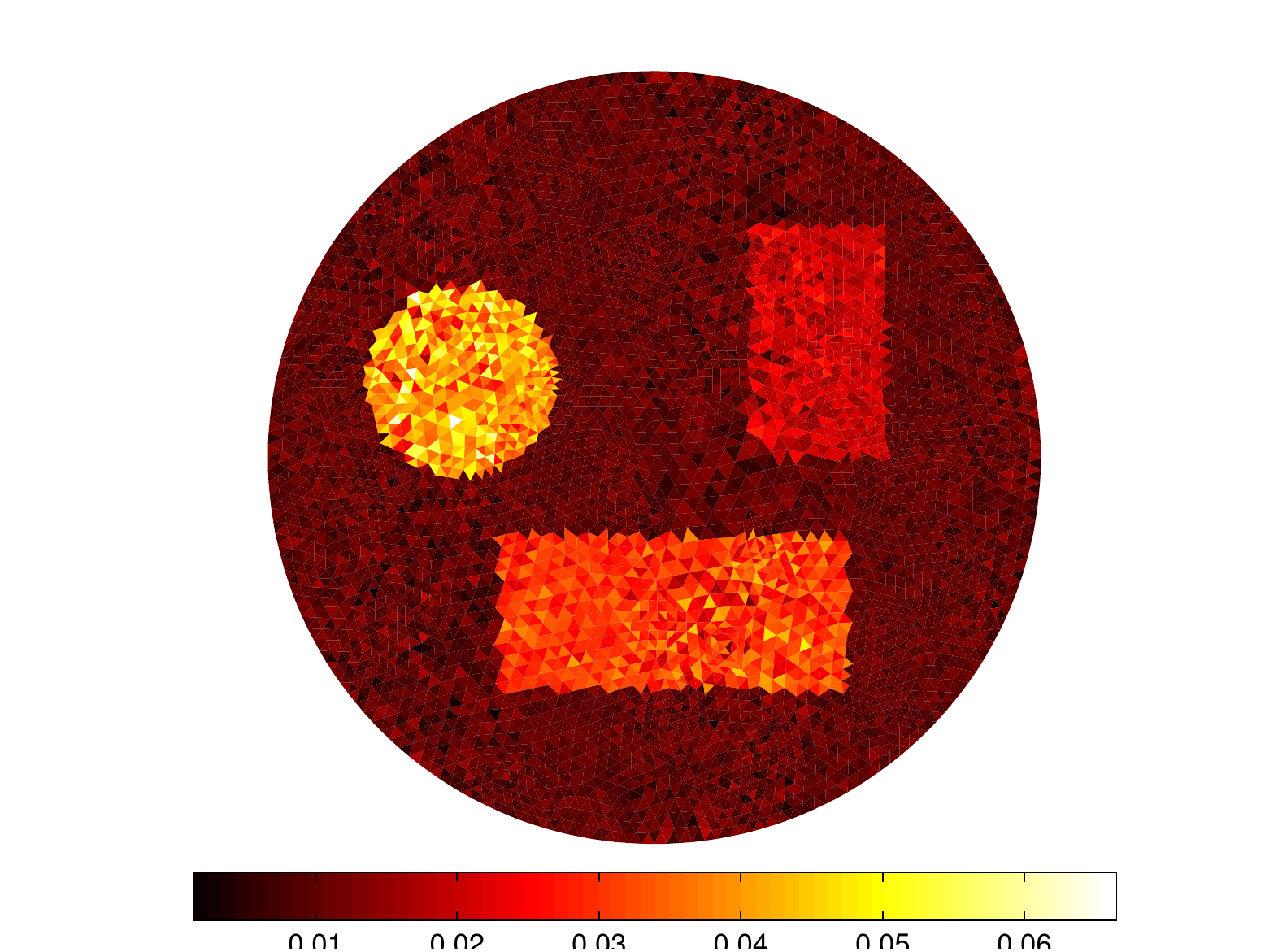}
\end{minipage}
\\
\begin{minipage}{0.24\linewidth}
\includegraphics[width=\textwidth]{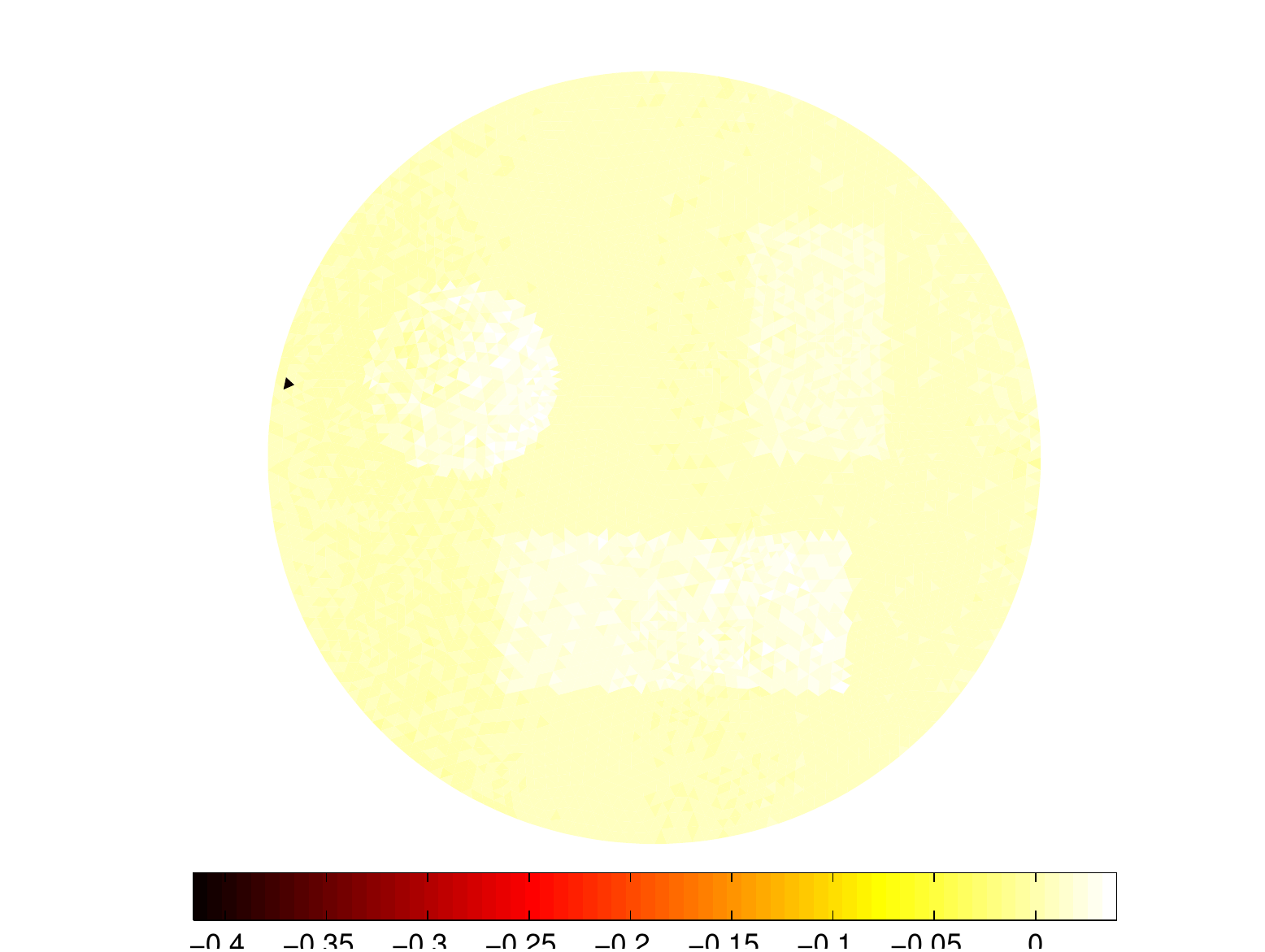}
\end{minipage}
\begin{minipage}{0.24\linewidth}
  \includegraphics[width=\textwidth]{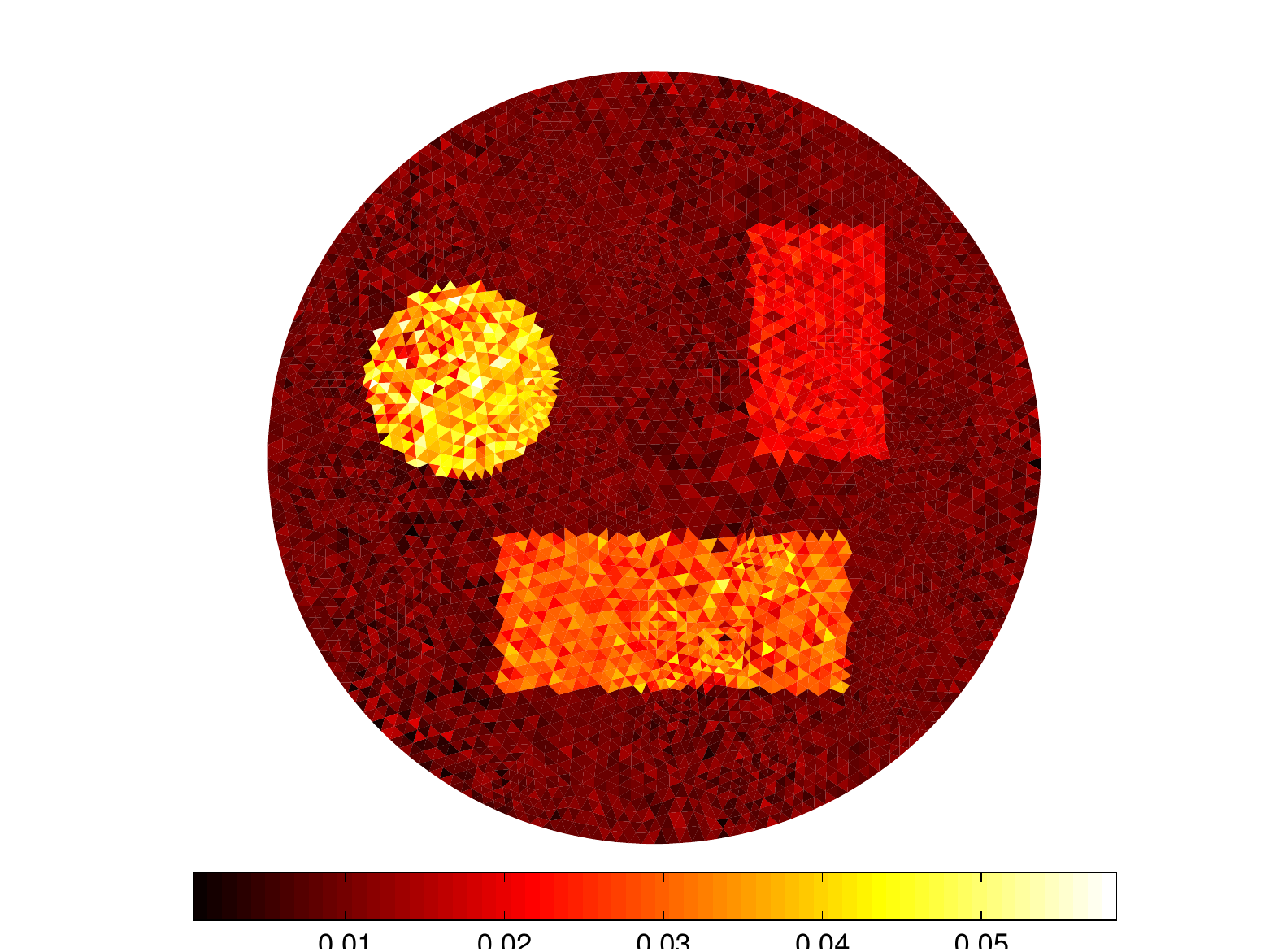}
\end{minipage}
\begin{minipage}{0.24\linewidth}
  \includegraphics[width=\textwidth]{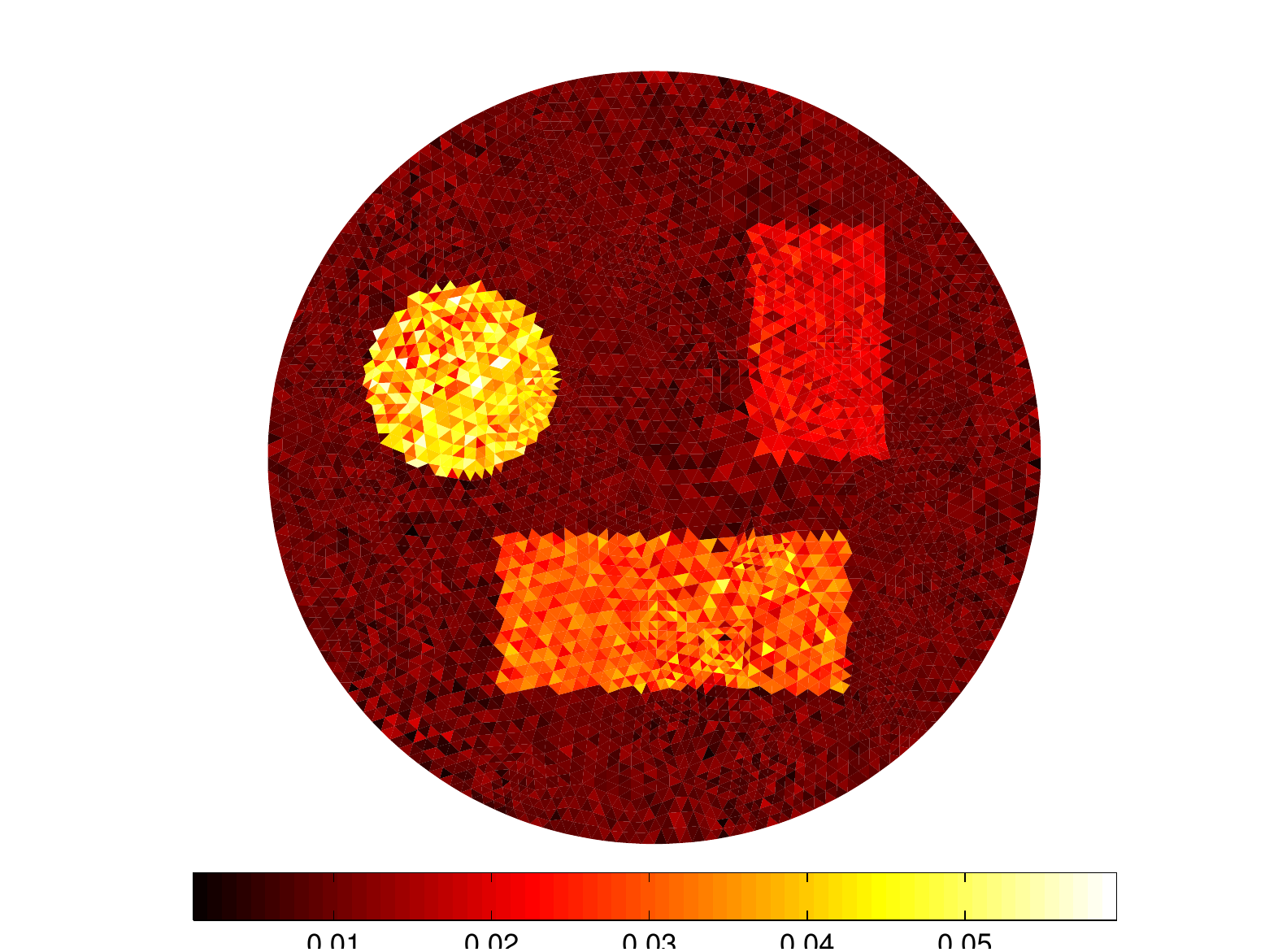}
\end{minipage}
\begin{minipage}{0.24\linewidth}
  \includegraphics[width=\textwidth]{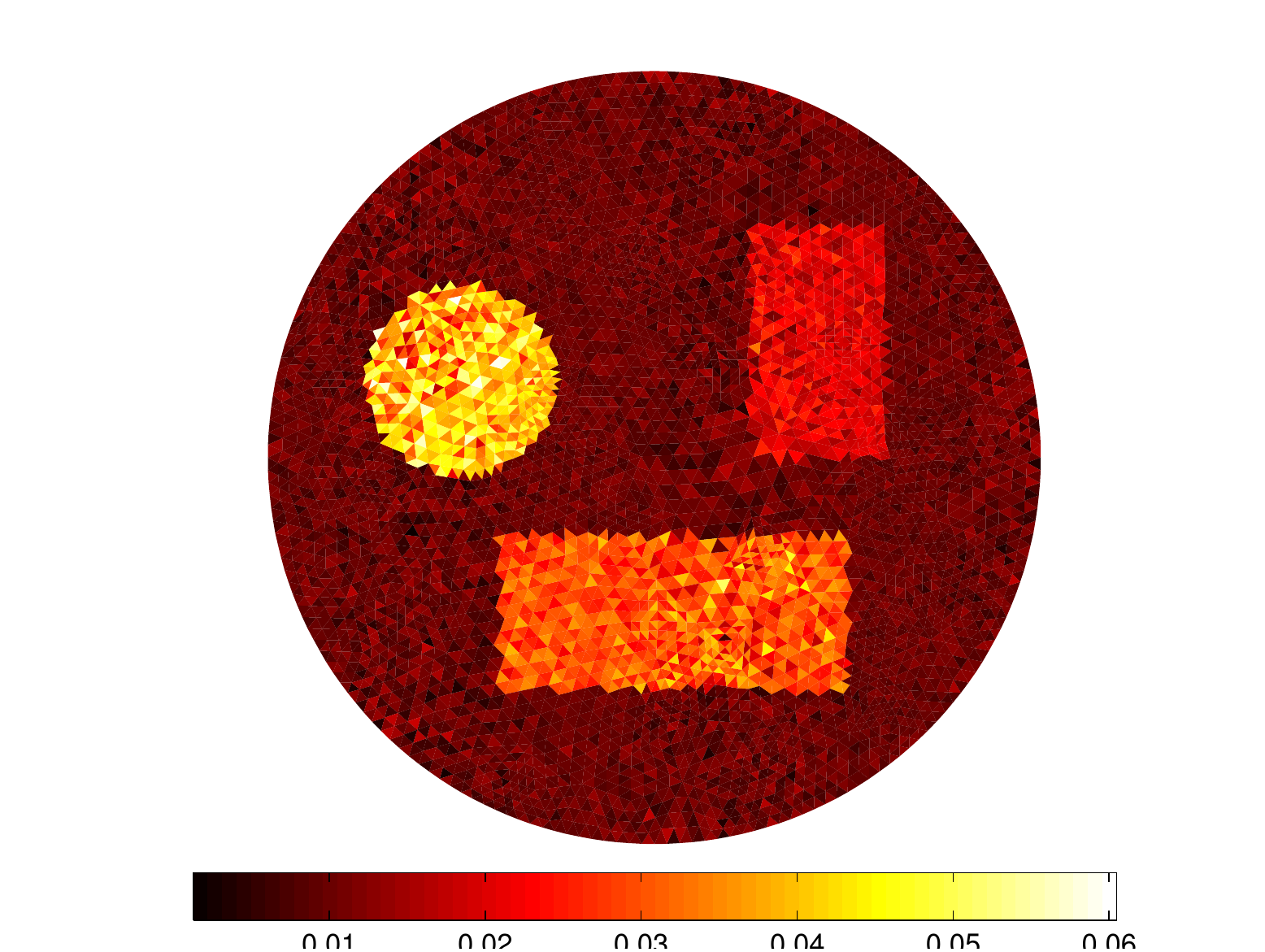}
\end{minipage}
 \caption{\label{fig:noise5}Reconstruction of $\mu_{a,xf}$ for 5$\%$ noise data by the hybrid method and the nonlinear optimization method. First, second row: first template. Third, fourth row: second template. First, third row: hybrid method. Second, fourth row: nonlinear optimization method. First, second, third, and fourth column: one-measurement, two-measurement, three-measurement and four-measurement.}
\end{figure}

\begin{figure}[htpb]
    \centering
\begin{minipage}{0.24\linewidth}
\includegraphics[width=\textwidth]{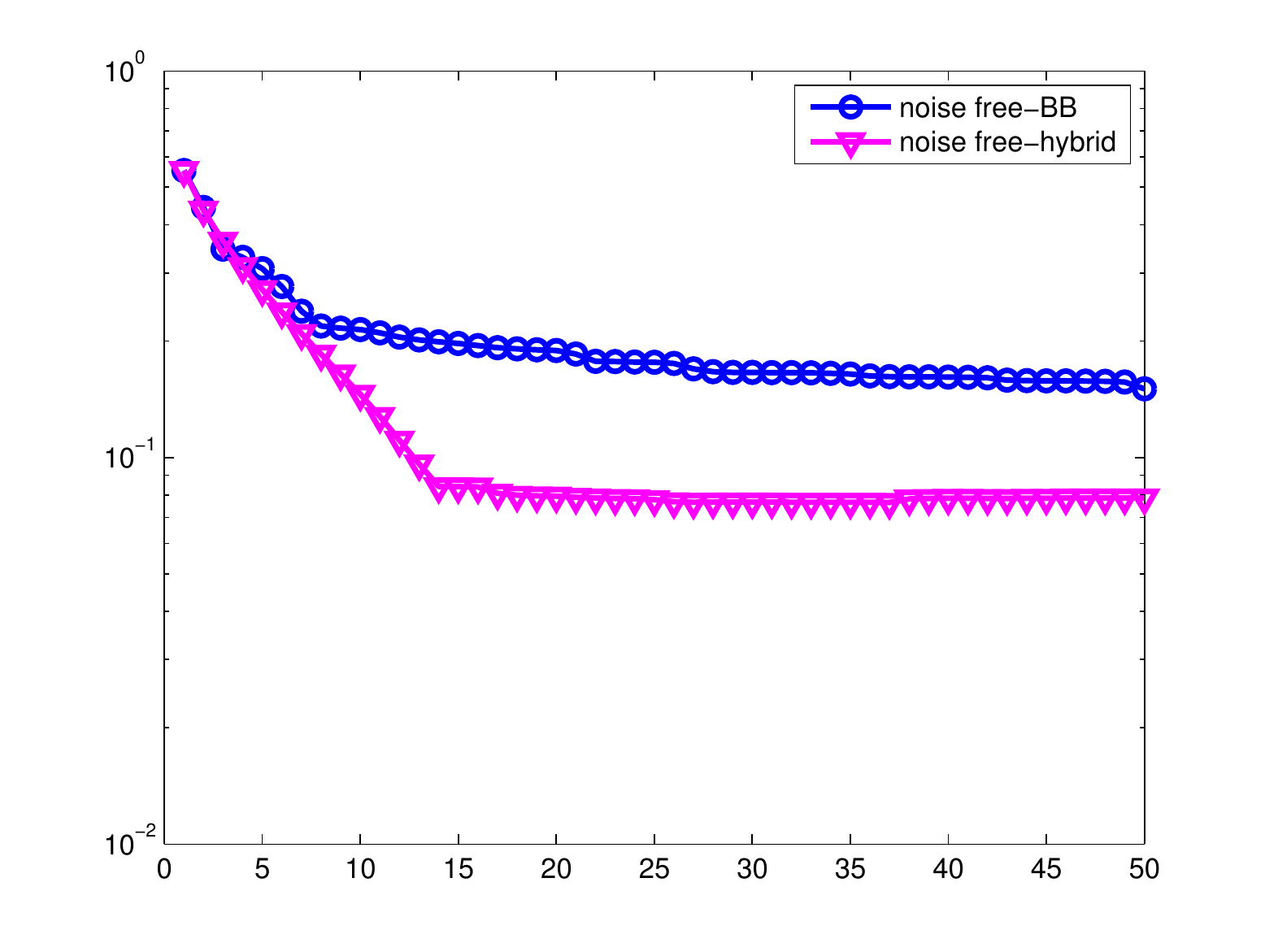}
\end{minipage}
\begin{minipage}{0.24\linewidth}
  \includegraphics[width=\textwidth]{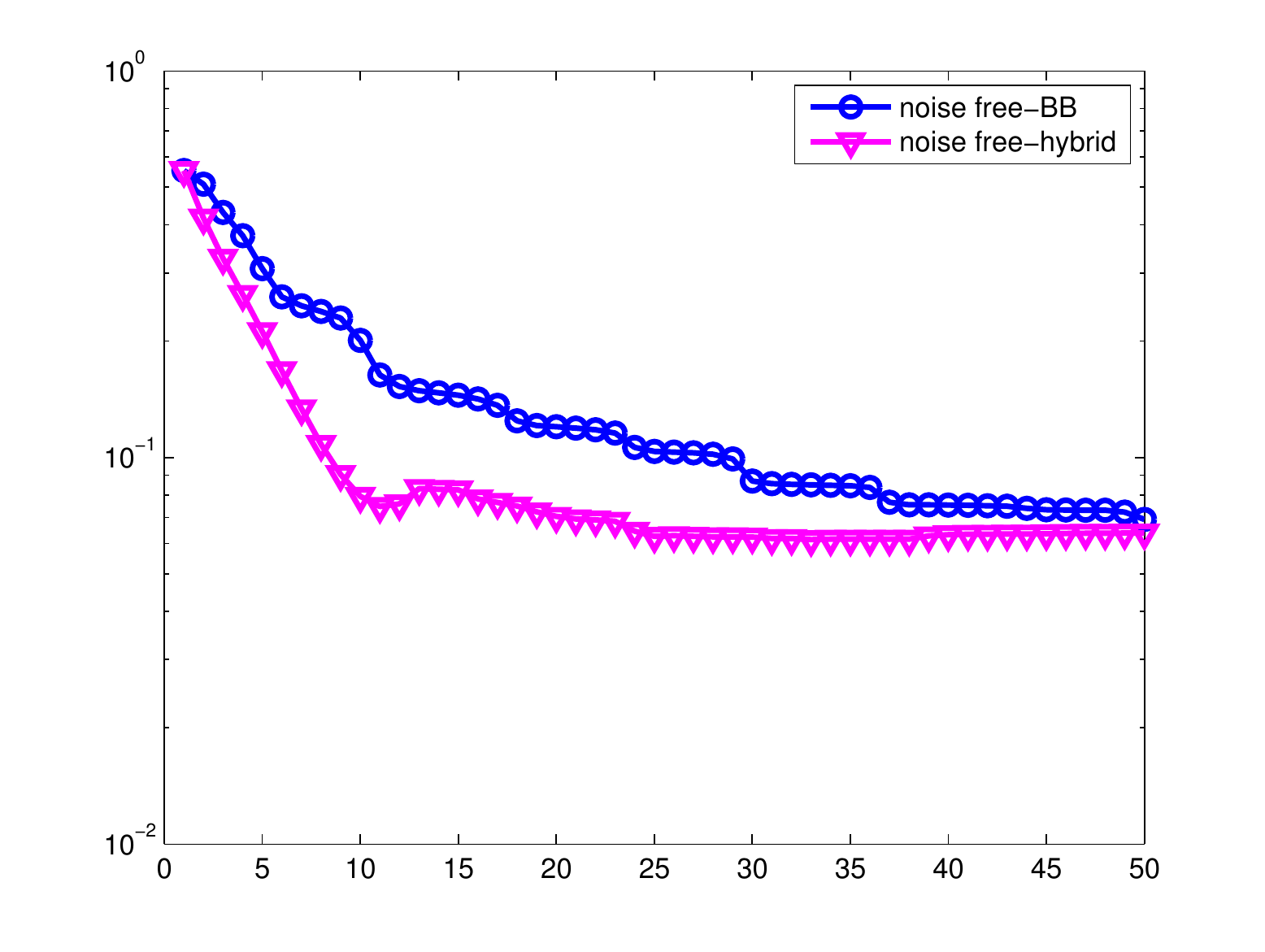}
\end{minipage}
\begin{minipage}{0.24\linewidth}
  \includegraphics[width=\textwidth]{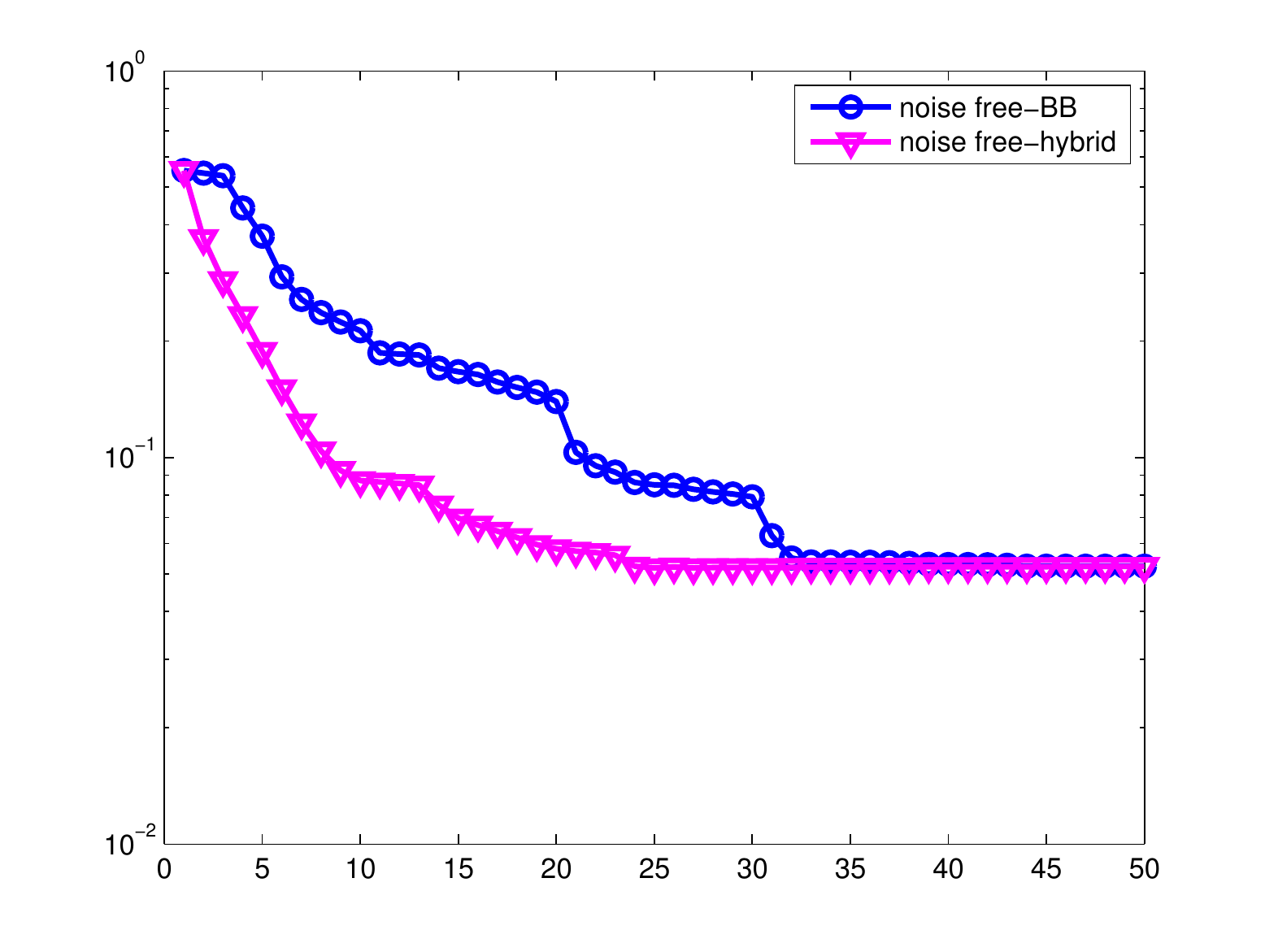}
\end{minipage}
\begin{minipage}{0.24\linewidth}
  \includegraphics[width=\textwidth]{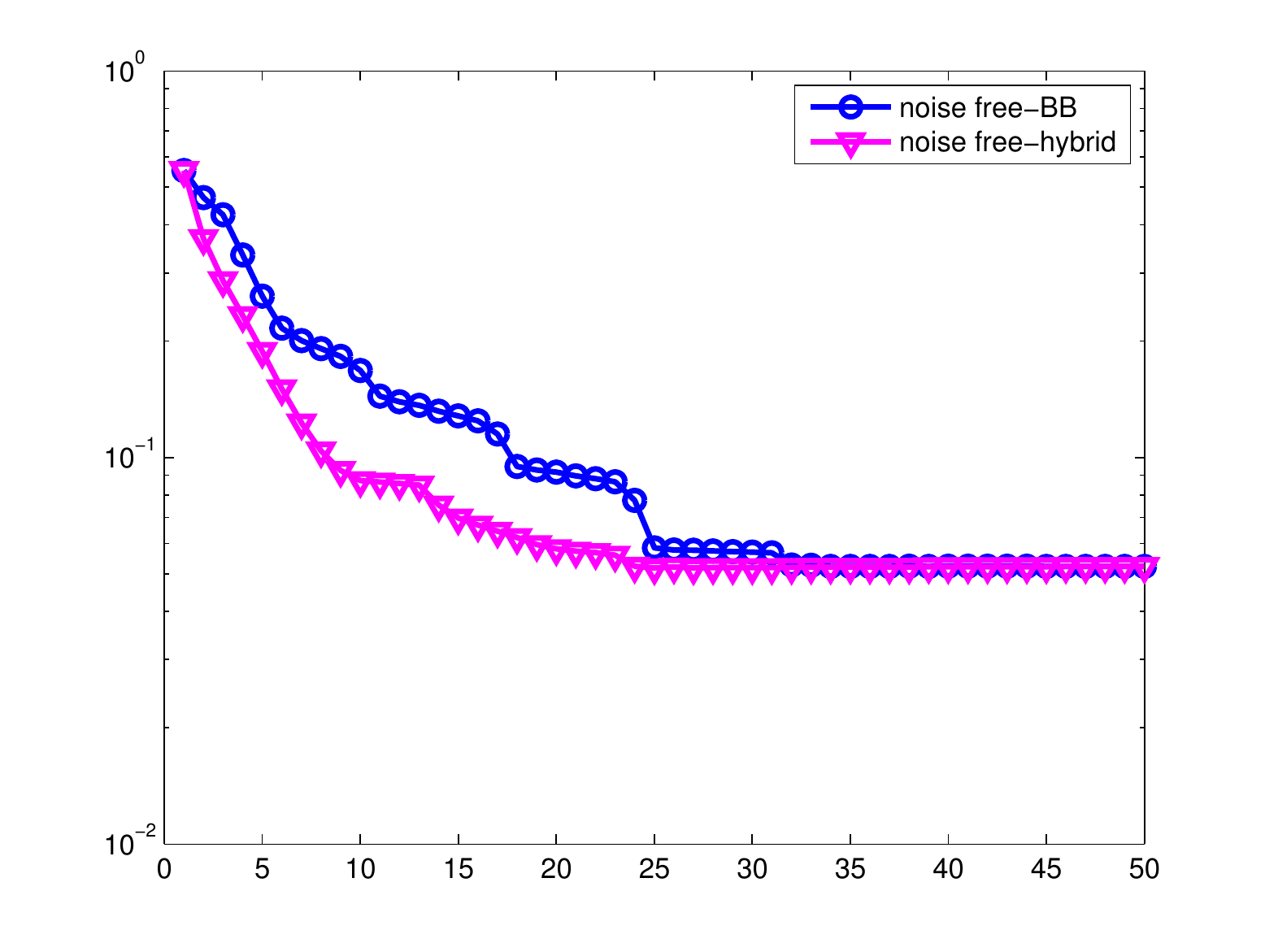}
\end{minipage}
\\
\begin{minipage}{0.24\linewidth}
\includegraphics[width=\textwidth]{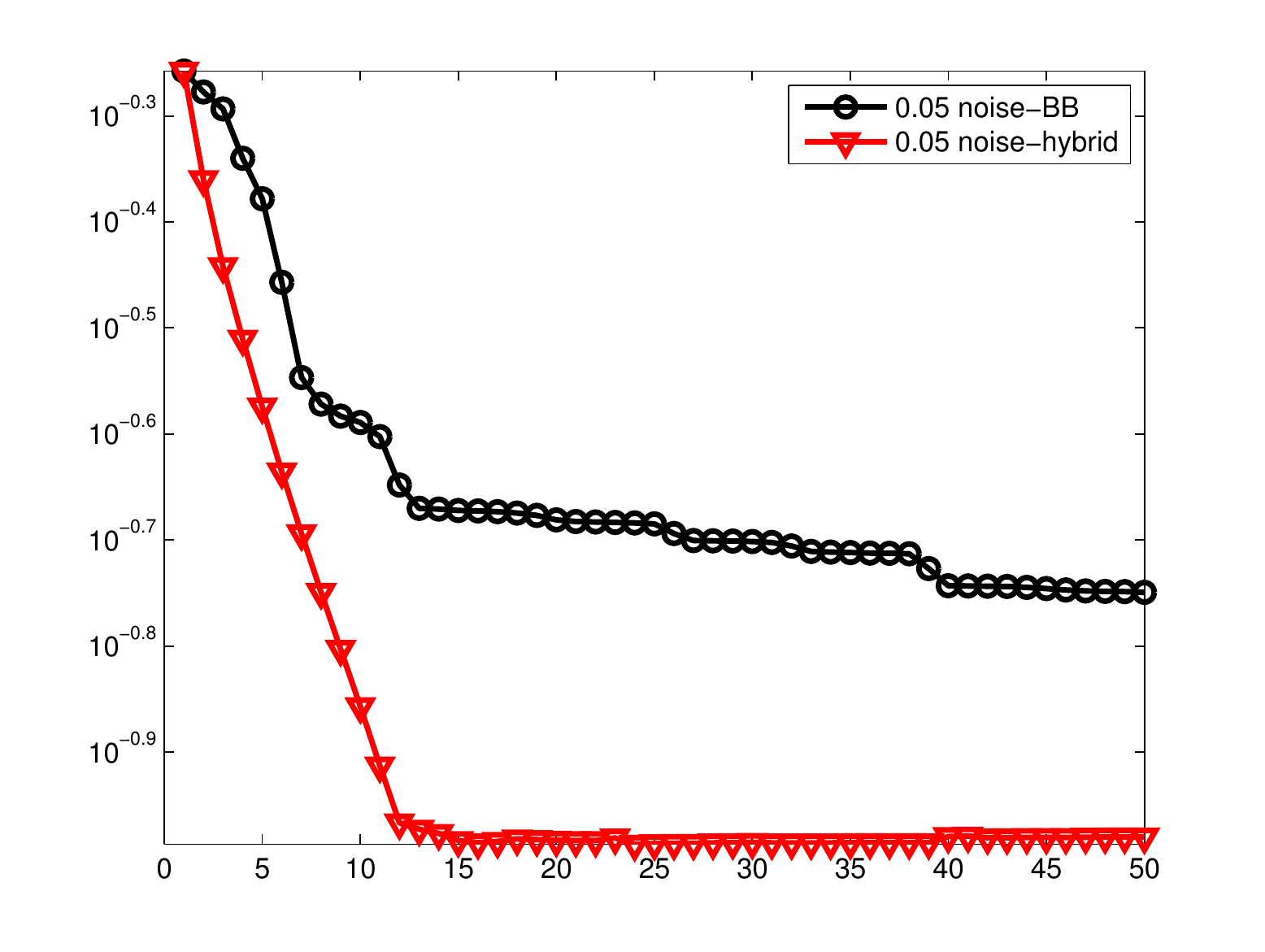}
\end{minipage}
\begin{minipage}{0.24\linewidth}
  \includegraphics[width=\textwidth]{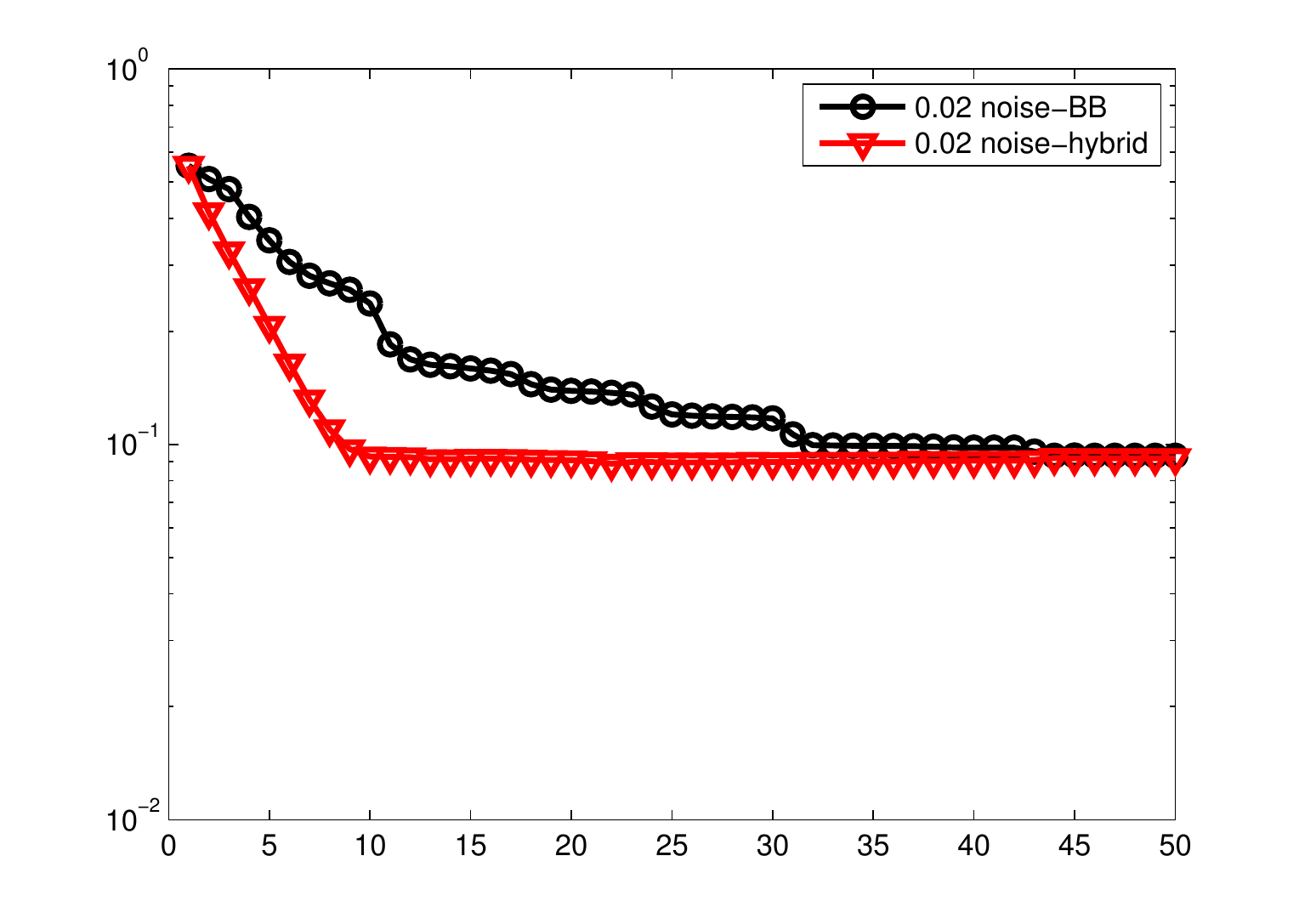}
\end{minipage}
\begin{minipage}{0.24\linewidth}
  \includegraphics[width=\textwidth]{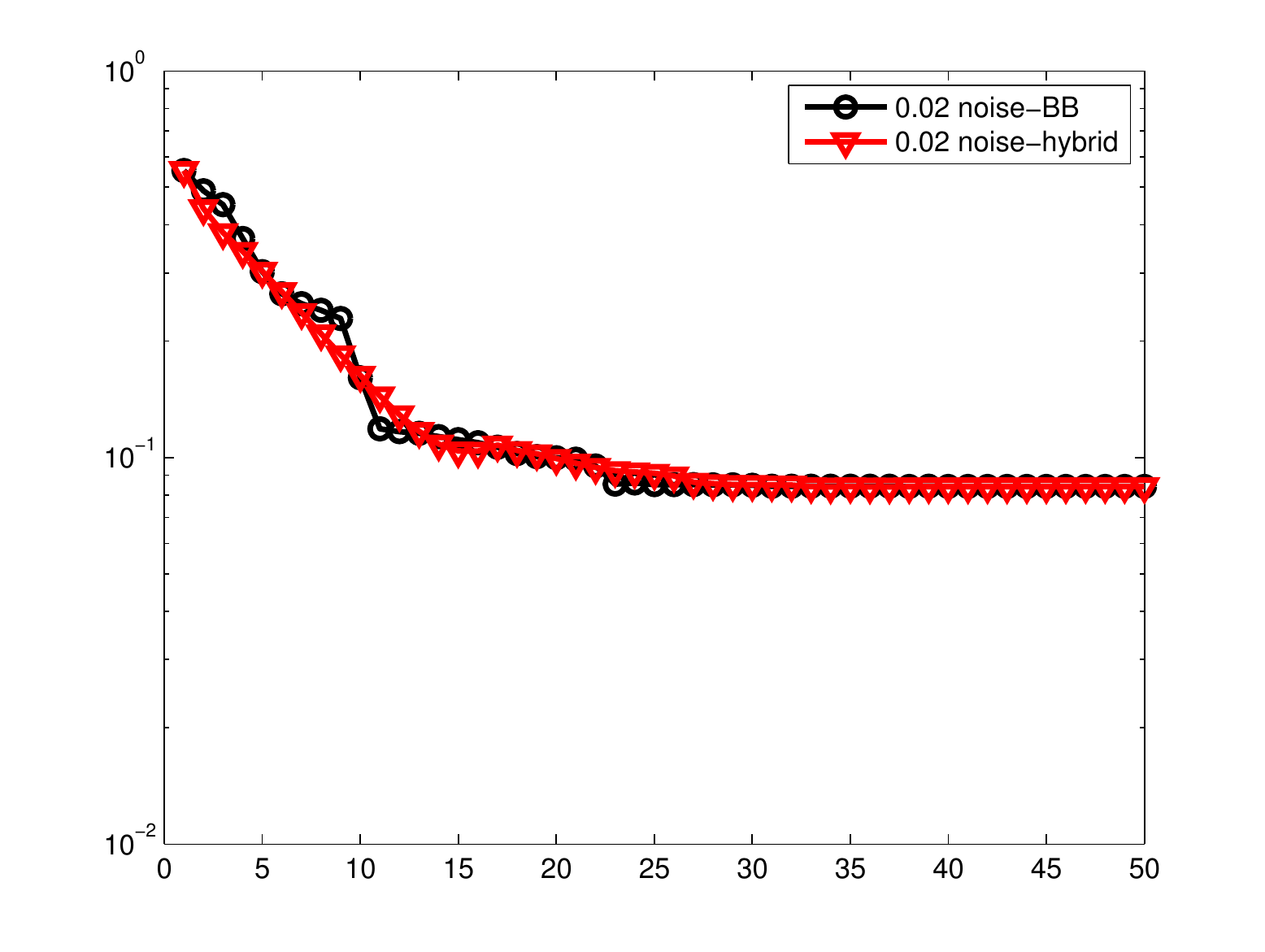}
\end{minipage}
\begin{minipage}{0.24\linewidth}
  \includegraphics[width=\textwidth]{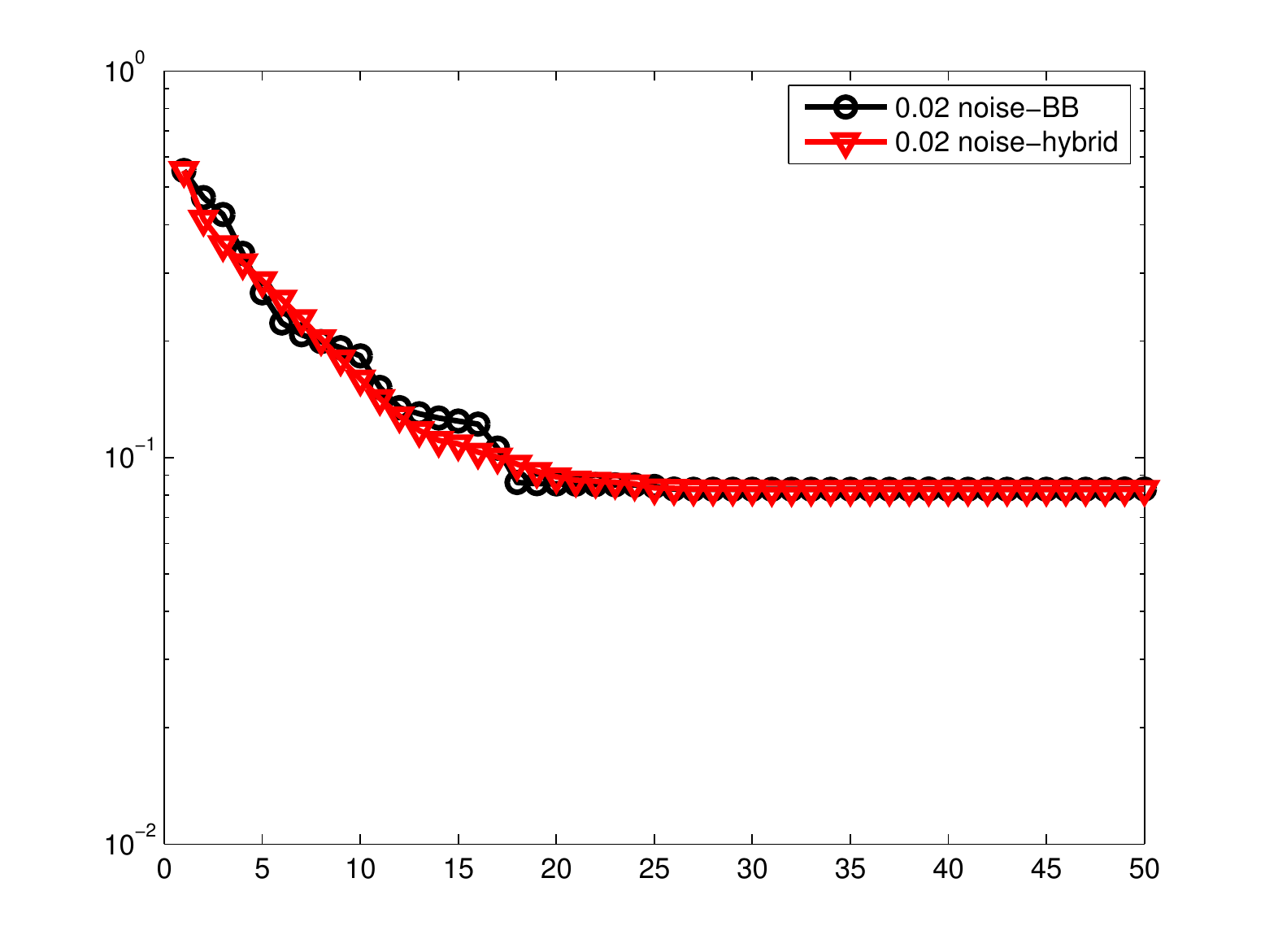}
\end{minipage}
\\
\begin{minipage}{0.24\linewidth}
\includegraphics[width=\textwidth]{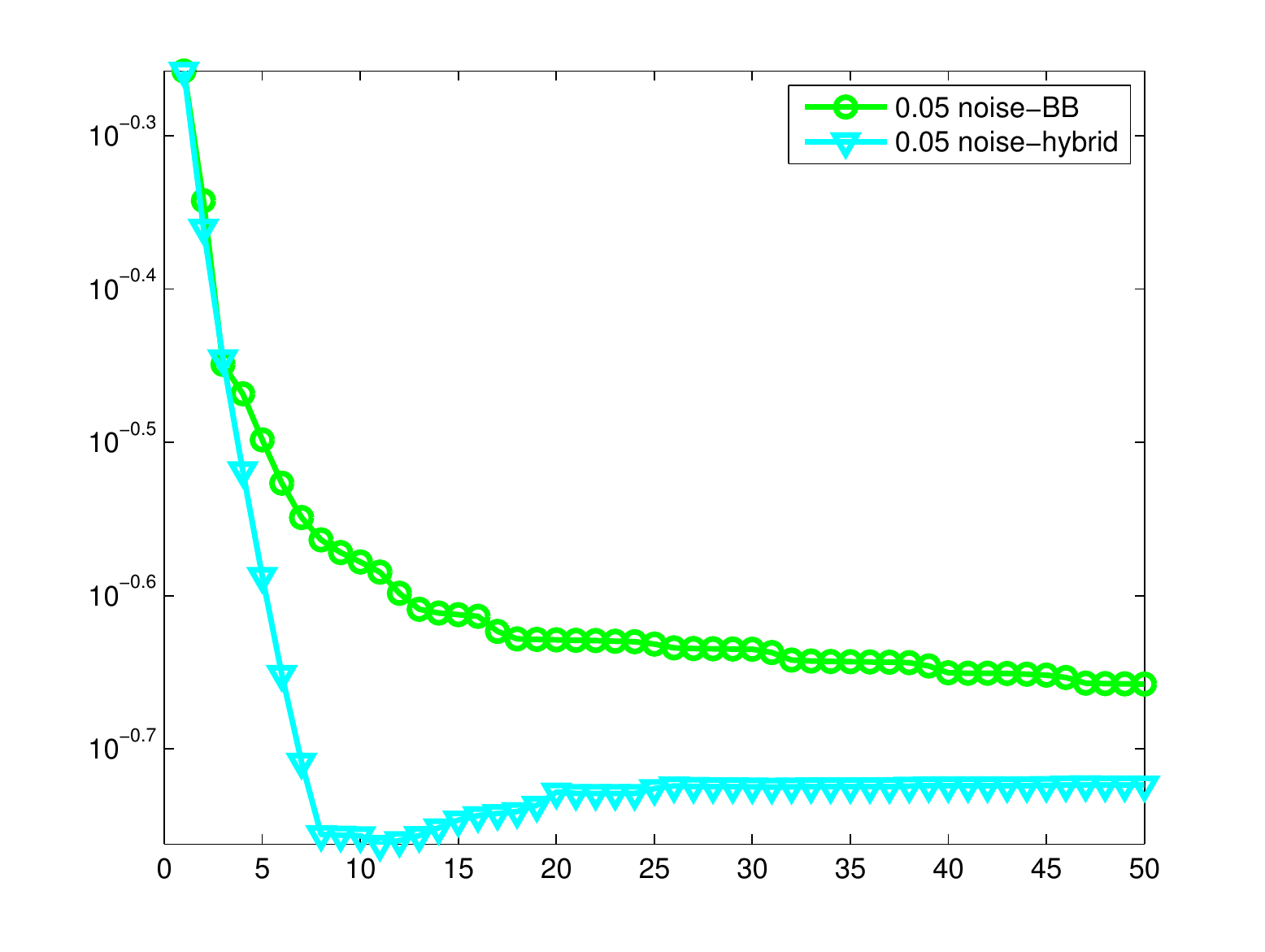}
\end{minipage}
\begin{minipage}{0.24\linewidth}
  \includegraphics[width=\textwidth]{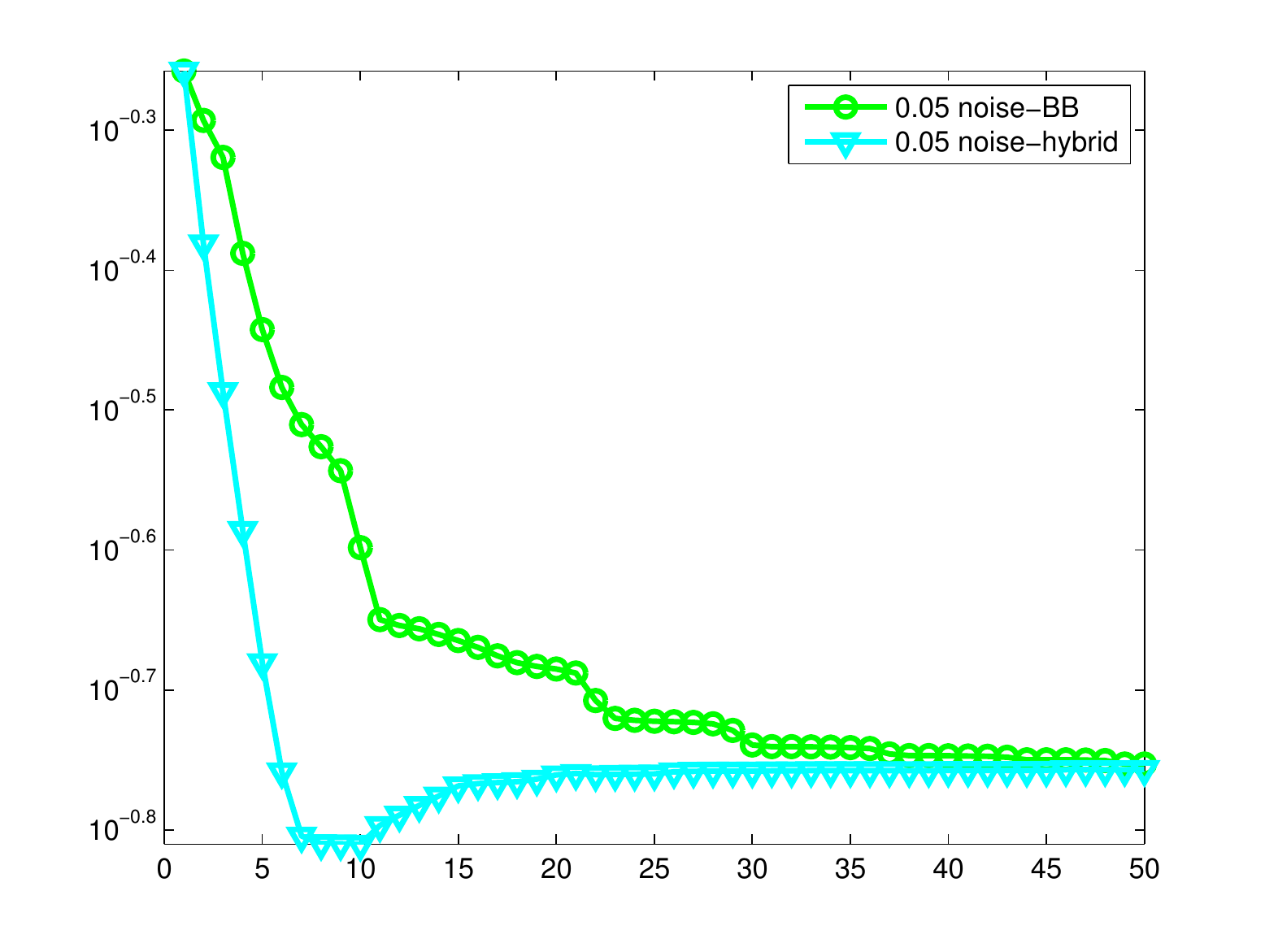}
\end{minipage}
\begin{minipage}{0.24\linewidth}
  \includegraphics[width=\textwidth]{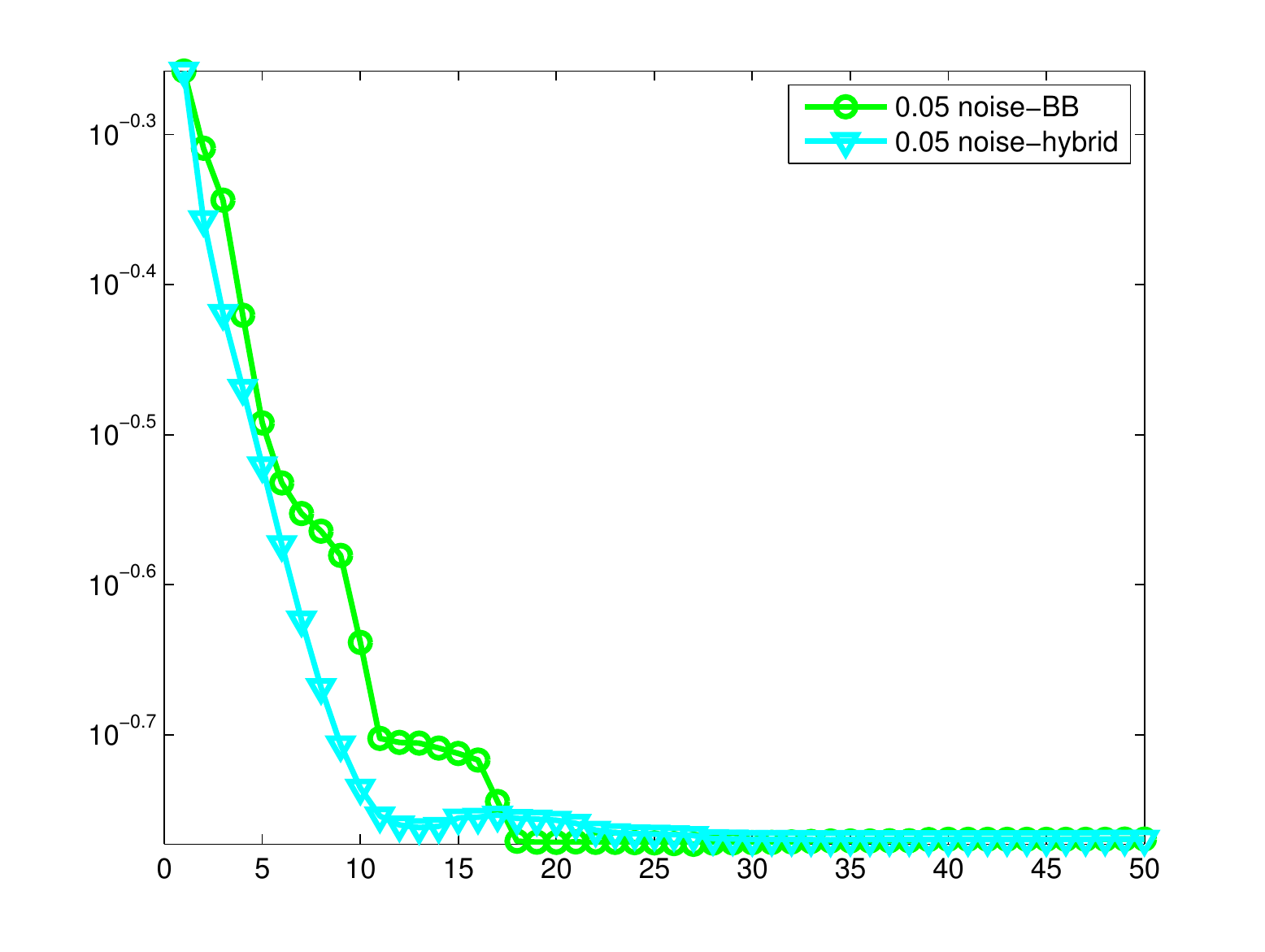}
\end{minipage}
\begin{minipage}{0.24\linewidth}
  \includegraphics[width=\textwidth]{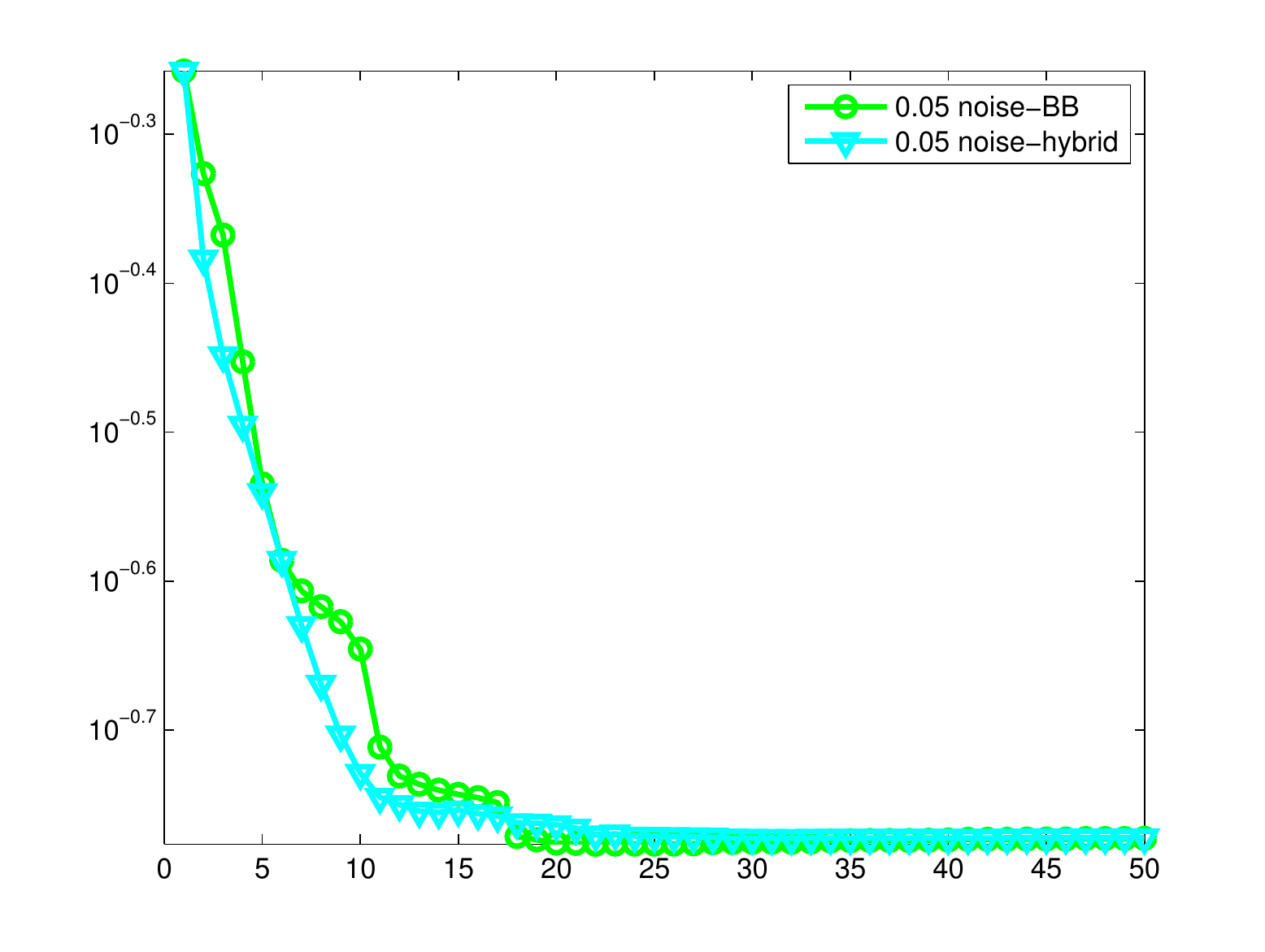}
\end{minipage}
 \caption{\label{fig:template 1}Comparison of the hybrid method and the nonlinear optimization method by relative error $\epsilon_f$ of reconstructed $\mu_{a,xf}$ for first template. First, second and third row: noise-free, 2$\%$ noise, and 5$\%$ noise data. First, second, third, and fourth column: one-measurement, two-measurement, three-measurement and four-measurement.}
\end{figure}

\begin{figure}[htpb]
    \centering
\begin{minipage}{0.24\linewidth}
\includegraphics[width=\textwidth]{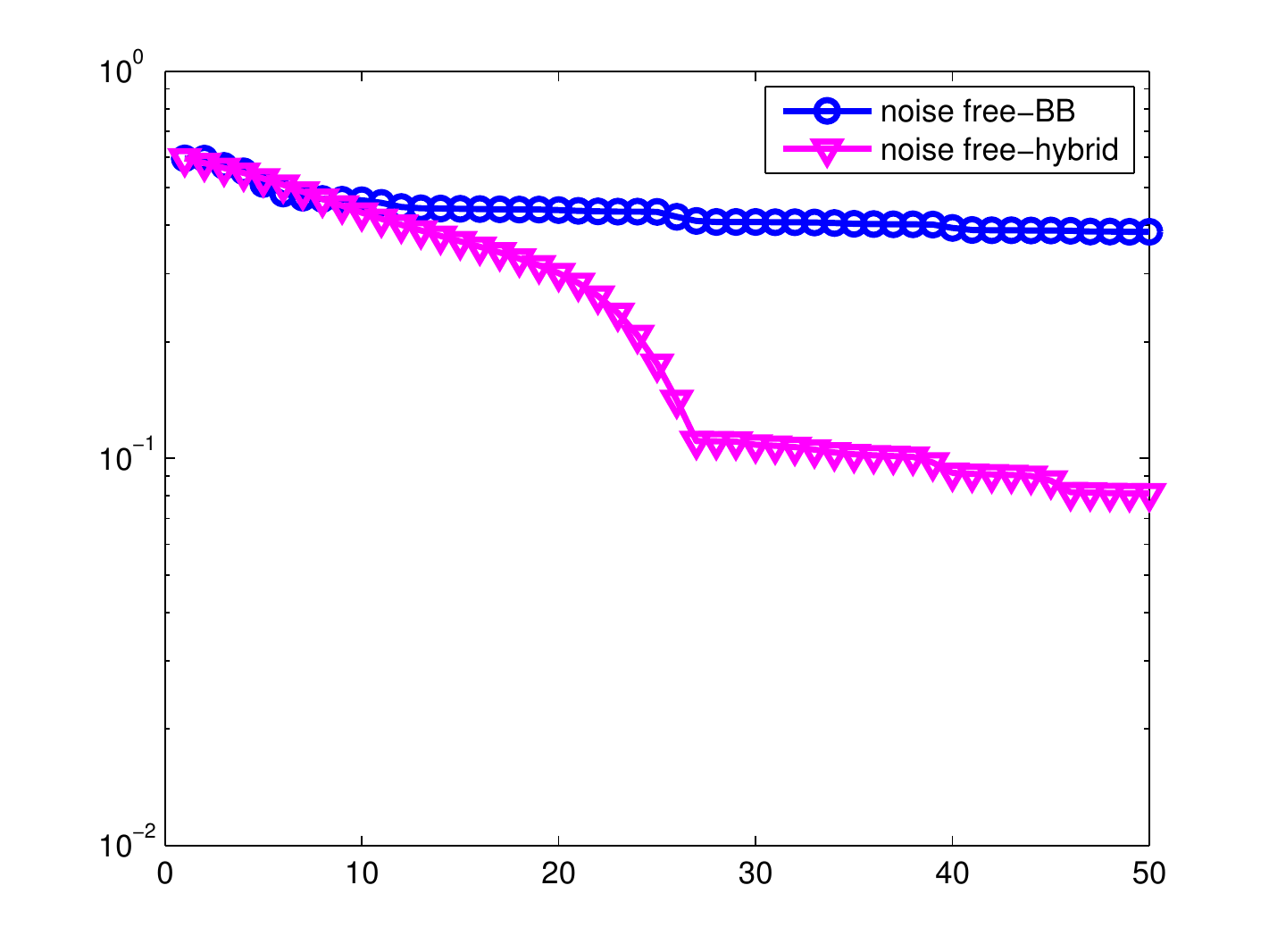}
\end{minipage}
\begin{minipage}{0.24\linewidth}
  \includegraphics[width=\textwidth]{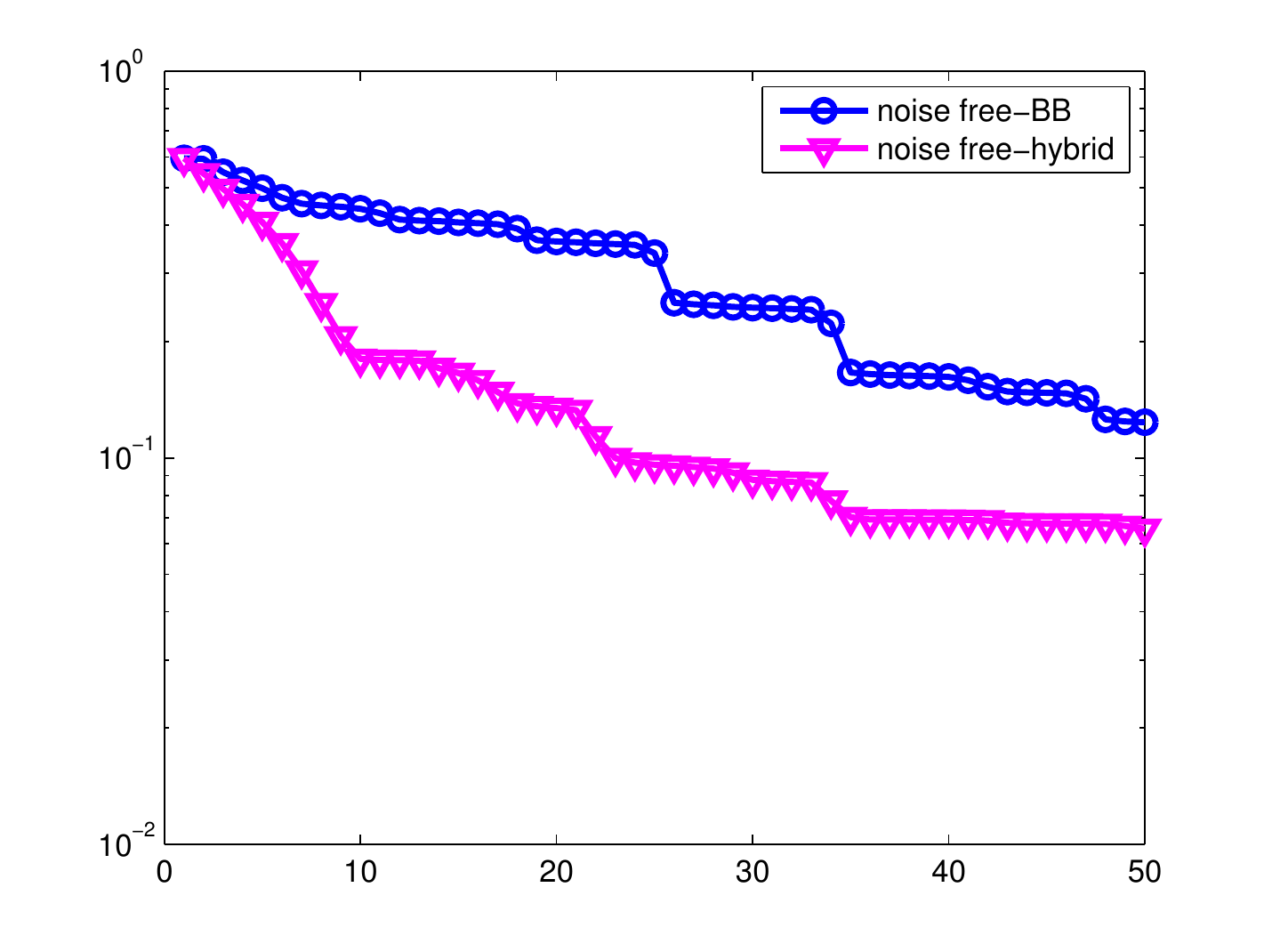}
\end{minipage}
\begin{minipage}{0.24\linewidth}
  \includegraphics[width=\textwidth]{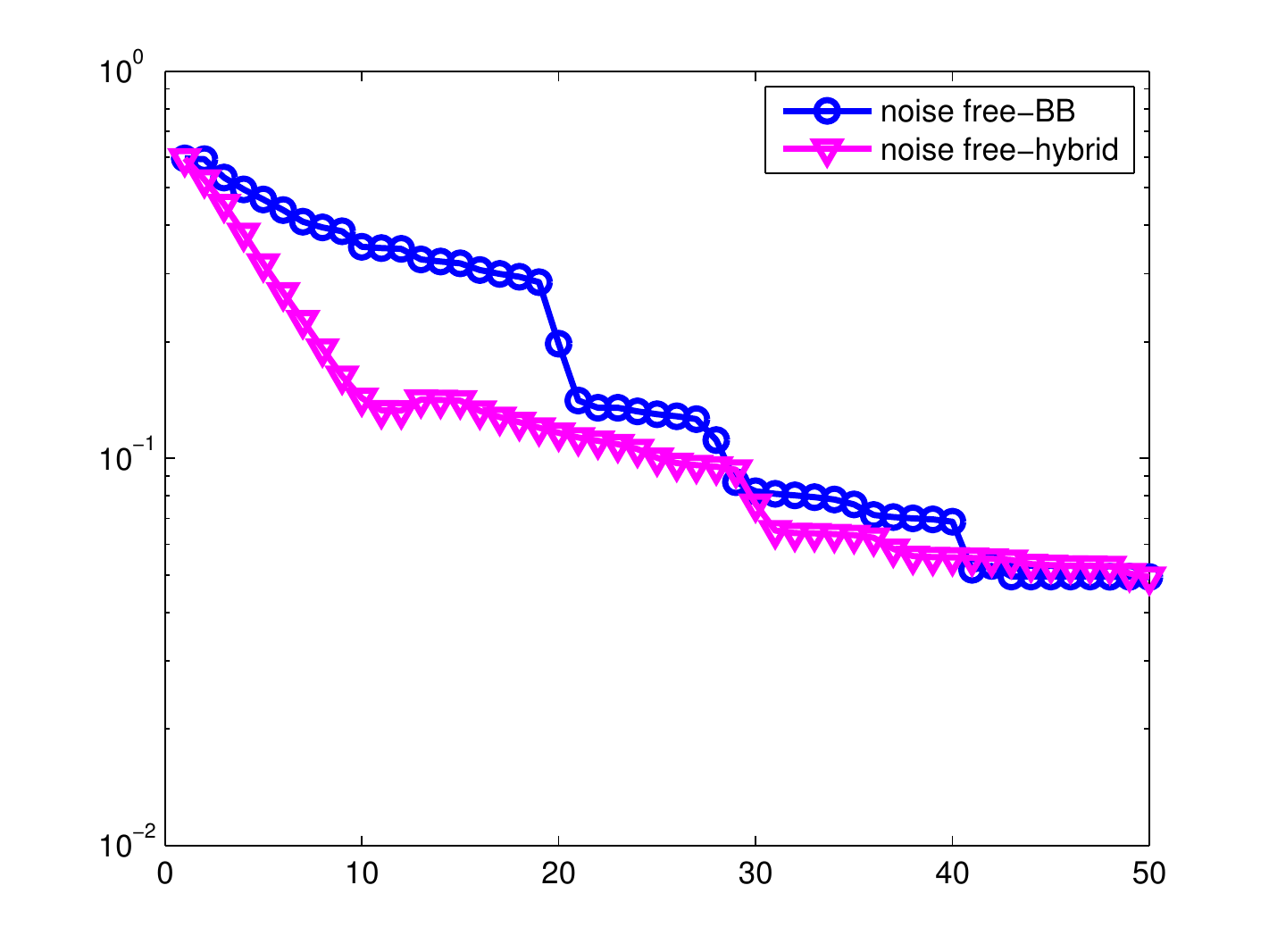}
\end{minipage}
\begin{minipage}{0.24\linewidth}
  \includegraphics[width=\textwidth]{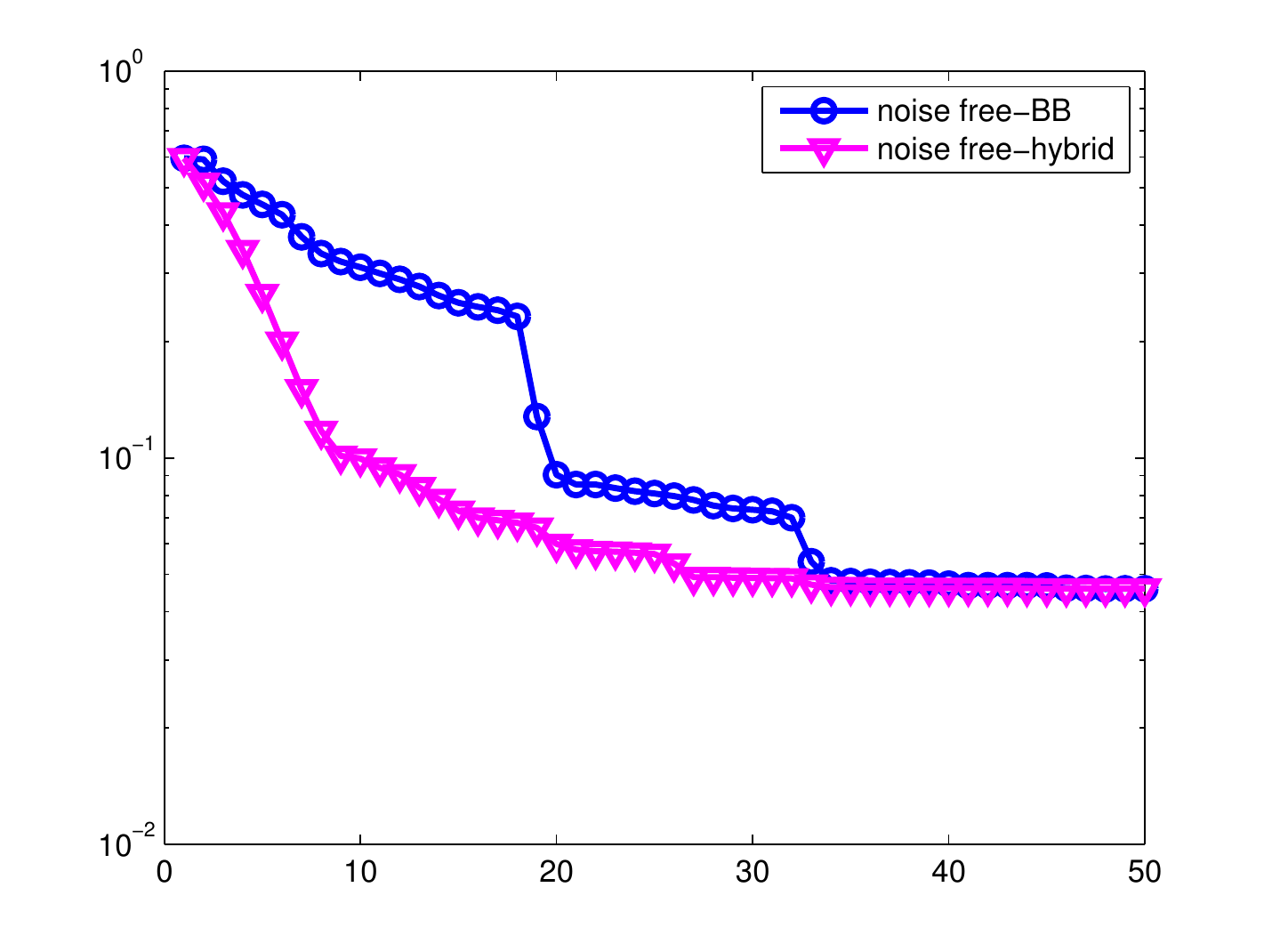}
\end{minipage}
\\
\begin{minipage}{0.24\linewidth}
\includegraphics[width=\textwidth]{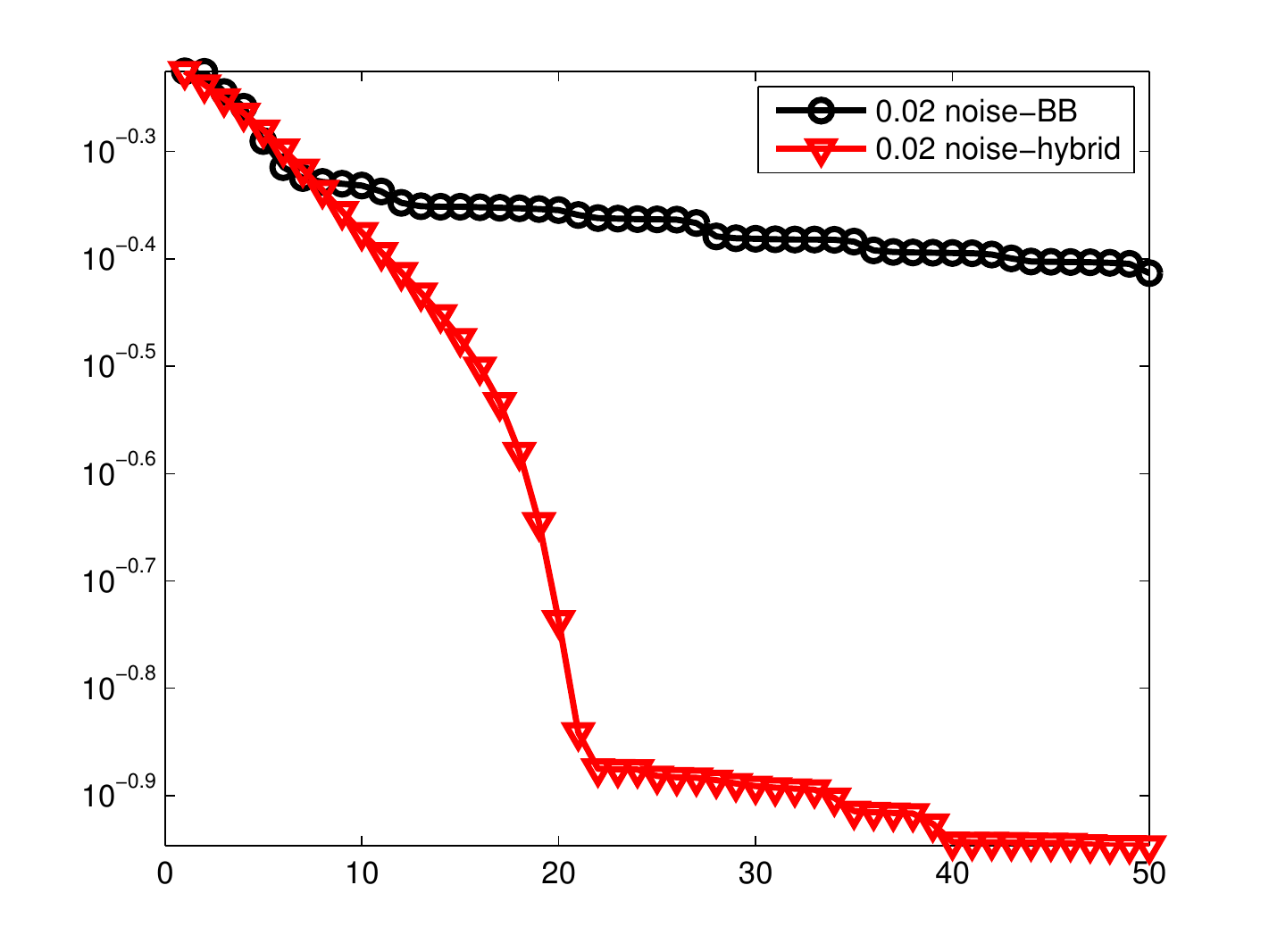}
\end{minipage}
\begin{minipage}{0.24\linewidth}
  \includegraphics[width=\textwidth]{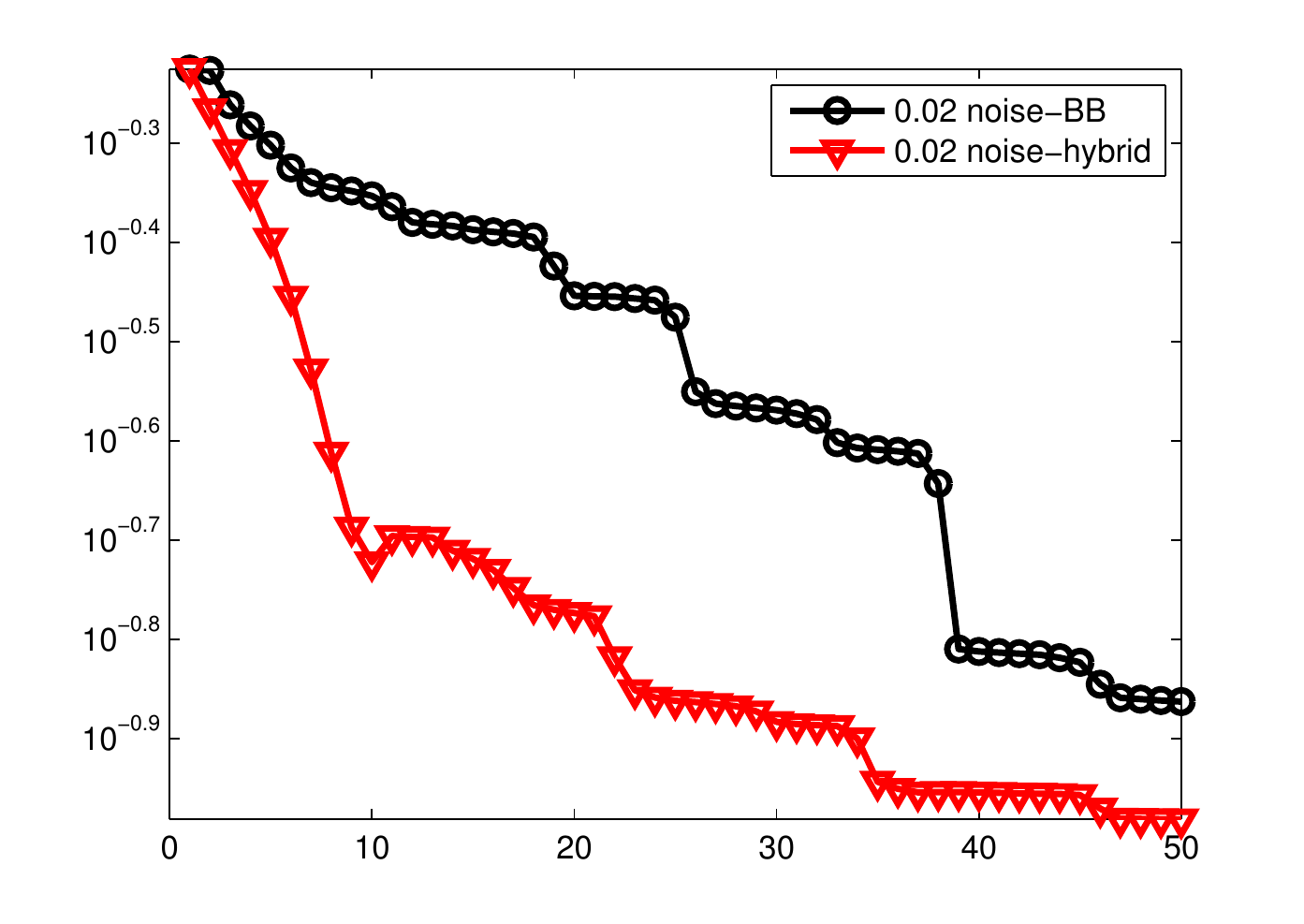}
\end{minipage}
\begin{minipage}{0.24\linewidth}
  \includegraphics[width=\textwidth]{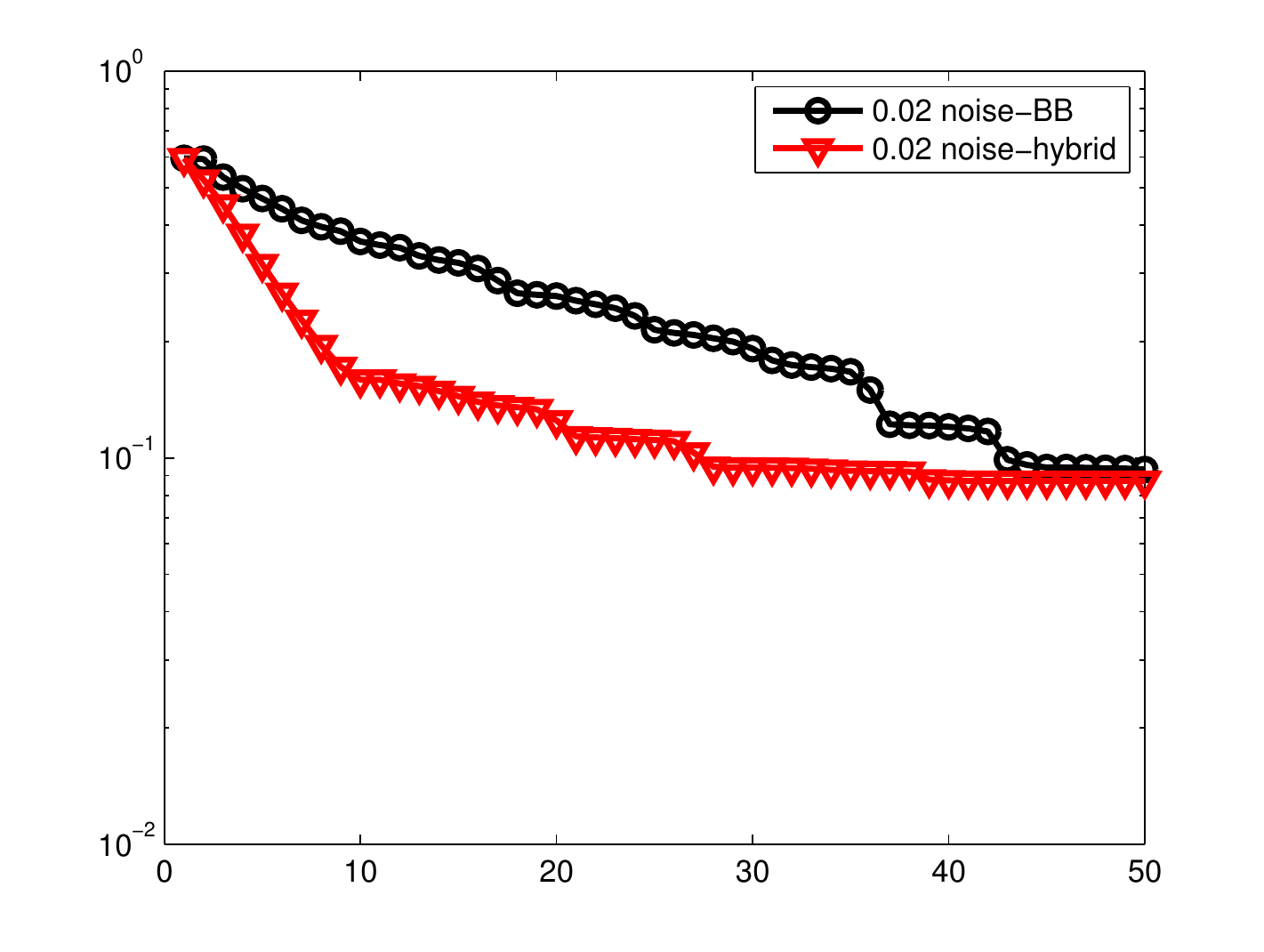}
\end{minipage}
\begin{minipage}{0.24\linewidth}
  \includegraphics[width=\textwidth]{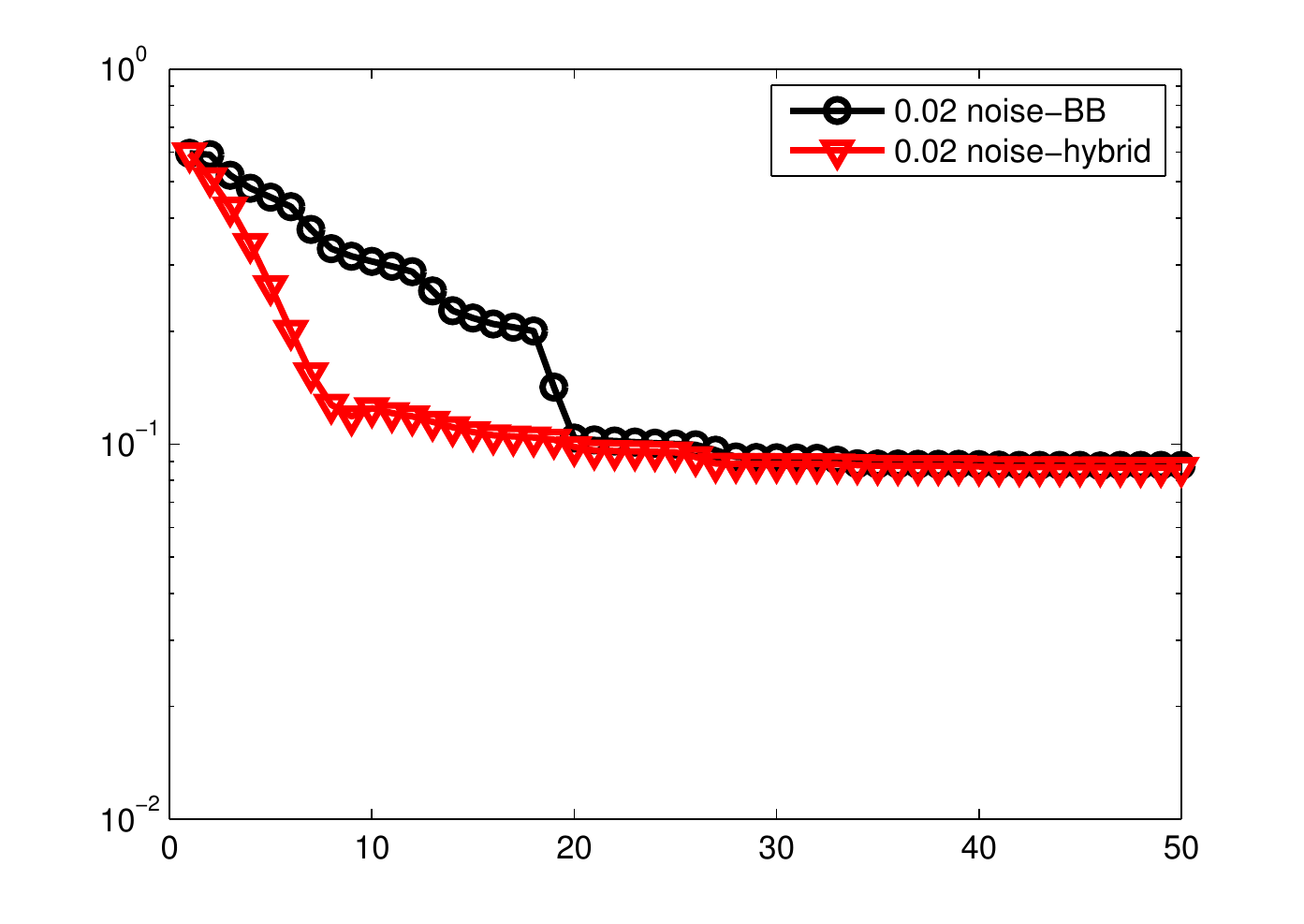}
\end{minipage}
\\
\begin{minipage}{0.24\linewidth}
\includegraphics[width=\textwidth]{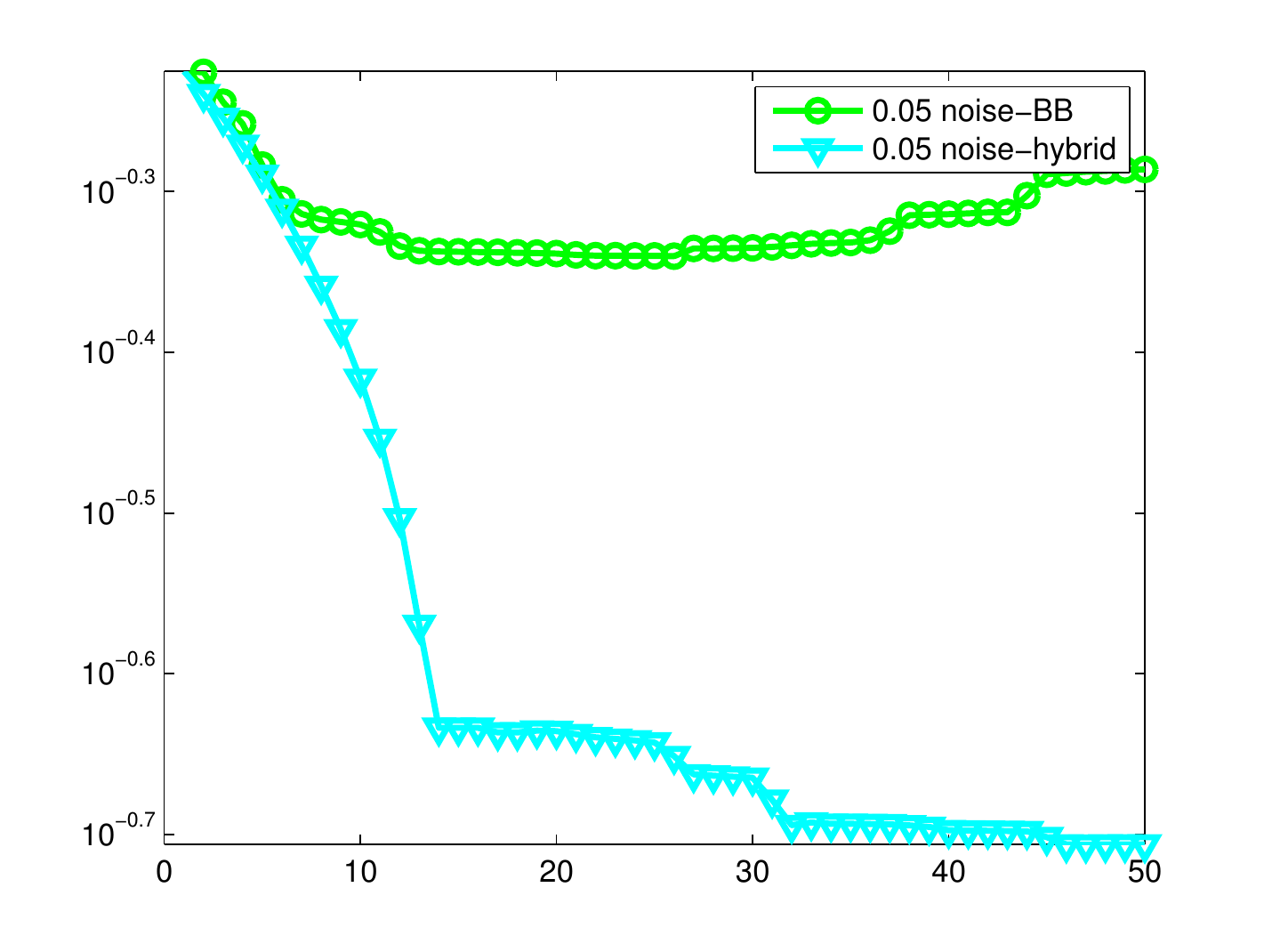}
\end{minipage}
\begin{minipage}{0.24\linewidth}
  \includegraphics[width=\textwidth]{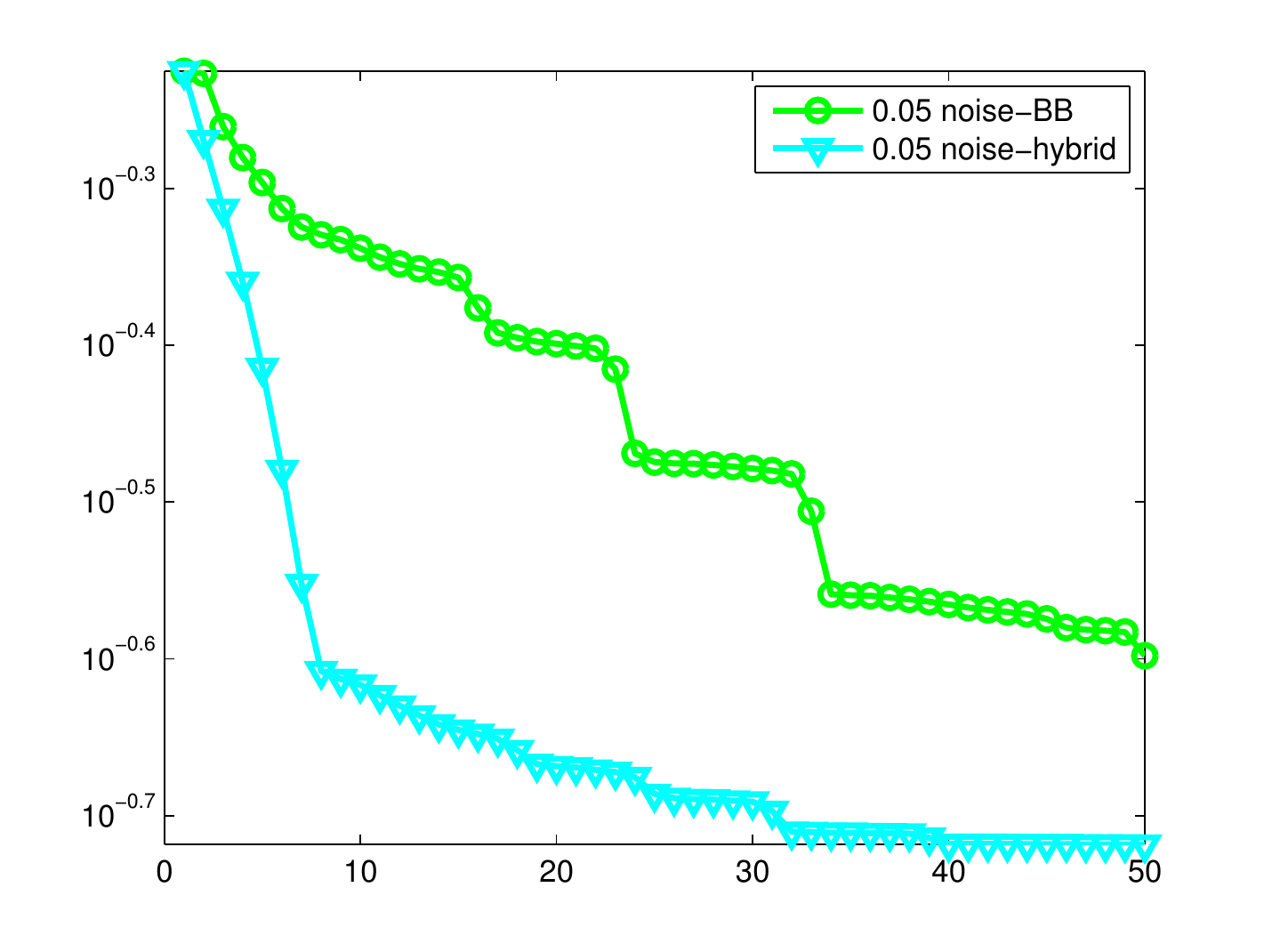}
\end{minipage}
\begin{minipage}{0.24\linewidth}
  \includegraphics[width=\textwidth]{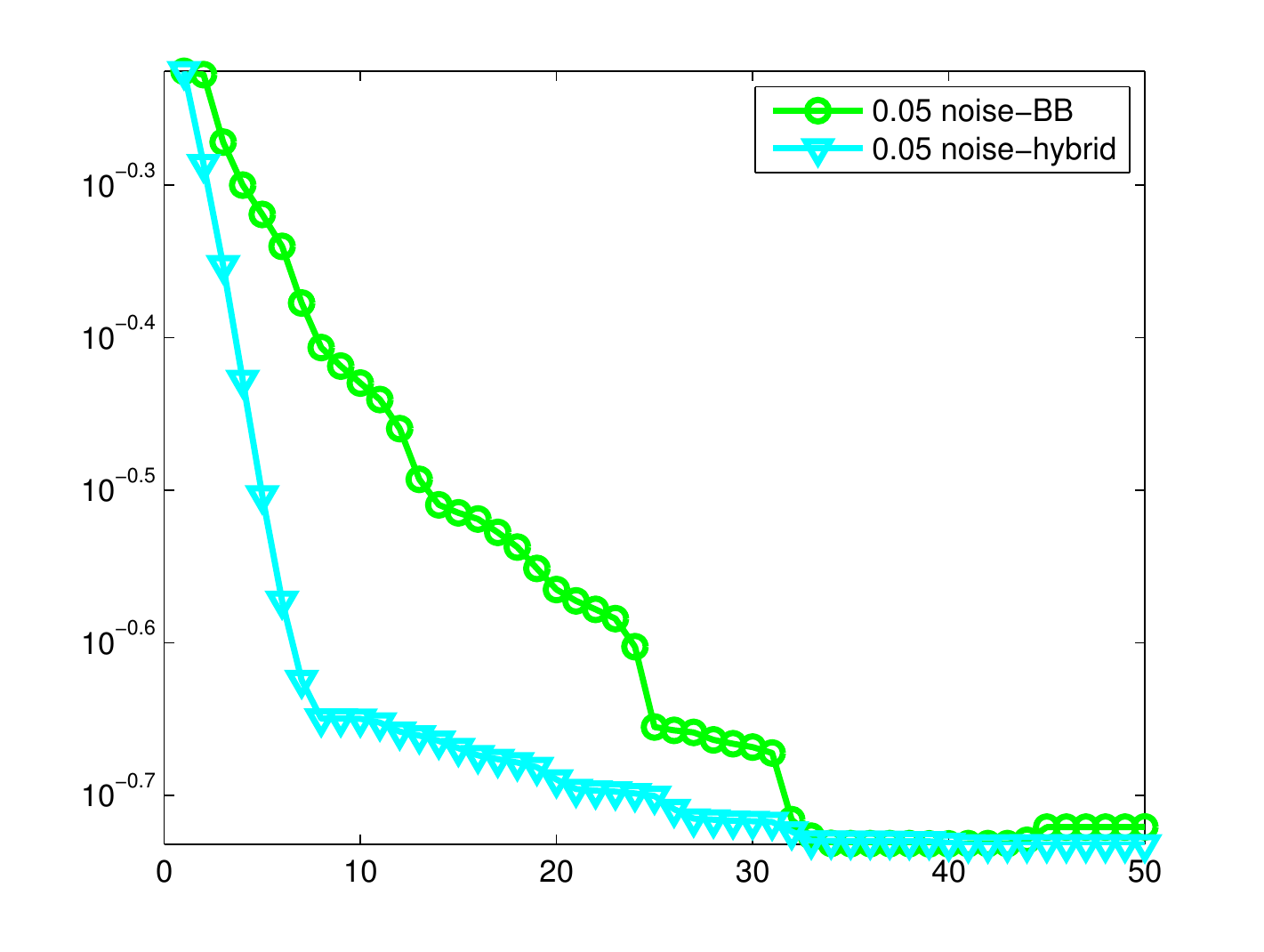}
\end{minipage}
\begin{minipage}{0.24\linewidth}
  \includegraphics[width=\textwidth]{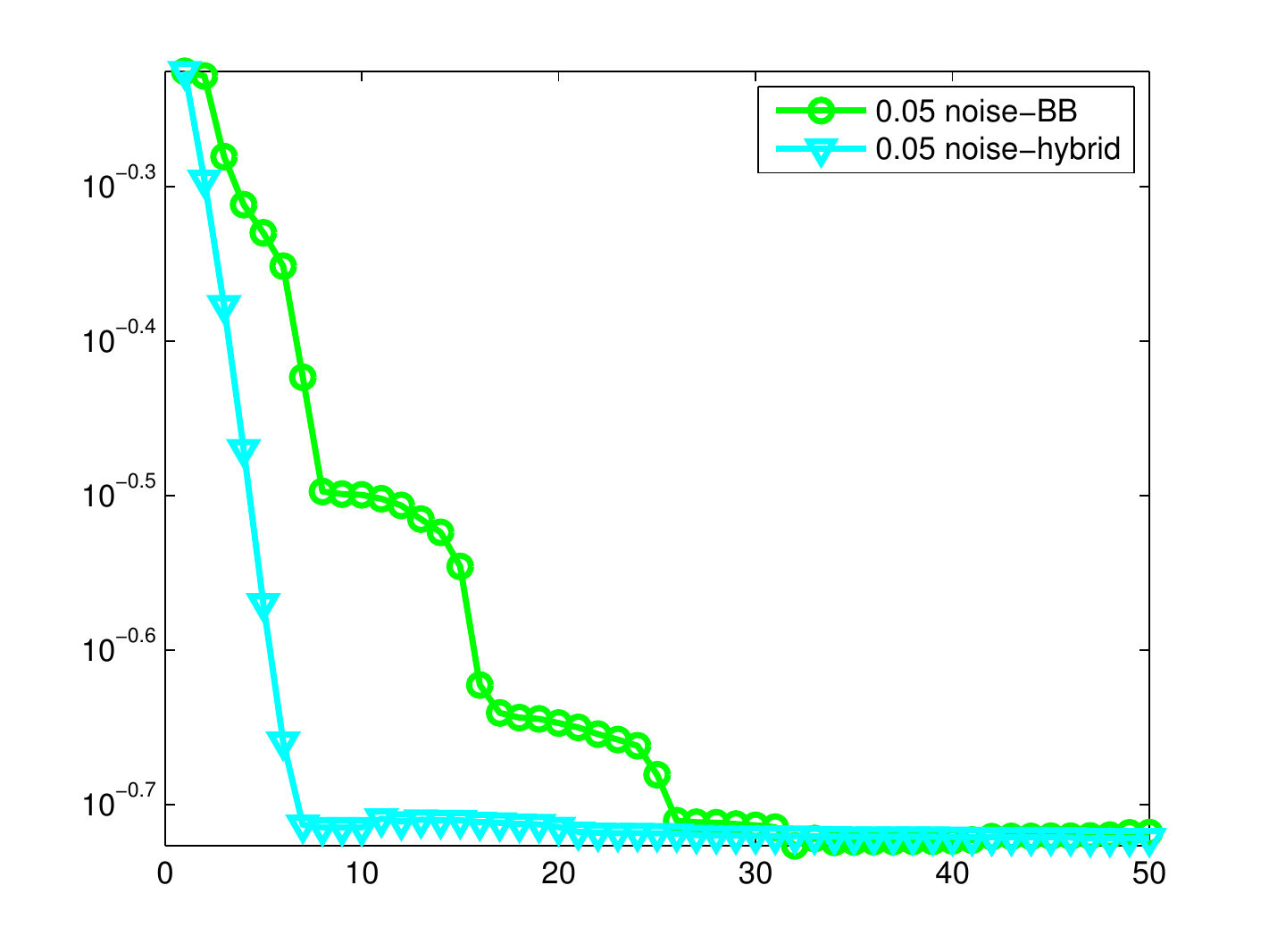}
\end{minipage}
 \caption{\label{fig:template 2}Comparison of the hybrid method and the nonlinear optimization method by relative error $\epsilon_f$ of reconstructed $\mu_{a,xf}$ for second template. First, second and third row: noise-free, 2$\%$ noise, and 5$\%$ noise data. First, second, third, and fourth column: one-measurement, two-measurement, three-measurement and four-measurement.}
\end{figure}

\begin{table}[htpb]
  \centering
\begin{center}
\begin{tabular}{c|c|c|c|c|c|c|c}
\hline
Noise level  ($\%$) & &\multicolumn{2}{c|}{0}     & \multicolumn{2}{c|}{2}  & \multicolumn{2}{c}{5}\\
\hline
Method & Meas. & Hybrid  & Opt. & Hybrid  & Opt.  & Hybrid  & Opt. \\
\hline
\multirow{4}*{\makecell*[tc]{First\\ template}} & 1& 7.85e-2  & 1.50e-1  & 1.05e-1  & 1.78e-1 & 1.89e-1 & 2.20e-1\\
 ~& 2 & 6.37e-2 & 6.92e-2   & 9.19e-2 & 9.32e-2 & 1.75e-1 & 1.76e-1\\
 ~& 3 & 5.23e-2 & 5.23e-2  & 8.41e-2  & 8.42e-2  & 1.70e-1 & 1.70e-1\\
 ~& 4 & 5.22e-2 & 5.23e-2 & 8.28e-2  & 8.28e-2 & 1.69e-1 & 1.69e-1 \\
\hline
 \multirow{4}*{\makecell*[tc]{Second\\ template}}& 1 & 8.12e-2 & 3.85e-1 & 1.13e-1 & 3.86e-1 & 1.97e-1 & 5.17e-1\\
~ & 2 & 6.56e-2 & 1.24e-1 & 1.04e-1 & 1.37e-1  & 1.91e-1& 2.52e-1\\
~ & 3 & 4.95e-2  & 4.94e-2& 8.72e-2  & 9.32e-2  & 1.85e-1  & 1.90e-1 \\
~ &4 & 4.59e-2  & 4.58e-2 & 8.64e-2  & 8.77e-2  & 1.90e-1  & 1.91e-1 \\
\hline
\end{tabular}
\end{center}
  \caption{Relative error $\epsilon_f$ of reconstructed $\mu_{a,xf}$ after 50 steps}
  \label{tab:1}
\end{table}

\section{Conclusion}
\label{sec:5}
In this paper, we propose a hybrid method to reconstruct the fluorescence absorption coefficient combining SIM method and the nonlinear optimization method. In SIM, two monotonic sequences are generated and they are expected to approach the exact coefficient from two sides. In numerical simulations, SIM performs well with lower accuracy. To stabilize the algorithm, nonlinear optimization as a state-of-art method is applied to mitigate the instability and achieve higher accuracy. In nonlinear optimization method, we take log-type function as our error function, and apply adjoint method and BB stepsize to obtain the gradient of error function and stepsize.

We use two templates to test our algorithms respectively on noise-free, 2$\%$ noise and 5$\%$ noise data. Compared to nonlinear optimization method, we find that in fewer measurements, hybrid method is more advantageous. In fewer measurements, since the error function is not convex, optimization method easily fall into a local minimum, even when there is no noise. However, due to an explicit $\mu_{a,xf}$ in \eqref{eq:4}, SIM is inspired by fixed-point iteration, so it is more inclined to satisfy \eqref{eq:4} and it is more likely to avoid the local minimum. In one-measurement case, hybrid method has higher accuracy. In three or four-measurement case, both methods can eventually achieve the same accuracy. Despite this, in most cases, hybrid method converges more rapidly and achieve approximately linear convergence in the first SIM steps. Therefore, compared to applying SIM or optimization method for quantitative FPAT alone, hybird method outperforms each of them with faster convergence and higher accuracy.

In the future, we intend to search for better error function so that the fluorescence absorption coefficient and quantum efficiency can be more accurately reconstructed with few measurements. Meanwhile, the theory on the convergence of hybrid method based on multi-measurement is also worth studying.

\section{Acknowledgments}
\label{sec:6}

 This work was supported by NSF grants of China (61421062, 11471024).

\appendix

\section{The derivative of gradient of error function}
In this section, we regard $ \mathcal{F}$ as a functional with respective to $\mu_{a,xf}$ and $\eta$. For convenience, we omit the subscript on measurement 's', and the error function is
\begin{equation}
  \label{eq:33}
  \mathcal{F}(\mu_{a,xf},\eta)=\frac{1}{2}\norm{\log (\bm{H}(\mu_{a,xf})-\log(h^*))}_2^2+\mathcal{R}(\mu_{a,xf},\eta),
\end{equation}

\begin{theorem}
  \label{th:4}
  Let $\phi_m^*$ and $\phi_x^*$ are the solutions of
  \begin{equation}
    \label{eq:34}
    \left\{
      \begin{aligned}
        &(-\theta\cdot\nabla+\mu_{a,m}+\mu_{s,m}-\mu_{s,m}\bm{K})\phi_m^*=\bm{A}^*\left(\frac{1}{h}(\log(h)-\log(h^*))\mu_{a,m}\right),\\
        &\phi_m^*|_{\Gamma_+}=0,
      \end{aligned}
      \right.
  \end{equation}
and
\begin{equation}
  \label{eq:35}
  \left\{
    \begin{aligned}
      &(-\theta\cdot \nabla+\mu_{a,x}+\mu_{a,xf}+\mu_{s,x}-\mu_{s,x}\bm{K})\phi_x^*\\
      &\phantom{aaa}=\bm{A}^*\left(\eta\mu_{a,xf}(\bm{A}\phi^*_m)+\frac{1}{h}(\log h-\log h^*)(\mu_{a,xi}+(1-\eta)\mu_{a,xf})\right),\\
    &\phi_x^*|_{\Gamma_+}=0.
    \end{aligned}
    \right.
\end{equation}
Then ignoring the regularization, assume $h_f$ and $h_\eta$ are two feasible direction of $\mu_{a,xf}$ and $\eta$ respectively, we deduce the gradient of $\mathcal{F}$ is
\begin{equation}
\begin{aligned}
  \label{eq:36}
  &\mathcal{F}'(\mu_{a,xf},\eta)(h_f,h_\eta)\\
 =&\left<\frac{1}{h}(\log h-\log h^*)(1-\eta)(\bm{A}\phi_x)+(\bm{A}\phi_m^*)(\tilde{\bm{A}}\phi_x)\eta-\bm{A}(\phi_x^*\phi_x),h_f\right>\\
    &\phantom{aaaaa}+\left<-\frac{1}{h}(\log h-\log h^*)-\mu_{a,xf}(\bm{A}\phi_x)+(\bm{A}\phi_m^*)(\tilde{\bm{A}}\phi_x)\mu_{a,xf},h_\eta\right>
  \end{aligned}
  \end{equation}

\end{theorem}
\begin{proof}
  According to the chain rule,
  \begin{equation}
    \label{eq:37}
    \begin{aligned}
      &\mathcal{F}'(\mu_{a,xf},\eta)(h_f,h_\eta)\\
      =&\left<\log h -\log h^*,\frac{1}{h}\left[(\mu_{a,xi}+(1-\eta)\mu_{a,xf})(\bm{A}\phi_x')+((1-\eta)h_f-h_\eta \mu_{a,xf})(\bm{A}\phi_x)+\mu_{a,m}(\bm{A}\phi_m')\right]\right>\\
      =&\left<\frac{1}{h}(\log h -\log h^*)(\mu_{a,xi}+(1-\eta)\mu_{a,xf}),\bm{A}\phi_x'\right>\\
      &+\left<\frac{1}{h}(\log h -\log h^*)\mu_{a,m},\bm{A}\phi_m'\right>\\
      &+\left<\frac{1}{h}(\log h -\log h^*)(1-\eta)(\bm{A}\phi_x),h_f\right>\\
      &+\left<-\frac{1}{h}(\log h -\log h^*)\mu_{a,xf}(\bm{A}\phi_x),h_\eta\right>.
    \end{aligned}
  \end{equation}
Firstly, we can simplify the first and second term, that is
\begin{equation}
  \label{eq:38}
  \begin{aligned}
    &\left<\frac{1}{h}(\log h -\log h^*)\mu_{a,m},\bm{A}\phi_m'\right>\\
    =&\left<\bm{A}^*\left[\frac{1}{h}(\log h -\log h^*)\mu_{a,m}\right],\phi_m'\right>\\
    =&\left<\phi_m^*,\eta\mu_{a,xf}(\tilde{\bm{A}}\phi'_x)+\eta h_f(\tilde{\bm{A}}\phi_x)+h_\eta\mu_{a,xf}(\tilde{\bm{A}}\phi_x)\right>\\
    =&\left<\tilde{\bm{A}}^*(\eta\mu_{a,xf}(\tilde{\bm{A}}\phi_m^*)),\phi_x'\right>\\
    &+\left<\eta(\bm{A}\phi_m^*)(\tilde{\bm{A}}\phi_x),h_f\right>\\
    &+\left<\mu_{a,xf}(\bm{A}\phi_m^*)(\tilde{\bm{A}}\phi_x),h_\eta\right>,
  \end{aligned}
\end{equation}
and
\begin{equation}
  \label{eq:39}
  \begin{aligned}
    &\left<\frac{1}{h}(\log h -\log h^*)(\mu_{a,xi}+(1-\eta)\mu_{a,xf}),\bm{A}\phi_x'\right>\\
    =&\left<\bm{A}^*\left[\frac{1}{h}(\log h -\log h^*)(\mu_{a,xi}+(1-\eta)\mu_{a,xf})\right],\phi_x'\right>.
  \end{aligned}
\end{equation}
Then putting the first term on the right side of equation \eqref{eq:38} and equation \eqref{eq:39}, it is
\begin{equation}
  \label{eq:40}
  \begin{aligned}
    &\left<\bm{A}^*(\eta\mu_{a,xf}(\tilde{\bm{A}}\phi_m^*)),\phi_x'\right>+\left<\frac{1}{h}(\log h -\log h^*)(\mu_{a,xi}+(1-\eta)\mu_{a,xf}),\bm{A}\phi_x'\right>\\
    =&\left<\bm{A}^*\left[\eta\mu_{a,xf}(\tilde{\bm{A}}\phi_m^*)+\frac{1}{h}(\log h -\log h^*)(\mu_{a,xi}+(1-\eta)\mu_{a,xf})\right],\phi_x'\right>\\
   =&\left<\phi_x^*,-h_f\phi_x\right>\\
    =&\left<-\bm{A}(\phi_x^*\phi_x),h_f\right>.
  \end{aligned}
\end{equation}
Therefore,
\begin{equation}
  \label{eq:41}
  \begin{aligned}
    &\mathcal{F}'(\mu_{a,xf},\eta)(h_f,h_\eta)\\
    =&\left<\frac{1}{h}(\log h -\log h^*)(1-\eta)(\bm{A}\phi_x)+\eta(\bm{A}\phi_m^*)(\tilde{\bm{A}}\phi_x)-\bm{A}(\phi_x^*\phi_x),h_f\right>\\
    &+\left<-\frac{1}{h}(\log h -\log h^*)\mu_{a,xf}(\bm{A}\phi_x)+\mu_{a,xf}(\bm{A}\phi_m^*)(\tilde{\bm{A}}\phi_x),h_\eta\right>.
  \end{aligned}
\end{equation}

\end{proof}
\section{Some techniques of inner operation}

In order to facilitate the derivation of the object function's gradient in theorem \ref{th:4}, we list some techniques of inner operation applied in theorem \ref{th:4}. For any function $f_1(x,\theta)\in\mathcal{L}^2(\X),\ f_2(x,\theta)\in\mathcal{L}^2(\X), \ f_3(x)\in\mathcal{L}^2(\Omega)$,
\begin{equation}
  \label{eq:app1}
  \begin{aligned}
    &\left<f_1(x,\theta),(\tilde{\bm{A}}f_2(x,\theta))f_3(x)\right>\\
    =&\frac{1}{S_d}\int_\Omega f_1(x,\theta)\oint_{\mathcal{S}^{d-1}}\left(\oint_{\mathcal{S}^{d-1}}f_2(x,\theta')\rd \theta'\right)f_3(x)\rd \theta \rd x\\
    =&\frac{1}{S_d}\int_\Omega f_2(x,\theta')f_3(x)\left(\oint_{\mathcal{S}^{d-1}}f_1(x,\theta)\rd \theta\right)\rd \theta' \rd x\\
    =&\left<\bm{A}^*[(\tilde{\bm{A}}f_1)f_3],f_2\right>,
  \end{aligned}
\end{equation}

\begin{equation}
  \label{eq:app2}
  \begin{aligned}
    &\left<f_1(x,\theta),(\tilde{\bm{A}}f_2(x,\theta))f_3(x)\right>\\
    =&\frac{1}{S_d}\int_\Omega f_1(x,\theta)\oint_{\mathcal{S}^{d-1}}\left(\oint_{\mathcal{S}^{d-1}}f_2(x,\theta')\rd \theta'\right)f_3(x)\rd \theta \rd x\\
    =&\int_\Omega \left(\int_\Omega f_1(x,\theta)\rd \theta\right)\left(\oint_{\mathcal{S}^{d-1}}f_2(x,\theta')\rd \theta'\right)f_3(x)\rd x\\
    =&\left<(\bm{A}f_1)(\tilde{\bm{A}}f_2),f_3\right>,
  \end{aligned}
\end{equation}

\begin{equation}
  \label{eq:app3}
  \begin{aligned}
    &\left<f_1(x,\theta),f_2(x,\theta)f_3(x)\right>\\
    =&\int_\Omega \oint_{\mathcal{S}^{d-1}} f_1(x,\theta)f_2(x,\theta)f_3(x)\rd \theta \rd x\phantom{aaaaaaaaaaaaaa}\\
    =&\int_\Omega (\oint_{\mathcal{S}^{d-1}} f_1(x,\theta)f_2(x,\theta)\rd \theta )f_3(x)\rd x\\
    =&\left<\bm{A}(f_1f_2),f_3\right>.
  \end{aligned}
\end{equation}

\section*{References}
 % \bibliography{qpat}

\end{document}